\documentclass{article}
\usepackage[utf8]{inputenc}
\usepackage{xcolor}
\usepackage{amsmath, amssymb}
\usepackage{hyperref}
\usepackage[format=plain,
            labelfont={bf,it},
            textfont=it]{caption}
\usepackage{chngcntr}
\usepackage{verbatim}
\counterwithin{figure}{section}

\usepackage{tikz}
\usetikzlibrary{arrows}
\usetikzlibrary{shapes}
\usetikzlibrary{shapes.geometric, arrows}
\usetikzlibrary{positioning}
\usetikzlibrary{calc,shapes.multipart,chains}
\usepackage{tikz-cd}
\usepackage{dutchcal}
\usepackage{subcaption}
\usepackage{multirow}
\usepackage{cleveref}
\usepackage{tkz-graph}
\usepackage{placeins}

\usepackage{xifthen}
\newcommand{\gpgp}[2]{
\mathcal{G}_{\ifthenelse{\isempty{#1}}{\ell}{#1}}(\overline{\mathbb{F}_{\ifthenelse{\isempty{#2}}{p}{#2}}})
}

\definecolor{cacolor}{RGB}{0,206,209}

\definecolor{jocolor}{RGB}{255,140,0}

\definecolor{jacolor}{RGB}{0,0,205}

\definecolor{kracolor}{RGB}{70,130,180}

 \definecolor{sacolor}{RGB}{147,112,219}

 \definecolor{trcolor}{RGB}{199,21,133}

 \usepackage[utf8]{inputenc}
  \usepackage{amsmath,amsthm}
  \usepackage{enumerate}
  \usepackage{graphicx}
  \usepackage{amssymb,amsxtra}
  \usepackage{fancyref}
  \usepackage{soul}
  \usepackage[all]{xy}
  \usepackage{algorithmicx}
  \usepackage{algorithm}
  \usepackage{amsmath}
  \usepackage{float}
  \usepackage{marginnote}
  \usepackage{hyperref}
  \usepackage[toc, page]{appendix}
  \usepackage{fullpage}
  \usepackage{color}
  \usepackage[OT2,T1]{fontenc}
  \usepackage{a4wide}
  
  \DeclareSymbolFont{cyrletters}{OT2}{wncyr}{m}{n}
  \DeclareMathSymbol{\Sha}{\mathalpha}{cyrletters}{"58}
  
  \newcommand{\Q}{\mathbb{Q}}
  \newcommand{\Z}{\mathbb{Z}}

  \newtheorem{theorem}{Theorem}[section]
  \newtheorem{lemma}[theorem]{Lemma}
  
  \newtheorem{cor}[theorem]{Corollary}
  \newtheorem{prop}[theorem]{Proposition}
  
  \newtheorem{definition}[theorem]{Definition}
  
  \newtheorem{remark}[theorem]{Remark}

  \newtheorem{example}[theorem]{Example}

  \newcommand{\p}{\mathfrak{p}}

  \newcommand{\Cl}{\operatorname{Cl}}
  
  \newcommand{\OO}{\mathcal{O}}
  
  \renewcommand{\mod}{\, \operatorname{mod} \,}
  
  \newcommand{\Res}{\operatorname{Res}}

  \newcommand{\F}{\mathbb{F}}

  \newcommand{\End}{\operatorname{End}}

  \newcommand{\FpGraph}{{\mathcal{G}_\ell(\mathbb{F}_p)}}
  \newcommand{\FpSubGraph}{\mathcal{S}}
  \newcommand{\FpBarGraph}{{\mathcal{G}_\ell(\overline{\mathbb{F}_p})}}
  
\newcommand{\FpBarGraphtwo}{{\mathcal{G}_2(\overline{\mathbb{F}_p})}}
\newcommand{\FpGraphtwo}{{\mathcal{G}_2(\mathbb{F}_p})}
\newcommand{\FpGraphthree}{{\mathcal{G}_3(\mathbb{F}_p})}
\newcommand{\FpBarGraphthree}{{\mathcal{G}_3(\overline{\mathbb{F}_p})}}
\newcommand{\jp}{{\mathbf{j}}}
\newcommand{\jpp}{{\mathbcal{j}}}

\newcommand{\FF}{\mathbb{F}}

\newcommand{\QQ}{\mathbb Q}

\newcommand{\Fp}{\mathbb{F} _p}

\newcommand{\Glfp}{\mathcal{G}_{\ell}(\mathbb F_p)}
\newcommand{\Fpbar}{\overline{\mathbb{F}}_p}
\newcommand{\Fptwo}{\mathbb F_{p^2}}

\title{Adventures in Supersingularland}
\author{Sarah Arpin, Catalina Camacho-Navarro,
Kristin Lauter,\\ 
Joelle Lim, Kristina Nelson, Travis Scholl, Jana Sot\'akov\'a}
\date{September 2019}

\begin{document}

\maketitle

\begin{center}
    \textit{Dedicated to Alice Silverberg}
\end{center}

\begin{abstract}
   In this paper, we study isogeny graphs of supersingular elliptic curves. Supersingular isogeny graphs were introduced as a hard problem into cryptography by Charles, Goren, and Lauter for the construction of cryptographic hash functions (\cite{CGL06}). These are large expander graphs, and the hard problem is to find an efficient algorithm for routing, or path-finding, between two vertices of the graph. We consider four aspects of supersingular isogeny graphs, study each thoroughly and, where appropriate, discuss how they relate to one another. 
   
First, we consider two related graphs that help us understand the structure: the `spine' $\mathcal{S}$, which is the subgraph of $\mathcal{G}_\ell(\overline{\mathbb{F}_p})$ given by the $j$-invariants in $\mathbb{F}_p$, and the graph $\mathcal{G}_\ell(\mathbb{F}_p)$, in which both curves and isogenies must be defined over $\mathbb{F}_p$. We show how to pass from the latter to the former. The graph $\FpSubGraph$ is relevant for cryptanalysis because routing between vertices in $\mathbb{F}_p$ is easier than in the full isogeny graph. The $\mathbb{F}_p$-vertices are typically assumed to be randomly distributed in the graph, which is far from true. We provide an analysis of the distances of connected components of $\mathcal{S}$.

Next, we study the involution on $\mathcal{G}_\ell(\overline{\mathbb{F}_p})$ that is given by the Frobenius of $\mathbb{F}_p$ and give heuristics on how often shortest paths between two conjugate $j$-invariants are preserved by this involution (mirror paths). We also study the related question of what proportion of conjugate $j$-invariants are $\ell$-isogenous for $\ell = 2,3$. We conclude with experimental data on the diameters of supersingular isogeny graphs when $\ell = 2$ and compare this with previous results on diameters of LPS graphs and random Ramanujan graphs.

\end{abstract}

\newpage

\tableofcontents

\newpage 
\section{Introduction}
Supersingular Isogeny Graphs have been the subject of recent study due to their significance in recently proposed post-quantum cryptographic protocols. In 2006, Charles, Goren, and Lauter proposed a hash function based on the hardness of finding paths ({\it routing}) in supersingular isogeny graphs \cite{CGL06}. A few years later, Jao, De Feo, and Plut proposed a key exchange based on supersingular isogeny graphs \cite{DFJP11}. The security of most cryptographic systems currently deployed today relies on either the hardness of factoring large integers of a certain form or the hardness of computing discrete logarithms in certain abelian cyclic groups. Both problems can be efficiently solved using Shor's algorithm on a quantum computer which can handle large scale computation \cite{shor1999polynomial}. In 2015, NIST announced a contest to standardize cryptographic algorithms that are not known to be broken by quantum computers. Now in its second round, SIKE (\url{https://sike.org/}, based on supersingular isogeny graphs) is still in the running for the next public key exchange standard.

While there are no known classical or quantum attacks that break the cryptographic protocols that use supersingular isogeny graphs, the graphs themselves have been relatively unstudied until recently. More study is needed before we can confidently recommend protocols which rely on the difficulty of the hard problem of finding paths in supersingular isogeny graphs.

For distinct primes $p$ and $\ell$, let $\FpBarGraph$ denote the graph whose vertices consist of isomorphism classes of supersingular elliptic curves over $\Fpbar$ and whose edges correspond to isogenies of degree $\ell$ defined over $\Fpbar$. The vertices can be labelled with the $j$-invariant of the curve, which is an $\Fpbar$-isomorphism invariant. For $p \equiv 1 \pmod{12}$ the graph $\FpBarGraph$ is known to be a $\ell+1$-regular Ramanujan graph, and is one of two known families of Ramanujan graphs. 

In this paper, we study two related graphs to help understand the structure of $\FpBarGraph$. First, the full subgraph of $\FpBarGraph$ consisting of only vertices $j\in\Fp$: We denote this subgraph by $\FpSubGraph$ and call it the \textit{spine}, which is new terminology. Second, we look at the graph $\FpGraph$ whose vertices are elliptic curves up to $\Fp$-isomorphism and edges are $\Fp$-isogenies of degree $\ell$, already studied by Delfs and Galbraith (\cite{DelGal01}). As we will need to be specific about the field of definition, we use $j$ to denote a general $j$-invariant, $\jp$ to denote a $j$-invariant in $\Fp$, and $\jpp$ to denote a $j$-invariant in $\mathbb{F}_{p^2}\setminus\Fp$. Note that if two elliptic curves $\Fp$ are twists of each other, then they share the same $j$-invariant. A more formal discussion of the relationship between these graphs can be found in Section 2. 

There have been several approaches tried so far to attack cryptographic protocols based on supersingular isogeny graphs.  One of them uses the quaternion analogue of the graph and presents an efficient algorithm for navigating between maximal orders \cite{KLPT}.  This approach leads to the results presented 
in~\cite{EHLMP} showing that the hardness of the path finding problem is essentially equivalent to the hardness of computing endomorphism rings of supersingular elliptic curves.
One of the other methods considered so far uses the structure of the $\FpGraph$ (\cite{DelGal01}). Better quantum algorithms are known for navigating between $\Fp$-points, so paths to these points are of particular interest. 

In Section 3 of this paper, we compare $\FpGraph$ and the spine $\FpSubGraph$.  The main results of Section 3 show how the components of $\FpGraph$ (which look like volcanoes) fit together to form the spine when passing to the full graph $\FpBarGraph$.  We define the notions of {\it stacking, folding, and attaching} to describe how $\FpGraph$ becomes the spine, when isogenies defined over $\overline{\mathbb{F}_p}$ are added and $j$-invariants which are twists are identified.  In particular for $\ell=2$, Theorem~\ref{theorem:stacking-folding-attaching} shows that only stacking, folding, or at most one  attachment by a new edge are possible to form the spine. Theorem~\ref{theorem:sfa_ell_3} gives an analogous description for $\ell=3$.  For any fixed $\ell$ and $p$,  the resulting shape of the spine depends on the congruence class of $p$, the structure of the class group $\Cl(\OO_K)$, where $K=\QQ(\sqrt{-p})$, and the behavior of the prime above $\ell$ in the class group of $K$, and we show how to determine it explicitly.
In Section~\ref{subsection:distances_of_components} we generate experimental data to study how the components of the spine are distributed throughout the graph and we estimate how many components there are in the spine.

Another important property of Supersingular Isogeny Graphs is that they have an involution which fixes the spine, and which sends a $\jpp$-invariant in $\F_{p^2}$ to its Galois conjugate $\jpp^p$.  If a $\jpp$-invariant in $\F_{p^2}$ has a short path to the spine, then the involution can be applied to that path to produce a path from $\jpp$ to its conjugate $\jpp^p$.  We call such paths {\it mirror-paths}.
In Section \ref{sec:dist-conj-pairs} we study the distance between Galois conjugate pairs of vertices, that is, pairs of $\jpp$-invariants of the form $\jpp$, $\jpp^p$. Our data suggests these vertices are closer to each other than a random pair of vertices in $\FpBarGraphtwo$. In Section \ref{sec:shortest-paths-through-spine} we test how often the shortest path between two conjugate vertices goes through the spine $\FpSubGraph$, or equivalently, contains a $j$-invariant in $\mathbb F_p$. We find conjugate vertices are more likely than a random pair of vertices to be connected by a shortest path through the spine. Finally, we examine the distance between arbitrary vertices and the spine $\FpSubGraph$ in Section \ref{sec:dist-to-Fp}. 

Section 5 provides heuristics on how often conjugate $\jpp$-invariants are $\ell$-isogenous for $\ell = 2,3$, a question motivated by the study of mirror paths provided in Section 4.

Another known family of Ramanujan graphs are certain Cayley graphs constructed by Lubotzky-Philips-Sarnak (LPS graphs \cite{LPS}).  The relationship between LPS graphs and supersingular isogeny graphs is studied in~\cite{WIN4}.  Sardari \cite{Sardari} provides an analysis of the diameters of LPS graphs, and in Section 6 of this paper, we provide heuristics and a discussion of the diameters of supersingular $2$-isogeny graphs.

Our experiments and data suggest a noticeable difference in the Supersingluar Isogeny Graphs $\mathcal{G}_\ell(\overline{\mathbb{F}_p})$ depending on the congruence class of the prime $p \pmod{12}$. It has been known since the introduction of Supersingular Isogeny Graphs into cryptography~\cite{CGL06} that the congruence class of the prime $p$ has an important role to play in the properties of the graph.  In particular it was shown there how the existence of short cycles in the graph depends explicitly on the congruence conditions on $p$.  In this paper, we extend this observation and find significant differences in the graphs depending on the congruence class of $p$.
In summary, the data seems to suggest the following:
\begin{itemize}
    \item $p\equiv 1,7\mod{12}$: 
    \begin{itemize}
        \item smaller $2$-isogeny graph diameters (Section \ref{subsec:diameters_mod_12}),
        \item larger number of spinal components (Section \ref{sec:num-comp}),
        \item larger proportion of $2$-isogenous conjugate pairs (Section \ref{sec:conjugate_pairs_mod_12}),
    \end{itemize}
    
    \item $p\equiv 5,11\mod{12}$: 
    \begin{itemize}
        \item larger $2$-isogeny graph diameters,
        \item smaller number of spinal components,
        \item smaller proportion of $2$-isogenous conjugate pairs.
    \end{itemize}
\end{itemize}

To accompany the experimental results of this paper, we have made the  {\tt Sage}\nocite{sage} code for all the computations available, along with a short discussion of the different algorithms included. The code is posted at

\url{https://github.com/krstnmnlsn/Adventures-in-Supersingularland-Data}.

\section*{Acknowledgements}

This paper is the result of a collaboration started at Alice Silverberg's 60th birthday conference on open questions in cryptography and number theory \url{https://sites.google.com/site/silverberg2018/}.
We would like to thank Heidi Goodson for her significant contributions to this project. We are grateful to Microsoft Research for hosting the authors for a follow-up visit. We would like to thank Steven Galbraith, Shahed Sharif, and Katherine E. Stange for helpful conversations. 
Travis Scholl was partially supported by Alfred P. Sloan Foundation grant number G-2014-13575. Catalina Camacho-Navarro was partially supported by Universidad de Costa Rica.

\section{Definitions}

\begin{definition}
An elliptic curve $E$ is a smooth, projective algebraic curve of genus 1 with a fixed point, denoted $\OO_E$.
\end{definition}

Elliptic curves have a group law, which makes them particularly rich objects to work with. For more background on this, see \cite{AEC}.  Elliptic curves defined over fields of characteristic $p<\infty$ come in two flavors: ordinary and supersingular. In the graphs of elliptic curves considered here, the ordinary and supersingular components are disjoint. We focus on the graph of supersingular elliptic curves, defined here as in Theorem 3.1 of \cite{AEC}:
\begin{definition}[Supersingular Elliptic Curve]
Let $E/K$ be an elliptic curve, $\operatorname{char}(K) = p<\infty$. We say that $E$ is supersingular if any of the following equivalent conditions hold:
\begin{enumerate}[(i)]
    \item The $p^k$-torsion of $E$ is trivial for all $k\in\mathbb{Z}_{k\geq1}$.
    \item The multiplication by $p$-map on $E$ is purely inseparable and the $j$-invariant of $E$ is in $\mathbb{F}_{p^2}$
    \item The endomorphism ring of $E$ is isomorphic to a maximal order in a quaternion algebra.
\end{enumerate}
Elliptic curves which are not supersingular are called ordinary.
\end{definition}
By definition, supersingular elliptic curves can all be identified with a $j$-invariant in $\mathbb{F}_{p^2}$. The $j$-invariants classify the $\Fpbar$-isomorphism classes of supersingular elliptic curves. When we consider supersingular elliptic curves with $j(E)\in\Fp$ up to $\Fp$-isomorphism, two $\Fp$-isomorphism classes will have a single $j$-invariant (\cite[Prop.~2.3]{DelGal01}). One can distinguish between the two curves for instance by considering Weierstrass models of the elliptic curves.

Whenever the field of definition of the $j$-invariant is relevant and not clear from context, we will use the following notation: 
\begin{itemize}
    \item $\jp$ denotes a $j$-invariant in $\Fp$
    \item $\jpp$ denotes a $j$-invariant in $\mathbb{F}_{p^2}\setminus\Fp$.
    \end{itemize}
Otherwise, or for a general $j$-invariant, we denote it simply by $j$.

A characterizing difference between the endomorphism rings of $\Fp$ $j$-invariants and $\Fptwo\setminus\Fp$ $j$-invariants is pointed out in \cite[Prop.~2.4]{DelGal01}: for $p>3$, a supersingular elliptic curve $E/\Fpbar$ is defined over $\Fp$ if and only if $\mathbb{Z}[\sqrt{-p}]\subset\End(E)$.

\subsection{Isogeny Graphs}
There are three graphs to consider. To introduce these graphs, we borrow the following notions from Sutherland \cite[Section 2.2]{Sut13}. We denote by $\Phi_{\ell}[X,Y]$ the $\ell$-modular polynomial. This is a polynomial of degree $\ell+1$ in both $X$ and $Y$, symmetric in $X$ and $Y$ and such that there exists a cyclic $\ell$-isogeny $\phi:E(j_1)\to E(j_2) $ if and only if $\Phi_{\ell}(j_1,j_2)=0$. For $\ell$ prime, all isogenies are cyclic.

In principle, the modular polynomials can be computed and they are accessible via tables for small values of $\ell$, however, their coefficients are rather large, as we see already for $\phi_2(X,Y)$:
\begin{equation}\label{phi2}
\begin{split}
 \Phi_2(X,Y) & =  -X^{2} Y^{2} + X^{3} + 1488 X^{2} Y + 1488 X Y^{2} + Y^{3} - 162000 X^{2} + 40773375 X Y  \\
    & - 162000 Y^{2} + 8748000000 X + 8748000000 Y - 157464000000000
\end{split}
\end{equation}

\begin{definition}[Supersingular $\ell$-isogeny graph over $\Fpbar$: $\FpBarGraph$]
\label{def:Fpbarsupersingular_graph}
The $\FpBarGraph$ graph has vertex set consisting of the $\Fpbar$-isomorphism classes of supersingular elliptic curves over $\Fpbar$, labeled by their $j$-invariants over $\mathbb{F}_{p^2}$. The directed edges from a vertex $j$ correspond to $(j,j')$ where $j'$ is a root of the modular polynomial $\Phi_\ell(j,Y)$. 
\end{definition}

Except possibly at vertices corresponding to $\jp=0,1728$, this is defines an $\ell+1$ regular graph. 
These graphs are known to be Ramanujan graphs (see \cite{CGL06} or \cite{WIN4}).

\begin{definition}[Spine: $\FpSubGraph$]
\label{def:spine}
The $\Fp$ spine, denoted $\FpSubGraph$, is the full subgraph of $\FpBarGraph$ consisting of all vertices with $j$-invariants defined over $\Fp$ and all their edges in $\FpBarGraph$. 
\end{definition}

The number of vertices in $\FpSubGraph$ can be determined from \cite{Cox}
$$
\#\FpSubGraph = \begin{cases}
\frac{1}{2}h(-4p) & \text{ if } p\equiv 1\pmod{4}\\
h(-p) & \text{ if } p\equiv 7\pmod{8}\\
2h(-p) & \text{ if } p\equiv 8\pmod{8}.
\end{cases}
$$
where $h(d)$ is the class number of the imaginary quadratic field $\QQ(\sqrt{d})$.

\begin{definition}[Supersingular $\ell$-isogeny graph over $\Fp$: $\FpGraph$]
\label{def:Fpgraph}
The $\FpGraph$ has vertex set $\Fp$-isomorphism classes of $j$-invariants $\jp\in\Fp$. The edges correspond to $\ell$-isogenies defined over $\Fp$ as well. As noted before, each $j$-invariant will appear as two distinct vertices in this graph.
\end{definition}

\begin{remark}
\label{remark:SpineAndFpDifferences}
It is worthwhile to highlight the differences between $\FpSubGraph$ and $\FpGraph$:
\begin{itemize}
    \item $\FpSubGraph$ has fewer vertices than $\FpGraph$, since the vertices are considered up to $\Fpbar$-isomorphism in the former and $\Fp$-isomorphism in the later.
    \item $\FpSubGraph$ has (likely) more edges than $\FpGraph$, since we consider the edges defined over $\Fpbar$ in the former but only those defined over $\Fp$ in the later. The ``appearance" of these edges when we move from $\FpGraph$ to $\FpSubGraph$ will be discussed more thoroughly in the sequel. 
\end{itemize}
\end{remark}

\begin{remark}
\label{Remark:UndirectedGraphs}

Note that we can consider these graphs can be considered to be un-directed except at the $j$-invariants $\jp = 0,1728$: Every $\ell$-isogeny $\phi: E\to E'$ has a dual $\widehat{\phi}:E'\to E$ of the same degree. The only issues we run into with $\jp=0,1728$ are the extra automorphisms of these curves can compose with the isogenies, affecting the regularity of the graphs at these vertices. We can still consider the graph to be undirected at these vertices, but we will not preserve the multiplicity of the edges with this relaxation.
\end{remark}

\begin{definition}\label{def:equivalent_isog}
We say two isogenies $\phi: E_1 \to E_2$ and $\phi': E_1' \to E_2'$, are equivalent over $k$ if there exist isomorphisms $\varphi: E_1' \to E_1$ and $\psi: E_2 \to E_2'$ over $k$ such that $\phi' = \psi \circ \phi \circ \varphi$.
\end{definition}

Let us make some remarks about $\FpBarGraph$. By definition, this is an $\ell+1$ regular graph, where we can associate an edge $\{j_1,j_2\}$ to an equivalence class of isogenies between two elliptic curve $E_1$ and $E_2$ with $j_1=j(E_1)$ and $j_2=j(E_2)$. Kohel \cite[Chapter 7]{Kohel} proved that very pair of supersingular elliptic curves are connected by a chain of degree $\ell$ isogenies, which implies that the graph is connected. If $p>3$, the number of vertices of $\FpBarGraph$ is

$$\left\lfloor\frac{p}{12}\right\rfloor +\begin{cases}
0 & \text{ if } p\equiv 1\pmod{12},\\
1 & \text{ if } p\equiv 5,7\pmod{12},\\
2 & \text{ if } p\equiv 11\pmod{12}.
\end{cases}
$$

(See \cite[Section V.4]{AEC}.) The congruence condition follows from whether or not the $j$-invariants 0 and 1728 are supersingular or not. 

\begin{remark}
The graph $\FpBarGraph$ is a component (called the supersingular component) of the general $\ell$-isogeny graph where the vertices also include the ordinary $j$-invariants and $\ell$-isogenies between them. Isogenies preserve the properties of ``being ordinary" and ``being supersingular", so these vertices do not mix on connected components of the full $\ell$-isogeny graph.
\end{remark}

It is natural to consider connections between these three graphs. Moving from $\FpGraph$ to $\FpSubGraph$ identifies vertices with the same $j$-invariant and adds edges. To move from $\FpSubGraph$ to $\FpBarGraph$, we can consider adding $\jpp$-invariants in conjugate pairs: starting with $\jp$ in $\FpSubGraph$, if there is an isogeny from $E(\jp)$ to $E(\jpp)$, there is a conjugate isogeny from $E(\jp)$ to the conjugate $E(\jpp)^{(p)}$. 

Indeed, this works for any two $j$-invariants $j, j'$:
if $j$ and $ j'$ satisfy $
    \Phi_\ell(j,j') =0 $ then also 
\begin{align*}
    \Phi_\ell(j^p, (j')^p) = (\Phi_\ell(j, j') )^p = 0
\end{align*} because $\Phi_\ell $ has integer coefficients. This means that for any edge $[j,j']\in \FpBarGraph$, there is a \textit{mirror} edge $[j^p, (j')^p]$. 
Constructing the graph from this perspective leads to the idea of a mirror involution on $\FpBarGraph$: 

\begin{definition}
\label{def:mirror_involution}
If $j$ is a supersingular $j$-invariant, so is its $\Fptwo$-conjugate $j^p$ (in the case that $j = \jp\in \Fp$, $\jp^p = j$). If there is an $\ell$-isogeny $\phi: E(j_1) \to E(j_2)$ then there exists an $\ell$-isogeny $\phi':E(j_1)^{(p)} \to E(j_2)^{(p)}$. This implies that the $(p)$-power Frobenius map on $\Fptwo$ gives an involution on $\FpBarGraph$. We call this the \textit{mirror involution}. 
\end{definition}

The mirror involution fixes the $\Fp$-vertices of the graph.

\begin{definition}\label{def:mirror_path}
We say that a path $P$ with vertices $\{j_0,j_1,j_2,\ldots, j_{n-1}, j_n\}$ (considered as an undirected path) is a mirror path if it is invariant under the mirror involution.
\end{definition}

There exists at least one mirror path between any two conjugate $j$-invariants. One way to find such a mirror path is to find a path from one $j$-invariant, say $\mathbcal{j}_0$, to an $\mathbb{F}_p$ $j$-invariant. Then, conjugate that path to connect with the conjugate of $\mathbcal{j}_0$, which we denote $\mathbcal{j}_0^p$. In summary, a path of the form:
\[\mathbcal{j}_0\to \mathbcal{j}_1\to \cdots \to \mathbcal{j}_n\to  \mathbf{j} \to \mathbcal{j}_n^p\to \cdots \to \mathbcal{j}_1^p \to \mathbcal{j}_0^p \]

Another possibility is for a mirror path between conjugate $\jpp$-invariants to pass through a pair of isogenous conjugate $\jpp$-invariants:

\[\mathbcal{j}_0\to \mathbcal{j}_1\to \cdots \to \mathbcal{j}_n\to \mathbcal{j}_n^p\to \cdots \to \mathbcal{j}_1^p \to \mathbcal{j}_0^p \]

\subsection{Special $j$-invariants} \label{subsection:special_j_invariants}

In this section, we establish a few general facts about $j$-invariants that require special attention.

First of all, there are $j$-invariants corresponding to curves with extra automorphisms that result in the undirectedness of the graph (see Remark \ref{Remark:UndirectedGraphs}). It is a standard fact that $j=1728$ is supersingular if and only if $p \equiv 3 \mod 4$ and $j = 0$ is supersingular if and only if $p \equiv 2 \mod 3$. The computation can be easily argued by CM theory (see, for instance, \cite{IGUSA}). The main idea used is that the $j$-invariant of an elliptic curve $E$ with CM by a quadratic order $\mathcal{O}$ generates the ring class field of $\mathcal{O}$ and such a curve reduces to a supersingular curve modulo $p$ if and only if $p$ is inert in $\mathcal{O}$.

\begin{example}
\label{prop:8000}
Elliptic curves with $j$-invariant $j=8000$ are supersingular over $\Fp$ if and only if $p\equiv 5,7\mod 8$.
\end{example}
\begin{proof}
The number field $\mathbb{Q}(\sqrt{-2})$ has Hilbert class polynomial $x - 8000$, meaning that the $j$-invariant $j=8000$ has CM by $\mathbb{Q}(\sqrt{-2})$. This $j$-invariant will be supersingular whenever $p$ is inert in $\mathbb{Q}(\sqrt{-2})$. Hence we only need to compute $\left(\frac{-2}{p}\right) = -1$, which gives the congruence conditions.
\end{proof}

\begin{example}
\label{prop:-3375}
Elliptic curves with $j$-invariant $j = -3375$ are supersingular over $\Fp$ if and only if $p\equiv 3,5,6\mod{7}$.

\end{example}

\begin{proof}
The Hilbert class polynomial of $\mathbb{Q}(\sqrt{-7})$ is $x+3375$. Hence $j = -3375$ will be supersingular over $\mathbb{F}_p$ whenever $p$ is inert in $\mathbb{Q}(\sqrt{-7})$. The rest follows from evaluating $\left(\frac{-7}{p}\right) = -1$.
\end{proof}

\subsubsection{Self-isogenies}
\label{subsubsec:SelfIsogenies}

Next we turn our attention to the self-loops in the graph $\FpBarGraph$, which are easily read off from the factorization of $\Phi_\ell(X,X)$ as follows:
A $j$-invariant $j$ admits a self-isogeny if and only if $\Phi_\ell(j,j)=0$.

For instance, consider the modular polynomial $\Phi_2(X,Y)$ (as seen in (\ref{phi2})). Now, $\Phi_2(X,X)$ factors  over $\Z$ as 
\begin{equation}\label{eq:self-loops-for-1728}
\Phi_2(X,X) = -(X + 3375)^2(X - 1728)(X - 8000),    
\end{equation}

therefore, the only loops in $\FpBarGraphtwo$ are at the following vertices:
\begin{itemize}
    \item $\jp=-3375$ has two loops, 
    \item $\jp = 1728$ has one loop (we will show where this loop comes from in Example \ref{cor:endo_of_1728}),
    \item $\jp= 8000$ has one loop. 
\end{itemize}
In particular, no $j$-invariant over $\Fptwo\setminus \Fp$ has a self isogeny. Note that these $j$-invariants may not be supersingular, so they may not appear on $\FpBarGraphtwo$ for every $p$.

\subsubsection{Double edges}

We use the following general lemma about double edges in the $2$-isogeny graph. Note that this lemma applies mutatis mutandis for ordinary curves, replacing $\FpBarGraph$ by the $\ell$-isogeny graph of ordinary elliptic curves.

\begin{lemma}[Double edge lemma] \label{lemma:double_edge_lemma}
If two $j$-invariants $j_1,j_2$ in the $\ell$-isogeny graph $\FpBarGraph$  have a double-edge between them, then they are roots of the polynomial
\begin{align} \label{polynomial:double_edges}
  \Res_\ell(X) :=  \operatorname{Res}\left( \Phi_
\ell(X,Y), \frac{d}{dY}\Phi_\ell(X,Y); \,\, Y .
\right) 
\end{align} which is a polynomial of degree bounded by $2 \ell \cdot (2 \ell - 1)$.
\end{lemma}

\begin{proof}[Proof of the double edge lemma.] 
Suppose that $j_1$ and $j_2$ are two vertices in the $\ell$-isogeny graph connected with a double edge. Considered as a polynomial in $Y$, this means
\[\Phi_\ell(j_1,Y) = (Y - j_2)^2\cdot g(j_1, Y)\]
for some $g(j_1, Y)\in\overline{\mathbb{F}}_{p}[Y]$. The derivative $$\frac{d}{dY}\Phi_\ell(j_1,Y)  $$with respect to $Y$ then also vanishes at $Y = j_2 $. 
This means that the polynomials $\Phi_\ell(X,Y)$ and $\frac{d}{dY}\Phi_\ell(X,Y) $ share a root when plugging in $X = j_1$. But this means that $j_1$ is a root of the resultant 
$$    \operatorname{Res}\left( \Phi_
\ell(X,Y), \frac{d}{dY}\Phi_\ell(X,Y); \,\, Y .
\right). $$

Since the total degree of $\Phi_\ell(X,Y)$ is $2 \ell$ and the total degree of $ \frac{d}{dY}\Phi_\ell(X,Y);$ is $2 \ell - 1$ and the resultant of two polynomials $P(X,Y)$ and $Q(X,Y)$ of total degrees $d$ and $e$ has generically degree $d\cdot e$, we obtain the bound 
\begin{align*}
    \# \{ j: \text{ there is a double edge from }j \} \leq 2\ell \cdot (2 \ell -1). \qedhere \end{align*}  
\end{proof}

The bound in the lemma is not tight as we will see in the following corollary for $\ell = 2$.

\begin{cor}[Double edges for $\ell = 2$] \label{cor:double_edges_for_ell_2}
If $\ell = 2$ and there is a double edge from $j$ in $\FpBarGraphtwo$, then the $j$-invariant $j$ is in the following list:
\begin{itemize}
    \item $\jp = 0$, 
    \item $\jp = 1728$,
    \item $\jp = -3375$,
    \item $j$ is a root of $X^2 + 191025X - 121287375$.
\end{itemize}

Moreover, at $\jp = 0$ we obtain a triple edge and $\jp= 0$ is the only $j$-invariant that admits a triple edge in $\FpBarGraphtwo$.
\end{cor}

\begin{proof}
By lemma \ref{lemma:double_edge_lemma}, double edges in $\FpBarGraphtwo$ can only occur at the roots of
\begin{equation}\label{eq:resultantPhi2andDeriv}
  \Res_2(X) =   \left(-1\right) \cdot 2^{2} \cdot X^{2} \cdot (X - 1728) \cdot (X + 3375)^{2} \cdot (X^{2} + 191025 X - 121287375)^{2}
\end{equation}
For the $j$-invariants $0, 1728$ and $-3375$, we identify the double edges by factoring $
    \Phi_2(j, X) .$
\begin{enumerate}
    \item For $\jp=0$, we have
    \begin{align*}
        \Phi_2(0, X) = (X - 54000)^{3}.
    \end{align*}
    There are three outgoing $2$-isogenies from $\jp=0$ to $\jp = 54000$.
    
    \item For $\jp= 1728$, we have
    \begin{align*}
        \Phi_2(1728,X)= (X - 1728) \cdot (X - 287496)^{2}
    \end{align*}
    there is always a self-$2$-isogeny (explained further in \ref{cor:endo_of_1728}) and two $2$-isogenies to $\jp = 287496$. These are defined over $\Fp$ for any $p$: from the model $y^2 = x^3 -x$, they are given by maps
    \begin{align*}
     &    \left(\frac{x^{2} + x + 2}{x + 1}, \frac{x^{2} y + 2 x y - y}{x^{2} + 2 x + 1}\right)  \text{ and }   \left(\frac{x^{2} - x + 2}{x - 1}, \frac{x^{2} y - 2 x y - y}{x^{2} - 2 x + 1}\right).
    \end{align*}

     \item For $\jp = -3375$, we have
     \begin{align*}
         \Phi_2(-3375, X) = (X - 16581375) \cdot (X + 3375)^{2}
     \end{align*} and so there are always two self-$2$-isogenies.
\end{enumerate}

We also note that $j = 0$ is the only $\jp$-invariant that can admit a triple edge. Indeed, since away from the vertices $1728$ and $0$ we can think of the graph as being undirected with $3$ edges from every vertex, having a triple edge would mean having two isolated vertices in $\FpBarGraphtwo$. This is not possible.
\end{proof}

The double-edges from Corollary \ref{cor:double_edges_for_ell_2} appear in the supersingular $2$-isogeny graph only when these $j$-invariants are supersingular. 

\begin{remark}
The factors of the polynomial $\Res_2(X)$ (as seen in  \eqref{eq:resultantPhi2andDeriv}) are Hilbert class polynomials for imaginary quadratic fields. This is to be expected: A double edge $[j_1, j_2]$ is a $2$-cycle of non-dual $2$-isogenies (not equal to the multiplication map $[2]$). The ring $\End_{\Fp}(E_{J_1})$ has an non-trivial element of norm $4$ corresponding to this 2-cycle. The only quadratic imaginary fields that contain an element of norm $4$ are $\Q(\sqrt{-1}), \Q(\sqrt{-3}), \Q(\sqrt{-7}), \Q(\sqrt{-15}) $. \end{remark}

\begin{remark} \label{remark:double_edges_large_ell} The above remark generalizes for any $\ell$: 
the polynomial \begin{align*}
    \Res_\ell(X)
\end{align*} is a product of (ring) class polynomials of quadratic orders containing a nontrivial element of norm $\ell^2$.
\end{remark}

Short cycles are considered carefully in Section 6 of \cite{CGL06}: a sufficient condition so that there are no cycles of length $2$ in $\FpBarGraphtwo$ is $p\equiv 1\mod{420}$. 

We will use Lemma \ref{lemma:double_edge_lemma}  together with Lemma \ref{lemma:one-two-isogeny} that if there is an edge in $\FpBarGraphtwo$  between two $j$-invariants $\jp_1, \jp_2 \in \Fp$ that corresponds to an isogeny which is not defined over $\Fp$ (i.e., which is not an edge in $\FpGraph$), then there is a double edge between $\jp_1$ and $\jp_2$ in $\FpBarGraphtwo$, and hence $\jp_1$ and $\jp_2$ are among the values listed in Lemma \ref{lemma:double_edge_lemma}.

\section{Structure of the $\Fp$-subgraph: the spine $\FpSubGraph$}

In this section, we investigate the shape of the spine $\FpSubGraph$, which was defined in \ref{def:spine} to be the subgraph of $\FpBarGraph$ consisting of the vertices defined over $\Fp$. 
The motivation for studying the structure of this subgraph is the existence of attacks on SIDH that work by finding $\Fp$ $\jp$-invariants: the idea for such a possible attack was presented in~\cite{DelGal01}, and a quantum attack based on this idea was given by Biasse, Jao and Sankar~\cite{BJS}. See also \cite{GPSTOnTheSecurity} for an overview of the security considerations for SIDH.

In Section \ref{volcano}, we start with the graph $\FpGraph$, the structure of which is understood well. 
It was studied in depth by Delfs and Galbraith in \cite{DelGal01}, on which we base our investigations. Moreover, a version of the $\Fp$-graph (allowing edges corresponding to $\ell$-isogenies for multiple primes) has also been proposed for post-quantum cryptography in \cite{CSIDH}.
We recall the results of \cite{DelGal01} in some detail and give a few explicit examples of their results about endomorphisms of certain elliptic curves (for instance, Example \ref{example:j=0_always_on_the_floor}).

In Section \ref{subsec:folding-stacking-attaching}, we discuss how the Spine $\FpSubGraph$ can be obtained from $\FpGraph$ in two steps: first vertices corresponding to the same $j$-invariant are identified and then a few new edges are added. The possible ways the connected components can identify are given by Definition \ref{def:stacking} and we call them \textit{stacking, folding, attaching along a $j$-invariant} and \textit{attaching by a new edge}.  

In Section \ref{sec:stacking_folding_attaching_large} we study stacking, folding and attaching for $\ell >2$ and give an example of the theory we develop for $\ell = 3$ in Section \ref{sec:stacking_folding_attaching_3}. In this section we also give a complete description of stacking, folding and attaching in Theorem \ref{theorem:sfa_ell_3}. 
We return to the case $\ell =2$ in \ref{subsec:stacking_folding_attaching_2}, giving a similarly complete theorem in \ref{theorem:stacking-folding-attaching}, and some data on how often attachment happens. Section \ref{subsection:distances_of_components} contains some experimental data on the distances between the connected components of $\FpSubGraph \subset \FpGraphtwo$.

\subsection{Structure of the {$\Fp$}-Graph {$\FpGraph$}}
\label{volcano}

\subsubsection{Preliminaries} \label{subsection:volcano}

To understand the spine $\FpSubGraph$ (Definition \ref{def:spine}), we look at $\FpGraph$ (Definition \ref{def:Fpgraph}). Recall that the vertices of $\FpGraph$ are all supersingular elliptic curves defined over $\Fp$, up to $\Fp$-isomorphism, and the edges in $\FpGraph$ are isogenies defined over $\Fp$. Keep in mind the differences between $\FpSubGraph$ and $\FpGraph$, highlighted in Remark \ref{remark:SpineAndFpDifferences}.

To see how many vertices of $\FpGraph$ correspond to the same $j$-invariant, we look at twists of elliptic curves. By Proposition 5.4 of \cite{AEC}[Chapter X], for $\jp \neq 0, 1728$, the set of twists is isomorphic to (assuming $p > 3$)
\begin{align*}
    \Fp^*/ (\Fp^*)^2 \cong \Z / 2\Z,
\end{align*} so there are two vertices corresponding to the same $j$-invariant $\jp$. 

Similarly, for $\jp = 1728$, the set of twists is isomorphic to \begin{align*}
    \Fp^* / (\Fp^*)^4 .
\end{align*}
The $j$-invariant $\jp = 1728$ is supersingular if and only if $p \equiv 3 \mod 4$ (equivalently, $p-1 \equiv 2 \mod 4$), so $(\Fp^*)^4 = (\Fp^*)^2$. Hence, there are two vertices of $\FpGraph$ corresponding to $\jp=1728$, as well. These vertices correspond to quartic twists, rather than quadratic twits, which we will use in Example \ref{cor:endo_of_1728}.

For $\jp = 0$, the set of twists is isomorphic to
\begin{align*}
     \Fp^* / (\Fp^*)^6 .
\end{align*}
We know that $\jp=0$ is supersingular if and only if $p \equiv 2 \mod 3$ (equivalently $p- 1 \equiv 1 \mod 3$), so $(\Fp^*)^6 = (\Fp^*)^2$. Hence there are also two vertices of $\FpGraph$ corresponding to $\jp =0$.

The structure of $\Glfp$ is explained in \cite{DelGal01}. They show that the $\FpGraph$ graph looks very similar to an isogeny graph of ordinary curves. Upon recalling some of the main definitions, we present a simplified version of their construction results here, restricting many general results to the case $\ell = 2$.

Let $K := \Q(\sqrt{-p})$. Start with the  definition of supersingularity for elliptic curves over $\Fp$: an elliptic curve $E/ \Fp$ is supersingular if and only if $\Z[\sqrt{-p}] \subset \End_{\Fp}(E)$. If $p \equiv 3 \mod 4$, the order $\Z[\sqrt{-p}]$ is contained in the maximal order $\Z \left[ \frac{1+\sqrt{-p}}{2} \right]$. To distinguish between these two possible endomorphism rings, we have the following definitions.

\begin{definition}[Surface and Floor] \label{surfacefloor} Let $E$ be a supersingular elliptic curve over $\Fp$. We say $E$ is on the surface (resp. $E$ is on the floor) if $End_{\Fp}(E) = \mathcal O_K$ (resp. $End_{\Fp}(E) = \Z[\sqrt{-p}]$). For $p \equiv 1 \mod 4$, surface and floor coincide.
\end{definition}

\begin{definition}[Horizontal and Vertical Isogenies.] \label{horizontal}
Let $\varphi$ be an $\ell$-isogeny between supersingular elliptic curves $E$ and $E'$ over $\Fp$. If $\End_{\Fp}(E) \cong \End_{\Fp}(E')$ then $\varphi$ is called horizontal. Otherwise, if $E$ is on the floor and $E'$ is on the surface, or vice versa, $\varphi$ is called vertical.\\
\end{definition}

The following is Theorem 2.7 of \cite{DelGal01}. We revisit this result in Sections \ref{subsection:1mod4} and \ref{sec:p3mod4}.

\begin{theorem}\label{theorem:volcanos}
Let $p > 3$ be a prime.

\begin{enumerate}
    \item For $\ell > 2$, there are two horizontal isogenies from any vertex and there are no vertical isogenies, provided $\left( \frac{-p}{\ell} \right) =1$, otherwise there are no $\ell$-isogenies. Hence every connected component of $\FpGraph$ is a cycle.
    
    \item Case $p \equiv 1 \mod 4$. There is one level in $\FpGraphtwo$: all elliptic curves $E$ have $\End_{\Fp}(E) = \Z[\sqrt{-p}]$. For $\ell=2$: from each vertex there is one outgoing $\Fp$-rational $2$-isogeny. 
    
    There are $h(-4p)$ vertices on the surface (which coincides with the floor).

\item Case $p \equiv 3 \mod 4$. There are two levels in $\FpGraphtwo$: surface and floor. 
For $\ell = 2$:
\begin{enumerate}
\item If $p \equiv 7 \mod 8$, there is exactly one vertical isogeny from any vertex on the surface to a vertex on the floor, every vertex on the surface admits two horizontal isogenies and there are no horizontal isogenies between the curves on the floor.

There are $h(-p)$ vertices on the floor and $h(-p)$ vertices on the surface.

\item If $p \equiv 3 \mod 8$, from every vertex on the surface, there are exactly three vertical isogenies to the floor, and there are no horizontal isogenies between any vertices.

There are $3 \cdot h(-p)$ vertices on the floor and $h(-p)$ vertices on the surface.
\end{enumerate}
\end{enumerate}
\end{theorem}

This implies that every connected component of $\FpGraphtwo$ is an isogeny volcano, first studied by Kohel \cite{Kohel}. For a reference on the name and basic properties we refer to \cite{Sut13}.

\begin{proof}Theorem 2.7 in \cite{DelGal01}. We will reference the methods in the proof:

\begin{enumerate}
    \item There is a one-to-one correspondence between supersingular elliptic curves over $\Fp$ and elliptic curves defined over $\mathbb{C}$ with CM by $\OO \in \{ \Z[\sqrt{-p}], \Z \left[ \frac{1+ \sqrt{-p}}{2 }  \right]\}$.
    
     \item Isogeny graphs of CM curves have a volcano structure, and the edges of the volcano reduce to edges in the graph $\FpGraph$. Hence, there will be a volcano-like structure over $\Fp$. 
    
    \item The reduction does not add any more edges. For $\ell = 2$ we reprove the main ingredient in Lemma \ref{lemma:full_two_torsion}. Adding more edges between $\Fp$ vertices would imply that $E[\ell] \subset E(\Fp)$ and this cannot happen for $\ell > 2$ because the curves are supersingular.
    
    Hence we will see the volcano structure over $\Fp$. \qedhere
\end{enumerate}

\end{proof}

Let $K = \Q(\sqrt{-p})$, $\p$ be a prime above $\ell = 2$ in $\OO_K$, and $h = \# \Cl(\OO_K)$ the class number of $K$. Let $n$ be the order of $\p$ in $\Cl(\OO_K)$. The surface of any volcano in $\FpGraphtwo$ is a cycle of precisely $n$ vertices. There are $h/n$ connected components (volcanoes) in $\FpGraphtwo$, the index of $\langle \p \rangle$ in $\Cl(\OO_K)$.

\begin{figure}[h!]
\centering
\includegraphics[scale=0.7]{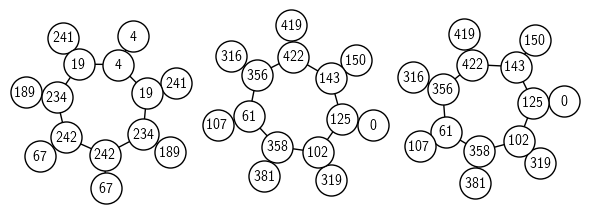}
\caption{The graph $\FpGraphtwo$ for $p = 431$. The integer labels are the $j$-invariants of each curve. Each component is a volcano, with an inner ring of surface curves and the outer vertices all being curves on the floor. $431 \equiv 7 \mod 8$, so we are in case $2.a$ of Theorem \ref{theorem:volcanos}. The class number of $\Q(\sqrt{-431})$ is $3 \cdot 7 = 21$ and the orders of the two primes above $2$ are $7$.}
\label{431}
\end{figure}

\begin{lemma} \label{lemma:full_two_torsion}
Let $p$ be a prime, $p>5$.  Let $E$ be a supersingular elliptic curve defined over $\mathbb{F}_p$. Then
\begin{align*}
   \End_{\Fp}(E) = \Z\left[\frac{1+\sqrt{-p}}{2}\right] \quad \text{if and only if} \quad  E[2] \subset E(\F_p).
\end{align*}
\end{lemma}

For $p \equiv 1 \mod 4$, the ring $Z[(1+\sqrt{-p})/2]$ is not an order of $\mathcal{O}_{K}$, so supersingular elliptic curves in $\Fp$ do not have their full two-torsion defined over $\Fp$.

\begin{proof}
This proof is an adaptation of the techniques on page 7 of \cite{DelGal01}. 

Let $E$ be any supersingular elliptic curve defined over $\FF_p$. Then $\# E(\FF_p) = p +1$ and the minimal polynomial of Frobenius $\pi$ is $x^2 + p$. This implies $\pi = \pm \sqrt{-p} \in \Z(\sqrt{-p})$. We have \begin{align*}
\Q(\sqrt{-p})\supseteq\End_{\Fp}(E) \supset \Z[\pi] = \Z [\sqrt{-p}].
\end{align*}

Thus, $\End_{\Fp}(E)$ is either isomorphic to $\Z [\sqrt{-p}]$ or $\Z\left[\frac{1+\sqrt{-p}}{2}\right]$, as an order in $\Q(\sqrt{-p})$.

First, suppose $E[2] \subset E(\F_p)$. Take $P\in E[2]$. Frobenius acts as the identity on  the $2$-torsion:
\[\pi(P) = P \Rightarrow (1 + \pi)(P) = 0,\]
where $0$ denotes the identity of $E$, since $-P = P$ for $P\in E[2]$.
Hence $E[2] \subset \ker(1 + \pi)$.
Isogenies have the universal property of a quotient, so we obtain the factorization 
\begin{align*} 
\xymatrix{E \ar[rr]^{[2]} \ar[rd]^{1 + \pi} & & E \\
& E \ar@{..>}[ur]_{\phi} & 
 }
\end{align*}
and conclude $1 + \pi = [2] \phi$. The map $\phi = \frac{1+\pi}{2}$ is  $\mathbb{F}_p$-rational, since it is the quotient of $\mathbb{F}_p$-rational maps, so $\frac{1 + \pi}{2} \in \End_{\mathbb{F}_p}(E)\cong \Z \left[ \frac{1+\sqrt{-p}}{2} \right]$.
 
Conversely, suppose $\End_{\mathbb{F}_p}(E) \cong \Z \left[ \frac{1+\sqrt{-p}}{2} \right]$. Consider $\phi = \frac{1-\pi}{2}\in \End_{\mathbb{F}_p}(E) \cong \Z \left[ \frac{1+\sqrt{-p}}{2} \right]$, where $\pi$ is Frobenius. Take any $P\in E[2]$:

\begin{align*}
(1- \pi)(P) = 2\phi(P) = \phi \cdot 2(P) = 0 \Longrightarrow  1 = \pi \quad \text{ on } E[2].
\end{align*} 
Frobenius acts trivially on $E[2]$, so we have $E[2] \subset E(\F_p)$.
\end{proof}

The following corollary of Lemma \ref{lemma:full_two_torsion} will be essential in our discussion in Section \ref{subsec:stacking_folding_attaching_2}. 

\begin{cor}[Endomorphism rings of quadratic twists]\label{lemma:quadratic_twists} Let $p > 3$ be a prime, let $E$ be an elliptic curve defined over $\Fp$ and let $E^t$ denote its quadratic twist. Then 
\begin{align*}
    \End_{\Fp} (E) \cong \End_{\Fp} (E^{t}).
\end{align*}
\end{cor}

\begin{proof}
Suppose $E$ is given by the equation
\begin{align*}
    E: y^2 = x^3 + ax+ b .
\end{align*}
Let $d$ be a quadratic non-residue modulo $p$. Then the quadratic twist is given by the equation 
\begin{align*}
     E^{t}  : y^2 = x^3 +d^2 a x + d^3b 
\end{align*} 
and the isomorphism $E \rightarrow E^{t}$ defined over $\F_{p^2}$ is given by 
\[(x,y) \mapsto \left(\frac{x}{d}, \frac{y}{d\sqrt{d}} \right).\]

$2$-torsion points $(x,y)$ satisfy $y= 0$, so $E[2] \subset E(\Fp)$ if and only if $E^{t}[2] \subset E^{t}(\Fp)$. The result follows from Lemma \ref{lemma:full_two_torsion}.
\end{proof}

 \begin{example}[$j = 0$ is always on the floor] Suppose $p > 3$. Any supersingular elliptic curve $E_0 / \Fp$ with $j$-invariant $0$ satisfies $\End_{\Fp} (E) = \Z[\sqrt{-p}]$. \label{example:j=0_always_on_the_floor}
\end{example} 
\begin{proof}
$E_0$ is supersingular if and only if $p \equiv 2 \mod 3$, so $p\equiv2\mod{3}$. Take a short Weierstrass model $E_0 : y^2 = x^3 - d$.
By inspection, 
\begin{align*}
         E[2] \subset E(\Fp) \Longleftrightarrow x^3 - d \text{ splits completely} \Longleftrightarrow \zeta_3 \in \Fp  \Longleftrightarrow 3 | p-1,
\end{align*} 
where $\zeta_3$ denotes a $3$rd root of unity. However, we have $p \equiv 2 \mod 3$, so $E[2]$ is not defined over $\Fp$.
By Lemma \ref{lemma:full_two_torsion} we have $\End_0(E) = \Z[\sqrt{-p}]$.
\end{proof}

\begin{example}[$j = -3375$ is always on the floor] Let $p > 3$ be a prime. Any supersingular elliptic curve $E_{-3375}/ \Fp$ with $j$-invariant $-3375$ satisfies $\End_{\Fp} (E_{-3375}) = \Z[\sqrt{-p}]$.
\end{example}

\begin{proof}
We have $\Phi_2(-3375, x) = (x - 16581375) \cdot (x + 3375)^{2}$. Suppose that either of the vertices corresponding to the $j$-invariant $-3375$ in the graph $\FpGraphtwo$ lies on the surface, and thus has endomorphism ring isomorphic to $\Z\left[\frac{1 + \sqrt{-p}}{2}\right]$. The $\jp$-invariants on the surface have three neighbours. Since there are no loops in $\FpGraphtwo$, this vertex would have two neighbours with $j$-invariants $-3375$, but there cannot be three vertices corresponding to $\jp = -3375$. Hence $\End_{\Fp}(E_{-3375}) \cong \Z[\sqrt{-p}]$.
     
Also note that the two self-isogenies of $\jp = -3375$ are not defined over $\Fp$. 
\end{proof}
If $\jp = 1728$ and $\jp = -3375$ are both supersingular ($p\equiv 3 \mod{4}$ and $p \equiv 3,5,6 \mod 7$), the proof also allows us to conclude that the endomorphism ring $\End_{\Fp} (E)$ of any supersingular elliptic curve $E$ with $j$-invariant $j(E) = 16581375$ is $\Z \left[ \frac{1+ \sqrt{-p}}{2} \right]$.

\begin{example}[The $j$-invariant $1728$ is both on the surface and on the floor.] \label{cor:endo_of_1728} Suppose $p > 3$ with $p \equiv 3 \mod 4$. The isogeny \begin{align*}
\phi: E_{1728}: y^2 = x^3 - x  & \rightarrow y^2 =x^3 + 4x = : E_{1728}^t \\
(x,y) & \mapsto \left(\frac{x^{2} + x + 2}{x + 1}, \frac{x^{2} y + 2 x y -  y}{x^{2} + 2
x + 1}\right)
\end{align*} is a vertical $2$-isogeny with kernel $(0,0)$ of non-$\Fp$-isomorphic supersingular elliptic curves with $j$-invariant $1728$.
\end{example}

Note that $E_{1728}^t$ is a quartic twist, not a quadratic twist, so Lemma \ref{lemma:quadratic_twists} does not apply. 
\begin{proof}
The isogeny $\phi$ was obtained by Velu's formulas.
Factoring the right-hand side of the Weierstrass equation for $E_{1728}$, we see $E[2]\subset E(\F_p)$. By Lemma \ref{lemma:full_two_torsion},
\begin{align*}
    \End_{\Fp}(E_{1728}) \cong \Z \left[ \frac{1+\sqrt{-p}}{2} \right]
\end{align*}
Since $p \equiv 3 \mod 4$, the $\Fp$ points of $E_{1728}^t[2]$ are precisely $\{\OO_E, (0,0)\}$. Again by Lemma 
\ref{lemma:full_two_torsion},
\begin{align*}
    \End_{\Fp}(E_{1728}^t) \cong \Z \left[ \sqrt{-p} \right],
\end{align*}
so $\phi$ is a vertical isogeny.
\end{proof}

Example \ref{cor:endo_of_1728} is the only vertical isogeny between two elliptic curves with same $j$-invariants.

\begin{cor} \label{cor:surface_floor_is_true}
Let $v_a, w_a$ be the distinct vertices in $\FpGraph$ corresponding to the $j$-invariant $\jp = a \in \Fp$ for $a \neq 1728$. Then, either $v_a$ and $w_a$ are either both on the floor or both on the surface of $\FpGraph$.
\end{cor}

\begin{proof} The case of $\jp= 0$ was handled in Example \ref{example:j=0_always_on_the_floor}. 
For $\jp \neq 0, 1728$, the two vertices $v_\jp$ and $w_\jp$ correspond to an elliptic curve and its quadratic twist. The result follows from Corollary \ref{lemma:quadratic_twists}.
\end{proof}

Another proof of this statement can be found in the appendix of \cite{KANEKO} and is obtained by a careful examination of Hilbert polynomials of discriminant $-p$ and $-4p$, considered modulo $p$. 

Kaneko actually proves that $$\gcd(h_{-4p}(x), h_{-p}(x)) = x - 1728, $$ which translates to the statement that $j = 1728$ is the only $j$-invariant that can be both on the surface and the floor. Kaneko in turn gives credit to 
\cite{IBUKIYAMA}, who proved the statement (and more) in purely quaternionic terms.

Now, we describe the potential shapes of $\FpGraphtwo$. The results are given in \cite{DelGal01}, however, we recall these potential shapes of $\FpGraphtwo$ to compare with those of $\FpSubGraph \subset \FpBarGraphtwo$.

\subsubsection{The graph $\FpGraphtwo$ in the case of $p \equiv 1 \mod 4$} \label{subsection:1mod4}

    For $p \equiv 1 \mod 4$, the ring $ \Z[\sqrt{-p}]$ is the maximal order in $\Q(\sqrt{-p})$ and the prime $2$ is ramified. 

\begin{lemma}\label{lem:spine-1-mod-12}
    Suppose that $p > 7$. Then each connected component of $\FpGraphtwo$ is a single edge and the edges correspond to horizontal isogenies.

\end{lemma}
\begin{proof} Since $p \equiv 1 \mod 4$, the ring $\Z[\sqrt{-p}]$ is the ring of integers in $\Q(\sqrt{-p})$ and hence, any supersingular elliptic curve over $\Fp$ satisfies $$\End_{\Fp} (E) \cong \Z[\sqrt{-p}]$$
All of these edges are horizontal isogenies because all the curves satisfy $\End_{\Fp}(E) = \Z[\sqrt{-p}]$.

The proof of this is already in \cite{DelGal01}. We present three proofs of the first statement.
    
    \begin{enumerate}
        \item Since every elliptic curve over $\Fp$ has $p +1$ points and $p + 1 \equiv 1 + 1 = 2 \mod 12$, we see that $\# E[2] =2 $, that is, there is exactly one point of order $2$ defined over $\Fp$ and hence exactly one $2$-isogeny defined over $\Fp$. 
        
        \item Because the ring $\Z[\sqrt{-p}]$ is already the maximal order, Lemma \ref{lemma:full_two_torsion}, we get that $E[2] \not\subset E(\F_p)$ and so there can only be one outgoing $2$-isogeny just like in the previous case.
        
        \item Since $(2)$ is ramified in $\OO_K = \Z[\sqrt{-p}]$, it has order $2$ in $\Cl(\OO_K)$ (this is since $p > 7$ and so there are no elements of norm $2$ in $\Z[\sqrt{-p}]$). We know that the volcano is a cycle with the number of edges equal to the order of the prime above $2$ in $\Cl(\OO_K)$, and hence we recover cycles of length $2-1= 1$. \qedhere 
    \end{enumerate} 
\end{proof}

\subsubsection{The graph $\FpGraphtwo$ in the case of $p \equiv 3 \mod 4$}
\label{sec:p3mod4}

We will use the construction in the proof of Theorem  \ref{theorem:volcanos} to describe the shape of the components of $\FpGraphtwo$.
Since $p \equiv 3 \mod 4$, we have two possible orders for endomorphism rings, and  \begin{align*}
   \mathcal{O} =  \Z[\sqrt{-p}] \subsetneq \mathcal{O}_K = \Z\left[\frac{1+\sqrt{-p}}{2}\right]
\end{align*} is an inclusion of orders of index $2$. 
To see how the prime above $2$ acts on the points in $\FpGraphtwo$ \ref{subsection:volcano}, consider the splitting behavior of $(2)\OO_K$:
\begin{enumerate}
    \item for $p \equiv 3 \mod 8$ the prime $2$ is is inert,
    \item for $p \equiv 7 \mod 8$ the prime $2$ splits into two prime ideals.
\end{enumerate}

These two congruence conditions will result in different shapes of $\FpGraphtwo$. We also consider $\jp = 1728$, as the extra automorphisms affect isogenies between $\jp=1728$ and its neighbors.

\subsubsection*{Case 1. $p \equiv 3 \mod 8$}

Let $K=\mathbb{Q}(\sqrt{-p})$. $2$ is inert in $K$, so the prime $\p$ of $\OO_K$ above $2$ has order $1$ in $\operatorname{Cl}(\OO_K)$. From Theorem \ref{theorem:volcanos}, any component of the $\FpGraphtwo$ will be a volcano with surface of size $1$ connected to the lower-levels as a `claw': There will be three edges going out of any vertex on the surface. See Figure \ref{Fig:Claws}. The elliptic curves $E$ on the surface have endomorphism ring $\Z[\frac{1+\sqrt{-p}}{2}]$ and $E[2] \subset E(\mathbb{F}_p)$, so there are three outgoing 2-isogenies defined over $\Fp$.

The volcano stops at this depth, because there are only two possible endomorphism rings: $\Z[\sqrt{-p}]$ and $\Z \left[ \frac{1+\sqrt{-p}}{2} \right]$. Therefore, the volcanoes will be \textit{claws} for $p\equiv 3\mod{8}$.

\begin{figure}[h]
\centering
\includegraphics[width=80mm]{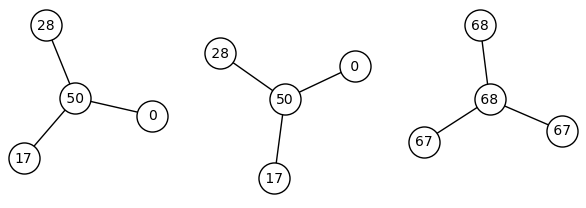}
\caption{The graph $\FpGraphtwo$ for $p= 83$: We clearly see the claw structure.}
\label{Fig:Claws}
\end{figure}

\subsubsection*{Case 2. $p \equiv 7 \mod 8$.}

In this case, the ideal $(2)\OO_K$ splits into two conjugate prime ideals. In general, they can have any order in the class group, but they are never principal. See Figure \ref{431} for an example of this case.

\subsubsection*{Neighbours of $j =1728$.}
\label{subsubsec:Neighboursof1728}

In Example \ref{cor:endo_of_1728}, we saw that that there is always a vertical isogeny from a vertex $v_{1728} $ on the surface to a vertex $w_{1728}$ on the floor of $\FpGraph$. Moreover, looking at the modular polynomial $$\phi_2(1728, x)  = (x - 1728) \cdot (x - 287496)^{2}, $$
we have the following:
\begin{enumerate}
    \item For $p \equiv 3 \mod 8$, the vertices $v_{287496}$ and $w_{287496}$ corresponding to quadratic twists with $j$-invariant $287496$ are on the floor, so
    \begin{align*}
        \End_{\Fp} (E_{287496}) \cong \Z[\sqrt{-p}].
    \end{align*}
    
    \item For $p \equiv 7 \mod 8$, the vertices $v_{287496} $ and $w_{287496}$ corresponding to quadratic twists with $j$-invariant $287496$ are on the surface of the volcano, so
    \begin{align*}
        \End_{\Fp} (E_{287496}) \cong \Z\left[\frac{1+\sqrt{-p}}{2}\right].
    \end{align*}
\end{enumerate}

\subsection{Passing from the graph $\FpGraph$ to the spine $\FpSubGraph \subset \FpBarGraph$} \label{subsec:folding-stacking-attaching}

The subgraph $\FpSubGraph$ of $\FpBarGraph$ can be obtained from the graph $\FpGraph$ in the following two steps:

\begin{enumerate}
    \item Identify the vertices with the same $j$-invariant: these two vertices of $\FpGraph$ merge to a single vertex on $\FpBarGraph$. Identify equivalent edges.
    \item Add the edges from $\FpBarGraph$ between vertices in $\Fp$ corresponding to isogenies which are defined over $\Fpbar\setminus\Fp$.
\end{enumerate}

One notation we use to distinguish between vertices of $\FpGraph$ and those of $\FpBarGraph$: Vertices of components of $\FpGraph$ corresponding to the $j$-invariant $a$ will be denoted $v_a$ and $w_a$, where $v_a$ is a vertex on the connected component $V$ of $\FpGraph$ and $w_a$ lies on the component $W$ (not necessarily distinct from $V$). Since the $j$-invariants uniquely determine the vertices of $\FpBarGraph$, we will use $a$ to denote a vertex of $\FpBarGraph$. It is useful to think of the vertices $v_a, w_a$ as elliptic curves that are twists of each other. 

Remember that $\FpGraph$ is not a subgraph of $\FpSubGraph$ since both the vertices and edges of $\FpGraph$ may be merged in $\FpSubGraph$. Fortunately, something weaker is true, as we show in the following lemma. It turns out distinct edges from the same vertex in $\FpGraph$ correspond to distinct edges in $\FpBarGraph$.

\begin{lemma}[$\F_p$-edges are rigid.]
\label{lemma:fp_edges_are_rigid}
Let $E$ be an elliptic curve with $j(E) \notin \{0,1728\}$ defined over $\F_p$ (with $p \geq 5$). Suppose that there are two $\ell$-isogenies from $E$ defined over $\F_p$. Then they are equivalent over $\F_p$ if and only if they are equivalent over $\Fpbar$. 
\end{lemma}
\begin{proof}
If the isogenies are equivalent over $\F_q$, then they are equivalent over $\Fpbar$.

Let $\phi_1: E \to E_1$ and $\phi_2: E \to E_2$ be two isogenies that are defined over $\Fpbar$. We want to show that they are equivalent over $\F_p$. By hypothesis, there exists ($\Fpbar$-)isomorphisms $\varphi: E \to E$ and $\psi: E_1 \to E_2$ such that $\phi_2 = \psi \circ \phi_1 \circ \varphi$.

Consider the commuting square\begin{align*}
    \xymatrix{E \ar[r]^{\phi_1}  \ar[d]^{\varphi_1} & E' \ar[d]^{\varphi_2} \\
    E \ar[r]^{\phi_2}& E'
    }
\end{align*}
We know that the kernel of the map $\varphi_2 \circ \phi_1$ is $\ker \phi_1$. Therefore, the kernel of the map $\phi_2 \circ \varphi_1$ also is $\ker \phi_1$. This means that \begin{align*}
    \varphi_1 (\ker(\phi_1)) = \phi_2 .
\end{align*}
By the hypothesis on $j(E)$, $\mathrm{Aut}_{\Fp}(E) = \mathrm{Aut}_{\Fpbar}(E) = \{\pm 1\}$, so this is not possible as $[\pm 1] \ker \phi_1 = \ker \phi_1$. 
\end{proof}
We note that the proof above works if we replace $q$ with $p^n$ and consider isogenies and curves defined over $\F_q$, however, this will not be needed in our discussion.

Lemma \ref{lemma:fp_edges_are_rigid} for $\ell = 2$ gives the following corollary.
\begin{cor} If the neighbors of an elliptic curve $E$ with $j$-invariant $\jp$ in $\FpGraphtwo$ are elliptic curves with $j$-invariants $\jp_1, \jp_2$ and $\jp_3$, then the neighbors of 
$\jp$ in $\FpBarGraphtwo$ are $\jp_1, \jp_2$ and $\jp_3$.
\end{cor}
Since there are always at most $2$ neighbours of any vertex in $\FpGraph$ for $\ell > 2$, the above corollary does not generalize. However, it is still true that if there are neighbours of $v_a, w_a$ (recall that $v_a, w_a$ are the two vertices in $\FpGraph$ that have $j$-invariant $a$) that have $j$-invariants $b,c,d,e$, then there are (not necessarily distinct) edges $[a,b],[a,c],[a,d]$ and $[a,e]$ in $\FpBarGraph$.

Defined below are the four processes that can happen to the components of $\FpGraph$ when passing to $\FpSubGraph$. We will show that this list is exhaustive.

\begin{definition}[Stacking, folding and attaching]\label{stacking}
\leavevmode
\begin{enumerate}
    \item Let $V$ and $W$ be two distinct components of $\FpGraph$. We say that $V$ and $W$ \textbf{stack} if, when we relabel the vertices $v_a$ by the $j$-invariant $a$, they become isomorphic as graphs.
    \label{def:stacking}
    
    \item A connected component $V$ of $\FpGraph$ \textbf{folds} in $\FpBarGraph$ if $V$ contains vertices corresponding to both quadratic twists of every $j$-invariant on $V$. The term is meant to invoke what happens to this component when the quadratic twists are identified in $\FpBarGraph$.
    \label{def:folding}

    \item Two connected components $V$ and $W$ of $\FpGraph$ become \textbf{attached by a new edge} in $\FpBarGraph$ if there is a new edge $[a,b] \in \FpBarGraph$ corresponding to an isogeny between vertices $v_a \in V$ and $w_b \in W$ that is not defined over $\Fp$.
    \label{def:attaching}
    
    \item We say that two components $V$ and $W$ of $\FpGraph$ for $\ell > 2$ \textbf{attach along the $j$-invariant} $a$ if they both contain a vertex $v_a \in V, w_a \in W$ that corresponds to $j$-invariant $a$ and such that there is a neighbour $v_b$ of $v_a$ with $j$-invariant $b$ such that the twist of $w_b$ is not a neighbour of $w_a$ and vice versa.
     \label{def:attaching_along_j}
    \end{enumerate}

\end{definition}

An example of the first three of the four phenomena are given by Figure  \ref{fig:attachment} and 
an example of attachment along a $j$-invariant can be seen in Figure \ref{fig:attachment_vertex}. 

\begin{figure}[h!]
\begin{subfigure}{.45\textwidth}
  \centering
  \includegraphics[width=\linewidth]{431fp.png}
  \caption{The $\FpGraphtwo$ for $p = 431$}
\end{subfigure}
\begin{subfigure}{.45\textwidth}
  \centering
  \includegraphics[width=\linewidth]{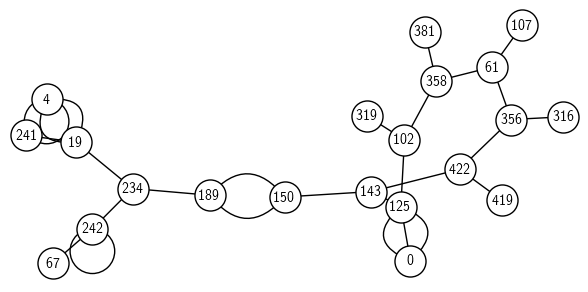}
  \caption{The spine $\FpSubGraph \subset \FpBarGraphtwo$ for $p = 431$.}
\end{subfigure}
\caption{Stacking, folding and attaching by an edge for $\p = 431$ and $\ell = 2$. The leftmost component of $\FpGraphtwo$ folds, the other two components stack, and the vertices $189$ and $150$ get attached by a double edge.}
\label{fig:attachment}
\end{figure}

\begin{figure}[h!]
\begin{subfigure}{.45\textwidth}
  \centering
  \includegraphics[width=\linewidth]{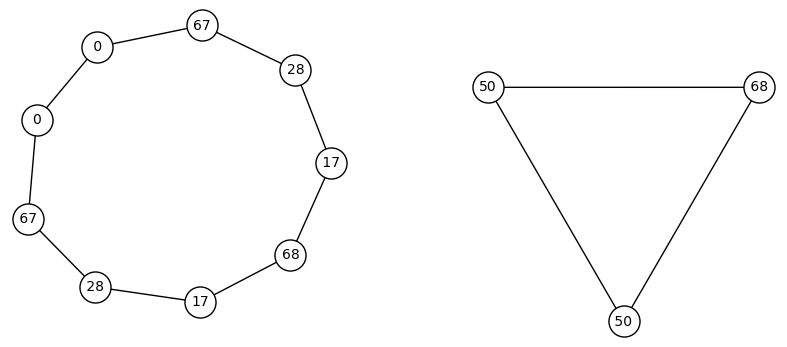}
  \caption{The $\FpGraphthree$ for $p = 83$}
\end{subfigure}
\begin{subfigure}{.45\textwidth}
  \centering
  \includegraphics[width=\linewidth]{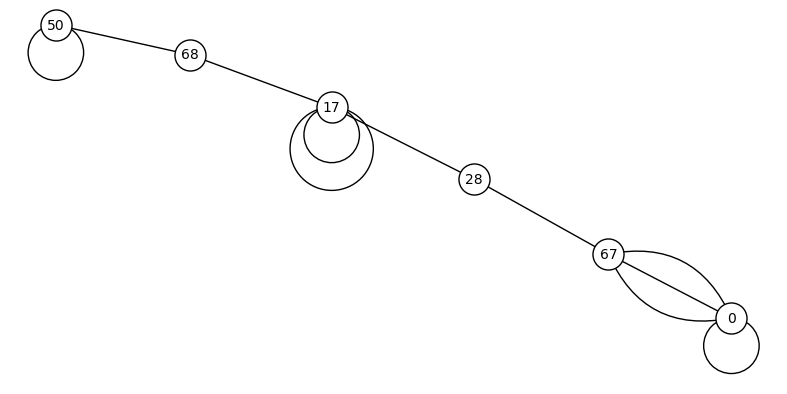}
  \caption{The spine $\FpSubGraph \subset \FpBarGraphthree$ for $p = 83$.}
\end{subfigure}
\caption{Attachment along a $j$-invariant for $p = 83$ and $\ell = 3$. We see that the two connected components of $\FpGraphthree$ are attached along the $j$-invariant $68 = 1728 \mod 83$.}
\label{fig:attachment_vertex}
\end{figure}

Note that it can happen that an attachment is actually attaching the component $V$ to itself. For instance, whenever there is only one component, new edges cannot attach distinct components. See Figure \ref{fig:attachment_same_component}.

\begin{figure}[h!]
\begin{subfigure}{.45\textwidth}
  \centering
  \includegraphics[width=\linewidth]{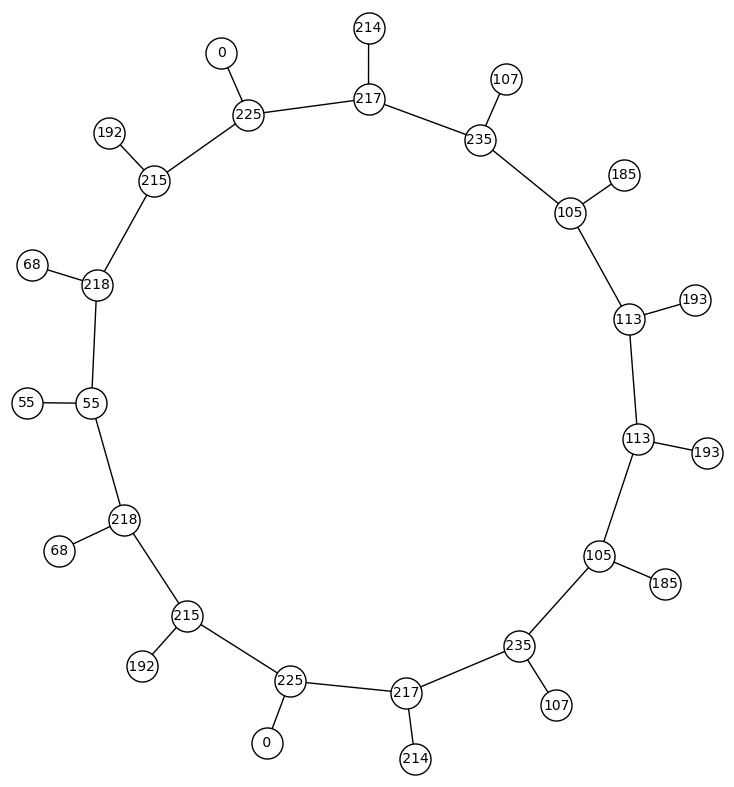}
  \caption{The $\FpGraphtwo$ for $p = 239$}
\end{subfigure}
\begin{subfigure}{.45\textwidth}
  \centering
  \includegraphics[width=\linewidth]{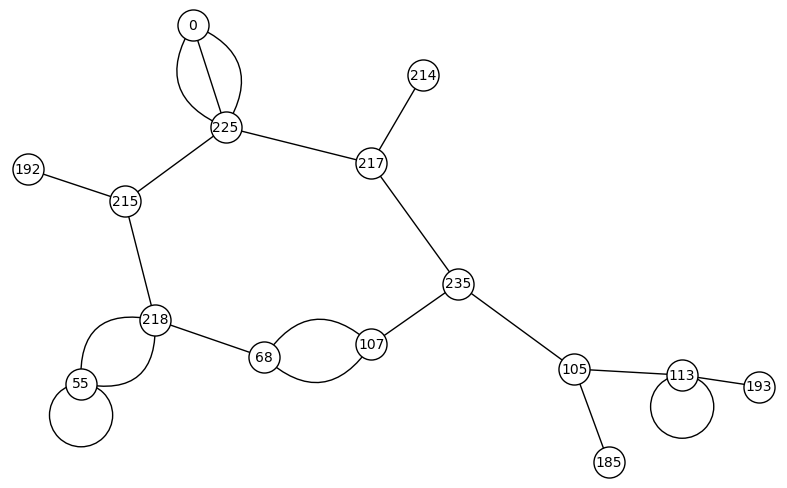}
  \caption{The spine $\FpSubGraph \subset \FpBarGraphtwo$ for $p = 239$.}
\end{subfigure}
\caption{Attachment by an edge that does not attach two distinct components. The vertices with $j$-invariants $68$ and $107$ are joined by a double edge.}
\label{fig:attachment_same_component}
\end{figure}

We now begin analyzing the new edges in $\FpSubGraph$ that do not come from edges in $\FpBarGraph$.  
For any elliptic curve $E$ and $\ell$ prime, $\ell$-isogenies are given by a cyclic subgroup of order $\ell$ of the $\ell$-torsion points $E[\ell]$. Such a subgroup is generated by a point of exact order $\ell$. An $\ell$-isogeny is defined over $\Fp$ if and only if its kernel is defined over $\Fp$ (that is, the kernel is fixed by the Frobenius morphism).

For $2$-isogenies, the kernel consists of the point at infinity, $O_E$ and a point $P \in E[2] \setminus \{ O_E\}$. The $2$-isogeny with kernel $\langle P \rangle$ is defined over $\Fp$ if and only if $P$ is defined over $\Fp$. 

For $\ell > 2$, the point $P$ generating $\ker \phi$ does not have to be defined over $\Fp$, only the whole kernel needs to be fixed by the Frobenius morphism.

\begin{remark}
      Let $E, E'$ be elliptic curves with $j$-invariants $j, j'$ and suppose that there is an edge $[j,j']$ in $\FpBarGraph$. Then there is an $\ell$-isogeny $\phi: E \rightarrow E'$. Even if both $j, j'$ are in $\Fp$, the isogeny $\phi$ is not necessarily defined over $\Fp$. We can see this in Figure \ref{fig:attachment}: the 2-isogenies between $150$ and $189$ in $\FpBarGraphtwo$ are not defined over $\Fp$. Also, the vertex $v_4$ ($4 \equiv 1728 \mod 431$) on the floor of $\FpGraphtwo$ has no edge to a curve with $j$-invariant $19$, but there is an edge $[4, 19] \in \FpBarGraphtwo$ coming from the two isogenies from the vertex $v_4$ on the surface. Moreover, Lemma \ref{lemma:fp_edges_are_rigid} gives us that there is a double edge $[4,19] \in \FpGraphtwo$. This not a coincidence, as we will explain Lemma \ref{lemma:one-two-isogeny}. 
\end{remark}

\begin{lemma}[One new isogeny implies two new isogenies] \label{lemma:one-two-isogeny}
Let $v_a, v_b \in \FpGraph$ correspond to $\Fp$-elliptic curves $E_a$ and $E_b$ with $j$-invariants $a,b$ respectively with $a \neq 1728, 0$. Assume that there is no edge $[v_a, v_b] \in \FpGraph$, but there is an edge $[a,b] \in \FpBarGraph$. Then there are two isogenies defined between $E_a \rightarrow E_b$ which are inequivalent over $\Fpbar$ and hence a double edge $[a,b] \in \FpBarGraph$.

\end{lemma}
\begin{proof}
We know that there is an $\ell$-isogeny $\phi: E_a \rightarrow E$ to some elliptic curve $E$ with $j$-invariant $b$. Since $j(E) = b$, then $E$ is isomorphic to $E_b$ over $\F_{p^2}$. Composing with this isomorphism, we obtain an $\ell$-isogeny $\psi: E_a \rightarrow E_b$. However, $\psi$ cannot be defined over $\Fp$, since we assumed there was no edge $[v_a, v_b] \in \FpGraph$. 

The kernel of $\psi$ is not defined over $\Fp$ (otherwise $\psi$ would be defined over $\Fp$), so the $p$-power Frobenius map $\text{Frob}:\Fp\to \Fp$ does not preserve $\ker \psi$. There is an isogeny from $E_a$ with kernel $\psi^\text{Frob}$. This isogeny has degree $\ell$ since $\psi^\text{Frob}$ has order $\ell$ and it is not equivalent to $\psi$. 
Using the construction of isogenies from V\'elu's formulae, we obtain the rational maps for defining $\psi^\text{Frob}$. In particular, the $j$-invariant of the target of $\psi^\text{Frob}$ is necessarily $\text{Frob}(b)=b$. Hence, there are two inequivalent isogenies between $E_a$ and $E_b$ and hence two edges $[a,b] \in \FpBarGraph$.

Note that we cannot simply compose with Frobenius, because that would give us an inseparable isogeny with degree $\ell \cdot p$. 
\end{proof}

The corollary below explains why for both attachment by a new edge (cf. Figure \ref{fig:attachment} and Figure \ref{fig:attachment_same_component179}) and attachment along a $j$-invariant (cf. Figure \ref{fig:attachment_vertex}), we always see double edges.

\begin{cor}\label{cor:attachment_means_double_edge}
Attachment of components from $v_a$ to $w_b$ forces a double edge $[a,b]$ in $\FpBarGraph$.
Attachment along the $j$-invariant $\jp$ implies a double edge from $\jp$ in $\FpBarGraph$.
\end{cor}
\begin{proof}
In the first case, we are adding an edge between $v_a$ and $w_b$ that is not defined over $\Fp$ and we can directly apply Lemma \ref{lemma:one-two-isogeny}. 

In the second case, we assume that there is a neighbour $v_b$ of $v_\jp$ such that $w_b$ is not a neighbour of $w_\jp$. Applying the Lemma \ref{lemma:one-two-isogeny} to the $\Fpbar$ isogeny from $w_\jp$ to $v_b$, we obtain a double edge $[\jp,b] \in \FpBarGraph$.
\end{proof}

The next step in understanding $\FpSubGraph \subset \FpBarGraph$ is understanding the neighbours of the two vertices $v_a, w_a$ corresponding to the same $j$-invariant. This is done in Lemma \ref{lemma:neighbors_ell_large} for $\ell > 2$ and in Lemma \ref{prop:neighbour_set} for $\ell = 2$. The case $\ell = 2$ is more involved because in this case, there exist vertical isogenies (if $p \equiv 3 \mod 4$), whereas for $\ell > 2$, all isogenies are horizontal.

\subsection{Stacking, folding and attaching for $\ell > 2$}
\label{sec:stacking_folding_attaching_large}

In this section, we consider the spine $\FpSubGraph$ for $\ell >2$. In this case, there are no vertical isogenies, hence the graph $\FpGraph$ is a union of disjoint cycles: The cycles of vertices corresponding to curves either only with endomorphism ring $\Z\left[\frac{1+\sqrt{-p}}{2} \right]$, or only with endomorphism ring $\Z[\sqrt{-p}]$. 

We will avoid on the case when the graph $\FpGraph$ is just a disjoint union of vertices (i.e., when there are no isogenies defined over $\Fp$). It suffices to assume that $p \equiv -1 \mod \ell$ (when $\ell| \# E(\Fp) = p+1$, there are $\Fp$-rational points of order $\ell$).

\begin{lemma}[The neighbour lemma]\label{lemma:neighbors_ell_large}
Suppose that $\ell >2$. Suppose that $v_a, w_a$ are the two vertices in $\FpGraph$ corresponding to elliptic curves with $j$-invariant $a$ and such that the neighbours of $v_a$ have $j$-invariants $b,c$ and the neighbours of $w_a$ have $j$-invariants $c,d$.

Then either $\{b,c\} = \{c,d\}$ with $b\neq c$ or there is a double edge from $a$ in $\FpBarGraph$.
\end{lemma}
\begin{proof}
Suppose that $d \neq b, c$. Since there is an edge in $[a,d] \in \FpBarGraph$ corresponding to the edge $[w_a, w_d] \in \FpGraph$, there is an isogeny from the elliptic curve $v_a$ to an elliptic curve with $j$-invariant $d$. This isogeny cannot be defined over $\Fp$ since the neighbours of $v_a$ in $\FpGraph$ have $j$-invariants $b,c$. This gives at least two edges $[a,d] \in \FpBarGraph$, by Lemma \ref{lemma:one-two-isogeny}.
\end{proof}

\begin{cor}
An attachment along a $j$-invariant $a$ implies a double edge from $a$ in $\FpBarGraph$.
\end{cor}
\begin{proof}
See Definition \ref{def:attaching_along_j}: At least one neighbour of $v_a$ is distinct from the neighbours of $w_a$.
\end{proof}
The main result of this section is the following result.

\begin{prop}[Stacking, folding and attaching for $\ell > 2$] \label{theorem:SFAA_for_ell_large}
While passing from $\FpGraph$ to $\FpSubGraph$, the only possible events are stacking, folding, and $n$ attachments by a new edge and $m$ attachments along a $j$-invariant with $m +2n \leq 2 \ell (2 \ell -1)$. 
\end{prop}
\begin{proof}
Suppose that $v_a$ is a vertex of $\FpBarGraph$ such that $a$ does not admit a double edge in $\FpGraph$. Then the neighbours of $v_a$ and $w_a$ (its twist) are the same by $\ref{lemma:neighbors_ell_large}$. The connected components of $v_a$ and $w_a$ look the same locally at $a$. 

Suppose further that the connected component $V \subset \FpGraph$ does not contain any vertex that admits a double edge. By Lemma \ref{lemma:neighbors_ell_large}, every vertex $v_a$ has the same neighbours as its twist $w_a$, so the component either folds (if $w_a \in V$) or stacks with the component $W$ of $w_a$, which is necessarily identical to $V$ when we replace the labels of the vertices by their $j$-invariants.

By Definition \ref{def:attaching}, attachment happens when we add an edge that cannot be defined over $\Fp$. By \ref{lemma:one-two-isogeny}, attachments necessarily imply double edge. Attachment along a $j$-invariant $a$ also implies that there is a double edge from $a$. However, double edges can only occur at $j$-invariants which are roots of 
 \begin{align*}
     \operatorname{Res}\left( \Phi_
\ell(X,Y), \frac{d}{dY}\Phi_\ell(X,Y); \,\, Y 
\right),
 \end{align*} which is a polynomial of degree bounded by $2\ell (2 \ell - 1)$ by Lemma \ref{lemma:double_edge_lemma}.
 Therefore, except at vertices corresponding to $j$-invariants that admit double edges, the components will either stack or fold. Even in components containing vertices that admit double edges, all other pairs of vertices corresponding to the same $j$-invariant will either stack onto each other or, if they share a neighbour, fold onto each other. See \ref{fig:attachment_vertex}.
 
 Finally, for attachment by an edge $[a,b] \in \FpBarGraph$, both endpoints admit a double edge in $\FpBarGraph$, hence both $a$ and $b$ are roots of $\Res_\ell(X)$. Since the degree of $\Res_\ell(X)$ is bounded by $2\ell \cdot (2\ell - 1)$, we obtain the bound.
\end{proof}
For any given $\ell$, we know the possible attachments: $\Res_\ell(X)$ is a product of Hilbert class polynomials, so having roots in $\Fp$ which give supersingular $j$-invariants is equivalent to satisfying certain congruence conditions on $p$. We can construct primes $p$ to avoid attachments.

Typically, the polynomial $\Res_\ell(X)$ will be smooth and have lots of repeated factors, so for any given choice of $\ell$, the bound in Proposition \ref{theorem:SFAA_for_ell_large} can be made more precise, which we will show in the following section.

\subsubsection{Example: stacking, folding and attaching for $\ell = 3$} \label{sec:stacking_folding_attaching_3}

In this section, we study the stacking, folding and attaching behaviour for $\ell = 3$. The case $\ell = 2$ will be discussed in Section \ref{subsec:stacking_folding_attaching_2}. The case $\ell = 3$ is cryptographically relevant, because of the use of $\FpBarGraphthree$ in SIDH and SIKE. Moreover, keeping $\ell$ small, we can give better bounds on the number of attachments and explain the results of the previous section in a more hands-on manner.

We start with factoring over $\Z$ the polynomial $\Res_3(x)$ introduced in \eqref{polynomial:double_edges}:
\begin{align*}
 \Res_3(x ) = &   \left(-1\right) \cdot 3^{3} \cdot x^{2} \cdot (x - 8000)^{2} \cdot (x - 1728)^{2} \cdot (x + 32768)^{2} \cdot (x^{2} - 52250000 x + 12167000000)^{2} \\
  &  \cdot (x^{2} - 1264000 x - 681472000)^{2} \cdot (x^{2} + 117964800 x - 134217728000)^{2}.
\end{align*} The irreducible factors are Hilbert class polynomials of discriminants $-3, -8, -4, -11, -32, -20 $ and $-35$, respectively. Removing the repeated factors, we see that there are at most $10$ vertices at which a double edge can occur. 

Double-edges also arise in loops (double self-3-isogenies), which accounts for some of the factors of $\Res_3(x)$. We find the self loops by factoring the modular 3-isogeny polynomial:
\begin{align*}
   \Phi_3(x,x) =  \left(-1\right) \cdot x \cdot (x - 54000) \cdot (x - 8000)^{2} \cdot (x + 32768)^{2}.
\end{align*}
At $\jp = 8000$ and $\jp = -32768$, there are two self-3-isogenies and no attachment at these vertices.

\begin{example}[Neighbours of the vertices with loops] \label{example:neighbours_loops_ell3}
In this example, we determine the $\FpGraphthree$ neighbours of $\jp = 0, 8000, 54000$ and $-32768$.
This is done by factoring $\Phi_3(\jp,x)$:
\begin{enumerate}
    \item $\jp =0$: $\Phi_3(0,x )= {\left(x + 12288000\right)}^{3} \cdot x$. From this we conclude, that there is an isogeny $v_0,w_0$ that is defined over $\Fp$ (the factor $x$ has multiplicity 1, indicating this is not a double-edge and thus cannot appear only over $\mathbb{F}_{p^2}$). Hence, the neighbours of $v_0$ are $w_0$ and $v_{-12288000}$ and the neighbours of $w_0$ are $w_{-12288000}$ and $v_0$. Moreover, the edges to $-12288000$ that are not defined over $\Fp$ will not be attaching edges.
    
    As an aside, we note that since all isogenies in $\FpGraph$ for $\ell > 2$ are horizontal and $\End_{\Fp}(E_0) \cong \Z[\sqrt{-p}]$, it follows that $\End_{\Fp}E_{-12288000} \cong \Z[\sqrt{-p}]$.
\begin{figure}[h!]
  \centering
  \includegraphics[width=120mm]{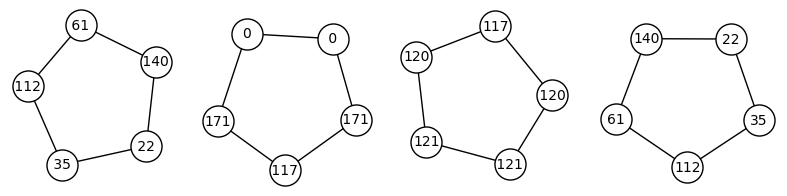}
\caption{The graph $\FpGraphthree$ for $p = 179$. We see that the neighbours of vertices with $j$-invariant $0$ both have $j$-invariant $-12288000 \mod 179 = 171$.}
\label{fig:attachment_same_component179}
\end{figure}
    
    \item $\jp = 54000$: There is one self-3-isogeny which arises from a 3-isogeny $\psi$ between (non-isomorphic) quadratic twists with $j = 54000$. Explicitly, let $E_{54000}:y^2 = x^{3} - 15 x + 22$, $E_{54000}':y^2 = x^{3} - 135 x - 594$. $\psi$ is given:
    \begin{align*}
        \psi = \left(\frac{x^{3} - 6 x^{2} + 33 x - 56}{x^{2} - 6 x + 9}, \frac{x^{3} y - 9 x^{2} y + 3 x y + 13 y}{x^{3} - 9 x^{2} + 27 x - 27}\right)
    \end{align*} that reduces modulo any prime $p$ to a $3$-isogeny over $\Fp$.

    Moreover, $\phi_3(54000,x)$ factors as \begin{align*}
        (x - 54000) \cdot (x^{3} - 151013228706000x^{2} + 224179462188000000x - 1879994705688000000000),
    \end{align*} so (for $p$ large enough) the $j$-invariant $54000$ cannot admit a double edge.
     
    \item $\jp = 8000$: factor $\Phi_3(8000, x) = {\left(x^{2} - 377674768000 \, x + 232381513792000000\right)} {\left(x - 8000\right)}^{2}$. There is a double loop $[8000,8000]\in \FpBarGraphthree$. Neither of the loops occur over $\Fp$: both cannot occur over $\Fp$ because there are no double edges in $\FpGraphthree$. If only one of them came from an isogeny over $\Fp$, we could use Lemma \ref{lemma:one-two-isogeny} to get a third loop, which is not possible (for large $p$).
    
    \item $\jp = -32768$: $\Phi_3(-32768,x) = (x^2 + 37616060956672\cdot x - 56171326053810176)\cdot (x + 32768)^2$. Repeating the argument we gave above for $\jp = 8000$, the self loops cannot come from isogenines over $\Fp$.
\end{enumerate}
To conclude: For $\jp =0, 54000$, the self-3-isogeny comes from an isogeny between the twists in $\FpGraphthree$, and for $\jp = 8000, -32768$, the double self-3-isogenies are not defined over $\Fp$.
\end{example}

The following lemma shows that we can distinguish attachment along a $j$-invariant and an attachment by a new edge looking at the neighbours of the given $j$-invariants.

\begin{lemma}[Attaching for $\ell = 3$] \label{lemma:attaching_ell_3} \begin{enumerate}
    \item Let $a$ be an attaching $j$-invariant. Then the neighbours of $v_a$ have the same $j$-invariant $b$ and induce a double edge $[a,b] \in \FpBarGraphthree$ and the neighbours of $w_a$ have the same $j$-invariant $c$ and induce a double edge $[a,c] \in \FpBarGraphthree$, with $b \neq c$.

    \item Let $[a,b]$ be an attaching edge in $\FpBarGraphthree$. Suppose that the neighbours of $v_a$ have $j$-invariants $c,d$. Then the neighbours of $w_a$ have $j$-invariants $c,d$. Necessarily $c \neq d$.
    \end{enumerate}
\end{lemma}

\begin{proof}
\begin{enumerate}
    \item This follows from Lemma \ref{lemma:one-two-isogeny}: Suppose the neighbouring vertices $w_b, w_c$ of $w_a$ have have $j$-invariants $b,c$. Suppose that $v_c$, the twist of $w_c$, is not a neighbour of $v_a$. Lemma \ref{lemma:one-two-isogeny} applied to the pair $v_a, w_c$ gives a double edge $[a,c] \in \FpGraphthree$.
    
    Similarly, there is a neighbour $v_d$ of $v_a$ with $j$-invariant $d$ such that $w_d$ is not a neighbour of $w_a$ and we obtain a double edge $[a,d]$ in $\FpGraphthree$. Since there are only $4$ edges from $a$ in $\FpGraphthree$ and since we assumed that at least one of the neighbours of $v_a$ had a different $j$-invariant than the neighbours of $w_a$ (and vice versa), we necessarily have that both $v_a$ and $w_a$ have two neighbours with the same $j$-invariant.
    
    \item Suppose that there is a new (double) edge $[a,b]$, not coming from an edge in $\FpGraphthree$. Let $v_a$ and $w_a$ be the twists corresponding to $\jp=a$. Let $v_c,v_d$ be the neighbours of $v_a$. The edges from $a$ in $\FpBarGraphthree$ are $[a,b], [a,b], [a,c]$ and $ [a,d]$. Since we assumed that the new edge does not come from the $\FpGraphthree$, the neighbours of $w_a$ cannot have $j$-invariant $b$ and are necessarily $w_c, w_d$.  \qedhere
\end{enumerate}
\end{proof}

\begin{cor}[Neighbours of twists] \label{cor:neighbours_twists_3}
For every $a \in \Fp$ that is not a root of  $\Res_3(x)$, the neighbours of $v_a$ and its twist $w_a$ have the same $j$-invariants $b,c$ with $b \neq c$.
\end{cor}
\begin{proof}
If at least one of the neighbours of $w_a$ had a different $j$-invariant than the neighbours of $v_a$, it would be an attaching $j$-invariant (every vertex in $\FpGraphthree$ has only two neighbours). The result follows from Lemma \ref{lemma:attaching_ell_3}.
\end{proof}

\begin{example}Let us work out the above lemmas for $p = 71$. $\FpGraphthree$ is given in Figure \ref{fig:71_3}. The supersingular $j$-invariants are $0, 17, 24, 40(\equiv54000\mod{71}), 41, 48 (\equiv 8000\mod{71}), 66$.
\begin{figure}[h!]
  \centering
  \includegraphics[width=100mm]{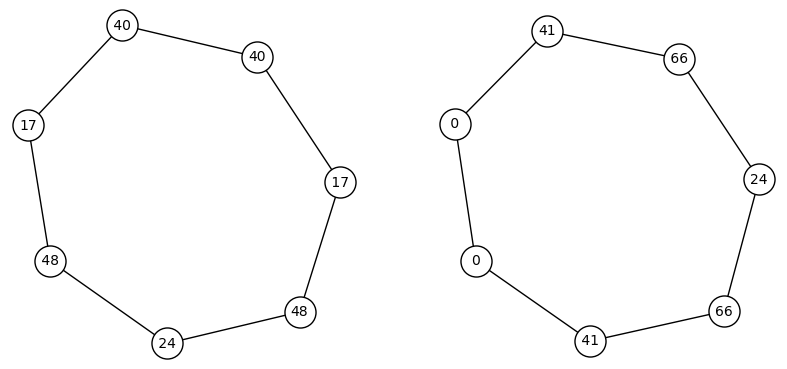}
  \caption{The $\FpGraphthree$ for $p = 71$}
\label{fig:71_3}
\end{figure}
\\The polynomial $\Res_3(x)$ factors in $\Fp$:
\begin{align*}
   \Res_3(x) = &  \left(44\right) \cdot x^{2} \cdot (x - 66)^{2} \cdot (x - 48)^{2} \cdot (x - 42)^{2} \cdot (x - 41)^{2}\\ & \cdot (x - 40)^{2} \cdot (x - 34)^{2} \cdot (x - 25)^{2} \cdot (x - 24)^{2} \cdot (x - 17)^{2}.
\end{align*} All the vertices of $\FpSubGraph$ are roots of $\Res_3(x)$, admitting a double edge. The ones that correspond to a double self-loop are the roots of
\begin{align*}
    \frac{\partial \Phi_3(x,x)}{\partial x} = \left(65\right) \cdot (x - 48) \cdot (x - 34) \cdot (x^{3} + 54 x^{2} + 54 x + 54).
\end{align*} Namely, there is a double self-loop at $48$ (because $34$ is not a supersingular $j$-invariant for $p= 71$). Finally, there are single loops at $0$ and $40$, as these are zeroes of $\Phi_3(x,x)$ with multiplicities 1. 

For the neighbours of $0$ and $24\equiv 1728 \mod 71$, we count the edges from these special vertices as one. Moreover, we see that only the edges $[40,66]$ and $[17, 41]$ are the attaching edges.
\begin{figure}[h!]
  \centering
  \includegraphics[width=85mm]{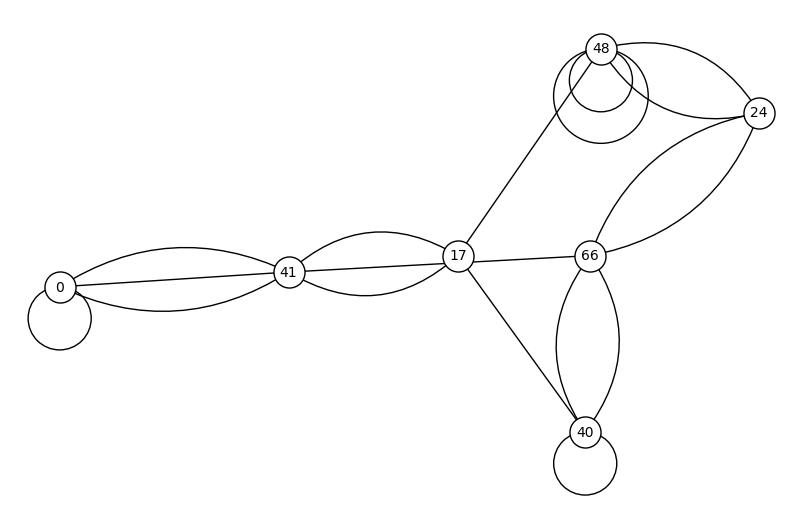}
  \caption{The $\FpSubGraph \subset \FpGraphthree $ for $p = 71$}

\label{fig:71_spine}
\end{figure}

\end{example}

The following theorem is a specialization of Proposition \ref{theorem:SFAA_for_ell_large}. We are mainly interested in the cryptographic applications, so we restrict to the case $p \equiv 3 \mod 4$. Then the class numbers $h(-p)$ and $h(-4p) = 3 \cdot h(-p)$ are both odd. With our assumption $p \equiv 2 \mod 3$, we have $p \equiv 11 \mod 12$.

\begin{theorem}[Stacking, folding and attaching for $\ell = 3$] \label{theorem:sfa_ell_3}
Let $p$ be a prime, $p \equiv 11 \mod 12$. When passing from $\FpGraphthree$ to the spine $\FpSubGraph \subset \FpBarGraphthree$, 
    \begin{enumerate}
        \item all components that do not contain $0$ or $54000$ stack,
        \item there are two distinct connected components $V$ and $W$ that contain a $j$-invariant $1728$, one of them contains both vertices with $j$-invariant $0$ and the other one both vertices with $j$-invariant $54000$. $V$ and $W$ fold and get attached at the $j$-invariant $1728$.
        \item At most $8$ vertices admit new edges, attaching at most $4$ pairs of components by a new edge.
\end{enumerate}
\end{theorem}

\begin{proof}
In Lemmas \ref{lemma:attaching_ell_3} and \ref{lemma:neighbors_ell_large}, we showed that $j$-invariants that attach by a new edge and $j$-invariants that do not admit a double edge look the same in the graph $\FpGraphthree$: that is, if the vertex $v_a$ has neighbours $v_b, v_c$, then the vertex $w_a$ has neighbours $w_b, w_c$. We do not need to treat these vertices with a separate case.

Suppose that there is a component $V$ that does not stack. This either means that there is a vertex $v_a$ whose neighbours are different than those of $w_a$, in this case $a$ is an attaching $j$-invariant and we will treat this case below.

Or, there is a $j$-invariant $a$ such that both the vertices $v_a, w_a$ are in the component $V$. We know that $V$ is a cycle. The vertices $v_a, w_a$ divide the cycle in two halves, choose either half $H$. Look at the neighbours of $v_a$ and $w_a$. If they have the same $j$-invariant $b$, replace $a$ with $b$ and continue moving along the halves of the cycle, until either of the following happens:

\begin{enumerate}[(i)]
    \item $v_a$ and $w_a$ are neighbours in $\FpGraphthree$ and hence induce a loop in $\FpGraph$. 

    \item The only neighbour of $v_a$ and $w_a$ is a vertex $v_\jp$ with $j$-invariant $\jp$. This is necessarily an attaching $j$-invariant because the neighbours of $w_\jp$ cannot have $j$-invariants $a$.

    \item The neighbour $v_b$ of $v_a$ has $j$-invariant $b $ and the neighbour $w_c$ of $w_a$ has $j$-invariant $w_c$, for $b \neq c$. \label{case:distinct_neighbours_cycle} 
    Then either the other neighbour of $w_a$ is $w_b$ or $a$ is an attaching $j$-invariant. Suppose $a$ is not an attaching $j$-invariant. Continuing along the whole cycle $V$ in the direction of the edge $[v_a, v_b]$, and symmetrically in the direction of $[w_a, w_b]$, we will reach 
     a point when the neighbour of some $v_{c}$ is not a neighbour of the $w_c$. This happens when class number is odd because we cannot get the same sequence in the half that has odd length and in the half that has even length. Here we also obtain an attaching $j$-invariant.
\end{enumerate}

We now discuss what happens with the components that contain an attaching $j$-invariant. 
The proof is a similar argument to the one above. Starting at any attaching $j$-invariant $\jp\in V$ (there could be multiple), we know that its neighbours $v_a, w_a$ have the same $j$-invariant by \ref{lemma:attaching_ell_3}. By walking away from $\jp$, we will at some point reach either \begin{enumerate}[(i)]
    \item a pair of vertices $v_b, w_b$ that are connected and induce a loop in $\FpBarGraphthree$. The component then folds.
    
    \item A pair of vertices $v_b, w_b$ such that the neighbour in the direction away from $\jp$ of $v_b$ is $v_c$ and of $w_b$ is $w_d$ for some $c \neq d$. But then $b$ is an attaching $j$-invariant and hence the neighbours of $v_b$ and $w_b$ have different $j$-invariants by Lemma \ref{lemma:attaching_ell_3}. But we assumed that they come from the chain from $\jp$ and so the neighbours of $v_b$ and $w_b$ in the direction of $\jp$ have the same $j$-invariant.
    This is a contradiction.

    \item A single vertex $v_b$ from `both sides'. But since the class number is odd, this gives us a contradiction. 
    \end{enumerate}
The above then shows that any component $V$ of $\FpGraphthree$ that contains an attaching $j$-invariant $\jp$ contains precisely one attaching $j$-invariant, folds and the `opposite vertices' (the vertices $v_b,w_b$ that are the furthest away from $\jp$) are connected by an $\Fp$-isogeny, hence inducing a loop in $\FpGraphthree$. 

Example \ref{example:neighbours_loops_ell3} showed that the only possible opposite vertices are $j$-invariants $0$ and $54000$. For $p \equiv 3 \mod 4$, there are two components containing $1728$: By Section \ref{subsubsec:Neighboursof1728}, one of the vertices corresponding to $1728$ is on the floor and the other one is on the surface, so they are on different components of $\FpGraphthree$. One of these vertices is on the same component of $\FpGraphthree$ as the vertices with $j$-invariant $0$ and the other one will contain both vertices with $j$-invariant $54000$.
\end{proof}

\begin{remark}
\begin{enumerate}
    \item The proof above shows that \begin{align*}
        \End_{\Fp}(E_{54000}) = \Z[\sqrt{-p}]
    \end{align*} whenever this $j$-invariant is supersingular.
    \item The proof above holds for any $p$ such that the order of the prime above $3$ in $\Cl(\OO_K)$ is odd, which is necessarily the case for $p \equiv 3 \mod 4$.
    
    \item It is possible to extend the proof for primes $p$ such that the order of the prime above $3$ in $\Cl(\OO_K)$ is even, however, one needs to consider the case of cyclic graphs like Figure \ref{fig:cycle} and correspond to case 3 in the proof of Theorem \ref{theorem:sfa_ell_3}.

    It should be possible to argue that the two distinct paths from $v_a$ to $w_a$ cannot collapse onto one loop if one adapts the proof of Lemma \ref{lemma:fp_edges_are_rigid} because a composition of cyclic isogenies with no backtracking will again be a cyclic isogeny.

    \begin{figure}[h!]
  \centering
  \includegraphics[width=45mm]{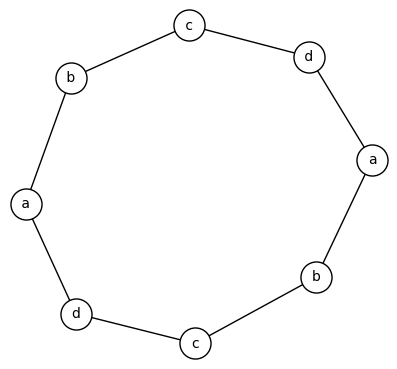}
\caption{In case \ref{case:distinct_neighbours_cycle} for primes $p$ such that the prime above $3$ has even order, one needs to disprove the situation depicted above.}
\label{fig:cycle}
\end{figure}
\end{enumerate}
\end{remark}

\subsection{Stacking, folding and attaching for $\ell = 2$} \label{subsec:stacking_folding_attaching_2}

We identify how the components of $\FpGraphtwo$ come together in $\FpSubGraph\subset\FpBarGraphtwo$. Vertical $2$-isogenies are possible, in contrast to the $\ell = 3$ case from Section \ref{sec:stacking_folding_attaching_3}. 
\\The main theorem of this section is the following:

\begin{theorem}[Stacking, folding and attaching] \label{theorem:stacking-folding-attaching} \label{sfatwo}
Only stacking, folding or at most $1$ attachment by a new edge are possible. In particular, no attachments by a $j$-invariant are possible.
\end{theorem}

Recall that we form the graph $\FpSubGraph \subset \FpBarGraphtwo$ from $\FpGraphtwo$ in two steps: identify vertices corresponding to the same $j$-invariant and identify the edges, then add new edges.

We will show that:
\begin{enumerate}
    \item the neighbours of the two vertices that correspond to twists have the same $j$-invariants (Proposition \ref{prop:neighbour_set}) and this will impliy that only stacking and folding is possible.

    \item At most one component folds, and for $p \equiv 3 \mod 4$ this is the component containing $\jp = 1728$ (Proposition \ref{prop:1728_folding}). 
    
    \item Attaching of components by a new edge happens between at most one pair of vertices, and those vertices are roots of the Hilbert class polynomial of $\Q(\sqrt{-15})$ (Proposition \ref{prop:possible_attachments}).
\end{enumerate}

We begin with some results on the neighbours of vertices corresponding to the same $j$-invariants. From Corollary \ref{lemma:quadratic_twists}, we know that (except for $1728$), twists have isomorphic endomorphism rings and hence lie on the same level in the volcano. More is true:

\begin{prop} \label{prop:neighbour_set}
Let $\jp$ be a supersingular $j$-invariant and let $v_\jp$ and $w_\jp$ be two distinct vertices in $\FpGraphtwo$ corresponding to elliptic curves with $j$-invariant $\jp$. If $j \neq 1728$, then the  two vertices corresponding to the same $j$-invariants have the same neighbours, that is:
\begin{enumerate}
    \item if $p \equiv 1 \mod 4$ and the neighbour of $v_\jp$ is $v_{\jp'}$ then the neighbour of $w_\jp$ is $w_{\jp'}$,
    \item if $p \equiv 3 \mod 4$ and if 
    \begin{enumerate}
        \item the vertices $v_\jp$ and $w_\jp$ are both on the floor, and $v_\jp$ and $w_\jp$ are each attached to a vertex with $j$-invariant $\jp'$,
    
        \item the vertices $v_\jp$ and $w_\jp$ are both on the surface. $v_\jp$ has three neighbours with distinct $j$-invariants $a,b,c$ and $w_\jp$ has three neighbours with the same distinct $j$-invariants $a,b,c$.
\end{enumerate}
\end{enumerate}
\end{prop}

The neighbours of $j = 1728$ are given in Section \ref{subsubsec:Neighboursof1728}.

\begin{proof}[Proof of Proposition \ref{prop:neighbour_set}.]
\begin{enumerate}
    \item For $p \equiv 1 \mod 4$, any connected component of $\FpGraphtwo$ is an edge. If $v_\jp$ and $w_\jp$ are on the same connected component, the result follows immediately.

If $v_\jp$ and $w_\jp$ are on different connected components, denote the neighbour of $v_\jp$ as $v_a$ and the neighbour of $w_\jp$ is $w_b$. If $a = b$, the result holds. If $a \neq b$, Lemma \ref{lemma:one-two-isogeny} applied to the pair $v_\jp, w_b$ gives a double edge $[\jp, a]$. Similarly, there is a double edge $[\jp, b]$ in $\FpBarGraphtwo$. There are only $3$ edges from $\jp$ in $\FpBarGraphtwo$, so we obtain a contradiction. 

\item Suppose now that $p \equiv 3 \mod 4$. By Corollary \ref{cor:surface_floor_is_true}, either both $v_\jp,w_\jp$ lie on the surface or they both lie on the floor of their respective components. Considering these two cases:
\begin{enumerate}
    \item Case 1: \label{case:surface_surface} The vertex $v_\jp$ is on the surface of $\FpGraphtwo$ component $V$ and $w_\jp$ is on the surface of component $W$. Since $v_\jp$ is on the surface of $V$, it has three (not necessarily distinct) neighbours $v_\mathbf{a}, v_\mathbf{b}$ and $v_\mathbf{c}$. 
  
    By Lemma \ref{lemma:fp_edges_are_rigid}, the three neighbors of $v_\jp$ in $V$ give the three neighbors of $v_\jp$ in $\FpBarGraphtwo$: Any neighbor of $w_\jp$ in $W$ has to be one of $w_\mathbf{a},w_\mathbf{b}$ or $w_\mathbf{c}.$ Any set of neighbors of $v_\jp$ in $V$ (counted with multiplicity) is a subset of the neighbors of $w_\jp$. Since $v_a$ and $w_a$ are both floor vertices and $a \neq b,c$, the vertices corresponding to $b$ and $c$ are on the surface. Suppose $b= c$. Since there are only two vertices with $j$-invariant $b$, $w_\jp $ is attached to the same two $j$-invariants $v_b, w_b$ as $v_\jp$ is. Then we see a cycle on the surface of length $4$, and this is a contradiction since for $p \equiv 3 \mod 4$, the class number $h(-p)$ is odd. Hence $v_\jp$ and $w_\jp$ have the same set of neighbors and those neighbours are all distinct.
    
    \item Case 2: $v_\jp$ and $w_\jp$ are on the floors of their respective $\FpGraphtwo$ components, $V$ and $W$. 
    
    Let $v_a$ denote the neighbour of $v_\jp$, where $v_a$ lies on the surface. Let $w_b$ denote the surface neighbor of $w_\jp$. Suppose $a \neq b$: We will show this leads to a contradiction. Lemma \ref{lemma:one-two-isogeny} applied to the pair $v_\jp, w_b$ gives a double edge $[\jp, b]$ in $\FpBarGraphtwo$. Similarly, we obtain a double edge $[\jp, a]$ in $\FpBarGraphtwo$ as well. This would mean that there are four inequivalent edges from $\jp$ in the graph $\FpBarGraphtwo$, which is not possible so $a = b$.

    \qedhere

\end{enumerate}

\end{enumerate}
\end{proof}

\begin{cor}[Isogenies for twists]
Let $\phi  : E \rightarrow E'$ be an $\Fp$-isogeny of degree $2$, $j(E),j(E')\neq 1728$. Then, there is an $\Fp$-isogeny of degree 2 between the quadratic twists $E^{t} \rightarrow (E')^{t}$.
\end{cor}
\begin{proof}
Suppose $\phi: E \rightarrow E'$ as in the statement, with $j(E) = \mathbf{a}, j (E') = \mathbf{b}$ and $E$ corresponds to the vertex $v_\mathbf{a}$. $\phi$ corresponds to an edge $[v_\mathbf{a}, v_\mathbf{b}] \in \FpGraphtwo$. Let $w_\mathbf{a}$ be the vertex in $\FpGraphtwo$ corresponding to the quadratic twist $E^t$. Proposition \ref{prop:neighbour_set} gives a neighbour $w_\mathbf{b}$ of $w_\mathbf{a}$ such that $[w_\mathbf{a}, w_\mathbf{b}] \in \FpGraphtwo$.

If $w_\mathbf{b}$ corresponds to the twist $(E')^t$, then the edge $[w_\mathbf{a}, w_\mathbf{b}]$ gives the desired $\Fp$-isogeny.

If, instead, $w_\mathbf{b} = v_\mathbf{b}$, there are two edges $[v_\mathbf{a}, v_\mathbf{b}], [w_\mathbf{a}, v_\mathbf{b}] \in \FpGraphtwo$. Suppose $z_\mathbf{b}$ is the vertex of $\FpGraphtwo$ corresponding to $(E')^t$. Since we assumed $\mathbf{b} \neq 1728$, Proposition \ref{prop:neighbour_set} gives that $z_\mathbf{b}$ also has two neighbours with $j$-invariants $\mathbf{a}$. This means there must be edges $[z_\mathbf{b}, v_\mathbf{a}]$ and $[z_\mathbf{b}, w_\mathbf{a}]$ in $\FpGraphtwo$, and $[z_\mathbf{b}, w_\mathbf{a}]$ gives the desired $\Fp$-isogeny $E^t \rightarrow (E')^{t}$.
\end{proof}

\begin{cor}[Attachment along a $j$-invariant for $\ell = 2$]  \label{cor:attach_j_invariant_2}
Attachment along a $j$-invariant cannot happen for $\ell = 2$.
\end{cor}
\begin{proof}
Proposition \ref{prop:neighbour_set} shows that, except at $1728$, the neighbours of the twists are exactly the same. Attachment along a $j$-invariant (Definition \ref{def:attaching_along_j}) only happens when at least one of the neighbours is distinct.

At $\jp = 1728$, we saw in \ref{subsubsec:Neighboursof1728} that the twists are connected by a $2$-isogeny in $\FpGraphtwo$. 
\end{proof}

By a new edge in $\FpBarGraphtwo$ we mean an edge that does not come from an edge in $\FpGraphtwo$.

\begin{prop}[Possible new edges and attachments]
\label{prop:possible_attachments}
A new edge in $\FpBarGraphtwo$ between $\Fp$ $\jp$-invariants can only be added between vertices whose $j$-invariants correspond to the roots of 
\[f(X) = X^2 + 191025X - 121287375\]
in $\Fp$, provided these are supersingular $\Fp$ $j$-invariants not equal to $-3375, 1728$ or $0$.

Attachment cannot happen at $\jp = 0, 1728$ or $-3375.$
\end{prop}

\begin{proof}
Let $v_\mathbf{a},w_\mathbf{b}\in\FpGraphtwo$ correspond to $j$-invariants $\mathbf{a},\mathbf{b}$, respectively, such that there is no edge in $\FpGraphtwo$ between $v_a$ and $w_b$, but there is an edge $[a,b]$ in $\FpBarGraphtwo$. By Lemma \ref{lemma:one-two-isogeny}, we obtain a two inequivalent edges $[\mathbf{a},\mathbf{b}],[\mathbf{a},\mathbf{b}] \in \FpBarGraphtwo$. By Lemma \ref{lemma:double_edge_lemma}, $\mathbf{a},\mathbf{b}$ must both be one of $0, 1728, -3375$ or an $\Fp$ root of $f(X) = X^2 + 191025X - 121287375$. However, no new edges can occur at the $j$-invariants $0, 1728$ and $-3375$ (see the discussion after the proof of Lemma \ref{lemma:double_edge_lemma}):
\begin{enumerate}
    \item For $\jp=0$ is already connected to its only neighbor $\jp=54000$ in $\FpGraphtwo$, as there are no isolated points in $\FpGraphtwo$.
    
    \item For $\jp= 1728$,
    all $2$-isogenies are defined over $\Fp$.

     \item For $\jp = -3375$, there are always two self-loops. Attachment is not possible, as it would require two additional inequivalent outgoing $2$-isogenies, giving $4$ edges at $-3375$ in $\FpBarGraphtwo$. \qedhere
\end{enumerate}

\end{proof}

\begin{prop}[Folding happens for the component containing $\jp = 1728$]
\label{prop:1728_folding} Let $p \equiv 3 \mod 4$ be prime. The connected component $V \in \FpGraphtwo$ containing the vertices corresponding to $\jp = 1728$ is symmetric over a reflection passing through the vertices $v_{1728}$ lying on the surface of $V$ and $w_{1728}$ lying on the floor of $V$.
In particular, the component $V$ folds when we pass from $\FpGraphtwo$ to $\FpSubGraph$.

\end{prop}
To understand this symmetry, picture the surface of the component $V$ as a perfect circle with equidistant vertices and all the edges to the floor are perpendicular to the surface. Then $V$ is symmetric with respect to the line extending the edge $[v_{1728}, w_{1728}]$.

This symmetry is mentioned in Remark 5 of \cite{CSIDH}, albeit without proof or reference. 

\begin{proof} 
For $p \equiv 3 \mod 4$, the possible shapes of the component $V$ are described in Section \ref{sec:p3mod4}.

\begin{enumerate}
    \item Case $p\equiv 3 \mod 8$. $V$ is a claw (see Figure \ref{Fig:Claws}) and the proof of Proposition \ref{prop:possible_attachments} shows that there is one $2$-isogeny down from the surface vertex corresponding to $\jp=1728$ to each vertex with $j$-invariant $287496$ and the other vertex corresponding $j$-invariant $1728$. The claw $V$ is clearly symmetric and folds as described. 
    
    \item Case $p\equiv 7 \mod 8$. In this case, $h(-p)$ is odd.

    We may assume that $h(-p) > 1$ (otherwise we are in the claw situation discussed above).

    $v = v_{287496}$ and $w = v_{287496}^{t}$ are both on the surface. By Proposition \ref{prop:neighbour_set}, their neighbours have the same $j$-invariants, say: $1728, \mathbf{a},\mathbf{b}$. Say the neighbours of $v$ are $v_\mathbf{a}, v_\mathbf{b}$ and the neighbours of $w$ are $w_\mathbf{a},w_\mathbf{b}$. Assume that $v_\mathbf{a}$ is on the floor. Since $\mathbf{a} \neq 1728$, Lemma \ref{cor:surface_floor_is_true} tells us $w_\mathbf{a}$ is also on the floor. Thus, both $v_\mathbf{b}$ and $w_\mathbf{b}$ are on the surface and the symmetry is preserved.

\[
\begin{tikzpicture}[node distance=1cm,auto]
    \draw (2,0) ellipse (1);
    \node [label={[shift={(0,-0.1)}]$1728$}] (v) at (2,-1)  {$\bullet$};
    \node[label={[shift={(0,-0.7)}]$1728$}] (w) at (2,-2)  {$\bullet$};
    \node[label={[shift={(-0.6,-0.5)}]$287496$}] (w1) at (1.29289321881345,-0.707106781186547)  {$\bullet$};
    \node[label={[shift={(0.6,-0.5)}]$287496$}] (v1) at (2.7,-0.707106781186547)  {$\bullet$};
    \node[label={[shift={(-0.3,-0.5)}]$w_{\mathbf{a}}$}] (w2) at (1.29289321881345,-1.707106781186547)  {$\bullet$};
    \node[label={[shift={(0.3,-0.5)}]$v_{\mathbf{a}}$}] (v2) at (2.7,-1.707106781186547)  {$\bullet$};
    \node at (1, 0){$\bullet$};
    \node at (3, 0){$\bullet$};
    \node[left] at (1, 0){$w_{\mathbf{b}}$};
    \node[right] at (3, 0){$v_{\mathbf{b}}$};
    \draw (2,-1) -- (2,-2);
    \draw (1.29289321881345,-0.707106781186547) -- (1.29289321881345,-1.707106781186547);
    \draw (2.7,-0.707106781186547) -- (2.7,-1.707106781186547);
\end{tikzpicture}
\]

    Continuing in this manner, because $h(-p)$ is odd, we will arrive at a pair of vertices $v_\mathbf{c}, w_\mathbf{c}$ that share an edge, accounting for all of the vertices in the component. The symmetry holds. \qedhere
\end{enumerate}

\end{proof}

\begin{remark}
Proposition \ref{prop:1728_folding} shows, for $p \equiv 7 \mod 8$, the $2$-isogeny between the pair of vertices $v_\mathbf{c},w_\mathbf{c}$ corresponding to the same $\jp$-invariant $\mathbf{c}$ at the end of the process will be precisely one loop at $\mathbf{c}$ in $\FpBarGraphtwo$. The only vertices with precisely one self-isogeny in $\FpBarGraphtwo$ are $\jp = 8000$ and $\jp=1728$. Since $v_\mathbf{c}, w_\mathbf{c}$ are on the surface of $\FpGraphtwo$, $\mathbf{c} = 8000$ (see Section \ref{subsubsec:SelfIsogenies}). There is an $\Fp$-rational $2$-power isogeny between any two supersingular elliptic curves with $\jp$-invariants $1728$ and $8000$.
\end{remark}

\begin{cor}[Folding] \label{cor:folding_only_once}
Suppose $V \subset \FpGraphtwo$ is a component which folds in $\FpSubGraph\subset\FpBarGraphtwo$. \begin{enumerate}
    \item If $p \equiv 1 \mod 4$, then $V$ is a single edge between two vertices with $j$-invariant $8000.$
    \item If $p \equiv 3 \mod 4$, then $V$ contains both the vertices corresponding to $\jp = 1728$.
\end{enumerate} 
\end{cor}
\begin{proof}
 \begin{enumerate}
     \item  If $p \equiv 1 \mod 4$, then $V$ is an edge: $[v_\mathbf{a}, v_\mathbf{b}]$. Folding happens if and only if $\mathbf{a}= \mathbf{b}$, resulting in a self-$2$-isogeny in $\FpBarGraph$. For $p \equiv 1 \mod 4$, the only vertices with self-$2$-isogenies are $\jp = -3375,8000$, when these $j$-invariants are supersingular (see Section \ref{subsubsec:SelfIsogenies}). 
     
     For $j = 8000$, there is a $2$-isogeny from the curve with $j$-invariant $8000$ given by the equation $E: y^2 = x^{3} - 4320 x - 96768 $ to its twist $y^2 = x^{3} - 17280 x - 774144 $. The latter is a twist of $E$ by $2$, and $8000$ is only supersingular for $p \equiv 5 \mod 8$, so $2$ is a nonsquare modulo $p$. 
     
     For $j=-3375$, there are two self-loops in $\FpBarGraphtwo$, and at least one of them not defined over $\Fp$. Applying Lemma \ref{lemma:one-two-isogeny} to this loop, we conclude that neither of these loops are defined over $\Fp$ and folding does not happen for the edge containing $-3375$.
\item 
If $p \equiv 3 \mod 4$, let $V$ be a component that folds. The surface has  $h(-p)$ vertices and this class number is odd. We assume that $V$ folds, so every vertex in it gets identified with the vertex corresponding to its twist. By Corollary  \ref{cor:surface_floor_is_true}, for $j \neq 1728$, the two vertices are either both on the surface or both on the floor. Since there are odd number of vertices on the surface, there cannot only be pairs of twists on the surface, so $V$ must contain the two vertices corresponding to $\jp = 1728$, one on the floor and the other on the surface. \qedhere

 \end{enumerate}
\end{proof}

Now, we prove Theorem \ref{theorem:stacking-folding-attaching}.

\begin{proof}[Proof of Theorem \ref{theorem:stacking-folding-attaching}]
Recalling the possible events when passing from $\FpGraphtwo$  to $\FpSubGraph\subset \FpBarGraphtwo$. We  identify the vertices with the same $j$-invariants, causing:

\begin{enumerate}
    \item Folding: Vertices corresponding to twists of the same $j$-invariant lie on the same component and get identified when we pass to $\FpBarGraphtwo$. \label{case:folding}
    \item Stacking: two isomorphic volcanoes (not just as graphs, but with vertices corresponding to the same $j$-invariants) have the twist vertices identified. 
    \item Attaching along a $j$-invariant: Corollary \ref{cor:attach_j_invariant_2} shows this is not possible. \end{enumerate}

First, let $p \equiv 1 \mod 4$. The components of $\FpGraphtwo$ are edges. Corollary \ref{cor:folding_only_once} shows that the edge containing the two vertices with $j$-invariant $8000$ folds (if $8000$ is a supersingular $j$-invariant for $p$, i.e. $p \equiv 5 \mod 8$). For the other edges, Proposition \ref{prop:neighbour_set} says that for any edge $[v_a,v_b] \in \FpGraphtwo$ the twists $w_a, w_b$ also give an edge $[w_a, w_b] \in \FpGraphtwo$. Moreover, Proposition \ref{prop:possible_attachments} gives that there is at most $1$ attachment among these edges. 

For $p \equiv 3 \mod 4$, take any component $V$ of $\FpGraphtwo$ and any vertex $v_a$ on the surface of $V$, $a\neq1728$. Choose a neighbour $v_b$ of $v_a$. Continue along the surface in the direction of the edge $[v_a,v_b]$ and consider the sequence $j$-invariants of neighbours $\mathcal{V} = \{v_i\}$ until we reach a vertex with $j$-invariant $a$. Similarly, on the component $W$ containing the edge $w_a, w_b$, consider the sequence of $j$-invariants of the neighbours $\mathcal{W} = \{w_i\}$ until we reach a vertex with $j$-invariant $a$ (every surface is a cycle, so this will happen in finitely many steps). We have the following possible outcomes:
\begin{enumerate}
    \item For some $i$, we find that $v_i \neq w_i$. This means that the curve $i$ away from $v_a$ on $V$ has a different neighbour than its twist, which is $i$ away from $w_a$. But this can only happen for $w_i = 1728$ and hence the component folds by Proposition \ref{prop:1728_folding}. 
    \item The sequences are equal, but $\mathcal{V}$ stops at the twist $w_a$ and $\mathcal{W}$ stopped at the curve $v_a$. Then $v_a, w_a$ are on the same component $V$ and the cycle on the surface has length $2 \cdot \text{length}(\mathcal{V})$. As $h(-p)$ is odd, this is not possible.
    \item The sequences $\mathcal{V}$ and $\mathcal{W}$ are the same and the components $V$ and $W$ are isomorphic as graphs upon replacing labels of vertices by their $j$-invariants. In this case, the components $V$ and $W$ stack.
\end{enumerate}
Finally, in Proposition \ref{prop:possible_attachments}, we showed that at most one attachment is possible. 
\end{proof}

Finally, we study the possible attachments given by the roots of the polynomial $f(X) =   X^2 + 191025X - 121287375$. Because the polynomial $f(X)$ is the Hilbert class polynomial of $\Q(\sqrt{-15})$, roots of $f(X)$ in $\Fpbar$ give supersingular $j$-invariant if and only if $p$ is inert in $\Q(\sqrt{-15})$. By factoring the discriminant of $f(X)$ 
\begin{align*}
    191025^2 +4 \cdot 121287375 = 36975700125 = 3^{6} \cdot 5^{3} \cdot 7^{4} \cdot 13^{2},
\end{align*} we see that there is a root in $\Fp$
 if and only if $p \equiv \pm 1 \mod 5$. Combining with the congruence condition that $\left( \frac{-15}{p} \right)= -1$, we obtain that there the roots of $f(X)$ are  $j$-invariants of a supersingular elliptic curves defined over $\Fp$:
\begin{enumerate}
    \item $p \equiv 1 \mod 4$ and $p \equiv 1 \mod 3$ and $p \equiv \pm 1 \mod 5$ $\longrightarrow$ $p \equiv 1, 24 \mod 60$
    \item $p \equiv 3 \mod 4$ and $p \equiv 2 \mod 3$ and $p \equiv \pm 1 \mod 5$ $\longrightarrow$ $p \equiv 11, 59 \mod 60$
\end{enumerate}

We have an additional result about when attachment occurs, as a corollary to Proposition \ref{prop:possible_attachments}:

\begin{cor}[Attachment happens for $p \not \equiv 7 \mod 8$.] Suppose that $p \not\equiv 7 \mod 8$ and suppose that $\jp$ and $\jp'$ are two distinct $\Fp$-roots of $f(X) = X^2 + 191025X - 121287375$ (it suffices to assume $p >101$). Then, the new edge $[\jp,\jp' ]\in \FpBarGraph$ is an attaching edge.
\end{cor}
Rephrased, this means that attachment happens whenever it can happen (i.e., when the roots of $f(X)$ are in $\mathbb{F}_p$) for $p \not\equiv 7\mod 8$.

\begin{proof}
First, let $p\equiv 1\pmod{4}$. The $\FpGraphtwo$ components are horizontal edges. Suppose that the $j$-invariant $\jp$ admits a double edge $[\jp,\jp'] \in \FpBarGraph$ that it is not an attaching edge, i.e., there is an edge $[v_\jp,v_{\jp'}]$ in $\FpGraphtwo$. By Lemma \ref{lemma:one-two-isogeny}, there is then a triple edge $[\jp, \jp]$. 
This is only possible if $\jp =0$. For $\jp$ or $\jp'$ to be equal to 0, we would need $X$ to be a factor of $f(X)$. Since $121287375 = 3^{6} \cdot 5^{3} \cdot 11^{3}$, for $p > 11$ attachment happens whenever it can.

Next, let $p\equiv 3\pmod{4}$. The components of $\FpGraphtwo$ are claws.
If the double edge is not between two different components, then $v_\jp$ and $v_{\jp'}$ are on the same claw (for some choice of the twists). Assume, $\jp \neq 1728$, they both lie on the floor.

Let $v_\mathbf{a}$ be the unique surface vertex of $V$ (see Figure \ref{fig:double_edge}).
    \begin{figure}[h!]
  \centering
  \includegraphics[width=50mm]{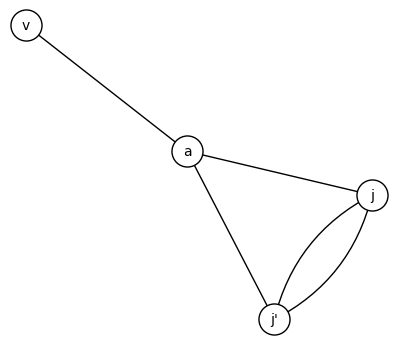}
\caption{The double edge from $\jp$ to $\jp'$.}
\label{fig:double_edge}
\end{figure}
This gives us two distinct loops in $\FpBarGraphtwo$ of length $3$.

These correspond to endomorphisms of norm $8$ in $\End_{\Fpbar}(E_{\jp})$.

We check for the existence of such an endomorphism using the modular polynomial $\Phi_8(X,X)$: We need to check whether the roots of $f(X)$ can simultaneously be the roots of the polynomial \begin{align*}
    \Phi_8(X,X)  =  &
    \left(-1\right) \cdot (X - 16581375)^{2} \cdot (X - 287496)^{2} \cdot (X + 3375)^{2} \cdot (X^{2} - 52250000 X + 12167000000)\\
     &\cdot (X^{3} + 3491750 X^{2} - 5151296875 X + 12771880859375)^{2} 
     \\ & \cdot (X^{3} + 39491307 X^{2} - 58682638134 X + 1566028350940383)^{2}.
\end{align*}
Take the resultant \begin{align*}
    \operatorname{Res}(f(X), \Phi_8(X,X))  = & 
    (-1) \cdot 3^{72} \cdot 5^{36} \cdot 7^{48} \cdot 11^{34} \cdot 13^{24} \cdot 37^{10} \cdot 41^{2} \cdot 43^8 \cdot 59^2 \cdot 71 \cdot 89^2 \cdot 101^2.
 \end{align*} For primes $p > 101$, this will be nonzero, and there is no such a loop in $\FpBarGraphtwo$, hence attachment happens. In the factorization of the resultant, there is one prime $p \equiv 11, 59 \mod 60$ and $3 \mod 4$. For $p = 11$, we only have one connected component of $\FpGraphtwo$, for $p = 59$, attachment happens.
\end{proof}

In the case $p \equiv 7 \mod 8$, attachments that can happen do \textit{not} necessarily.
We checked this for all primes $p \equiv 7 \mod 8$ between $50,000$ and $100,000$ such that the primes above $2$ do not generate the class group (in this case, there is only one component in $\FpGraphtwo$, see the following Section \ref{subsection:distances_of_components}). There are $217$ such primes, and for $41$ of them the attachment happens. However, there are $12$ primes $p$ for which the attachment can happen ($p\equiv11$ or $59$ $\mod 60$) but there is no attachment:
\begin{align*}
    53639, 58511, 66959, 71879, 72431, 72551, 79151, 86711, 88919, 90239, 93911, 99719.
\end{align*}

For $p =53639$, the two roots of $f(X)$ are $\jp = 30505$ and $\jp = 46665$. There are two elliptic curves with these $j$-invariants on the same component of $\FpGraphtwo$ which are $48$ edges apart.

\subsection{Distances of components of the $\Fp$-subgraph $\FpSubGraph$.} \label{subsection:distances_of_components}

In the above section, we have fully described how the spine $\FpSubGraph$ is formed by passing from $\FpGraphtwo$ to $\FpBarGraphtwo$. A natural question is how the spine $\FpSubGraph$ sits inside the graph $\FpBarGraph$.

For primes $p \equiv 1 \mod 4$, the subgraph is given by single edges (with a possibility of a few isolated vertices and one component of size $4$), as we proved in Section \ref{subsec:folding-stacking-attaching}. These components seem to be distributed the same way in the graph as random vertices: we compare the mean of the distances of the components with the distances between random vertices (100 random choices), normalized by the diameter. We compared these distances for $254$ primes with $p\equiv 1\mod{4}$ from $10253$ to $65437$. The primes were chosen to be spaced with a gap of at least $200$. Our results are shown in Figure ~\ref{fig:dist_comp}.

\begin{figure}[h!]
\begin{subfigure}{.45\textwidth}
  \centering
  \includegraphics[width=\linewidth]{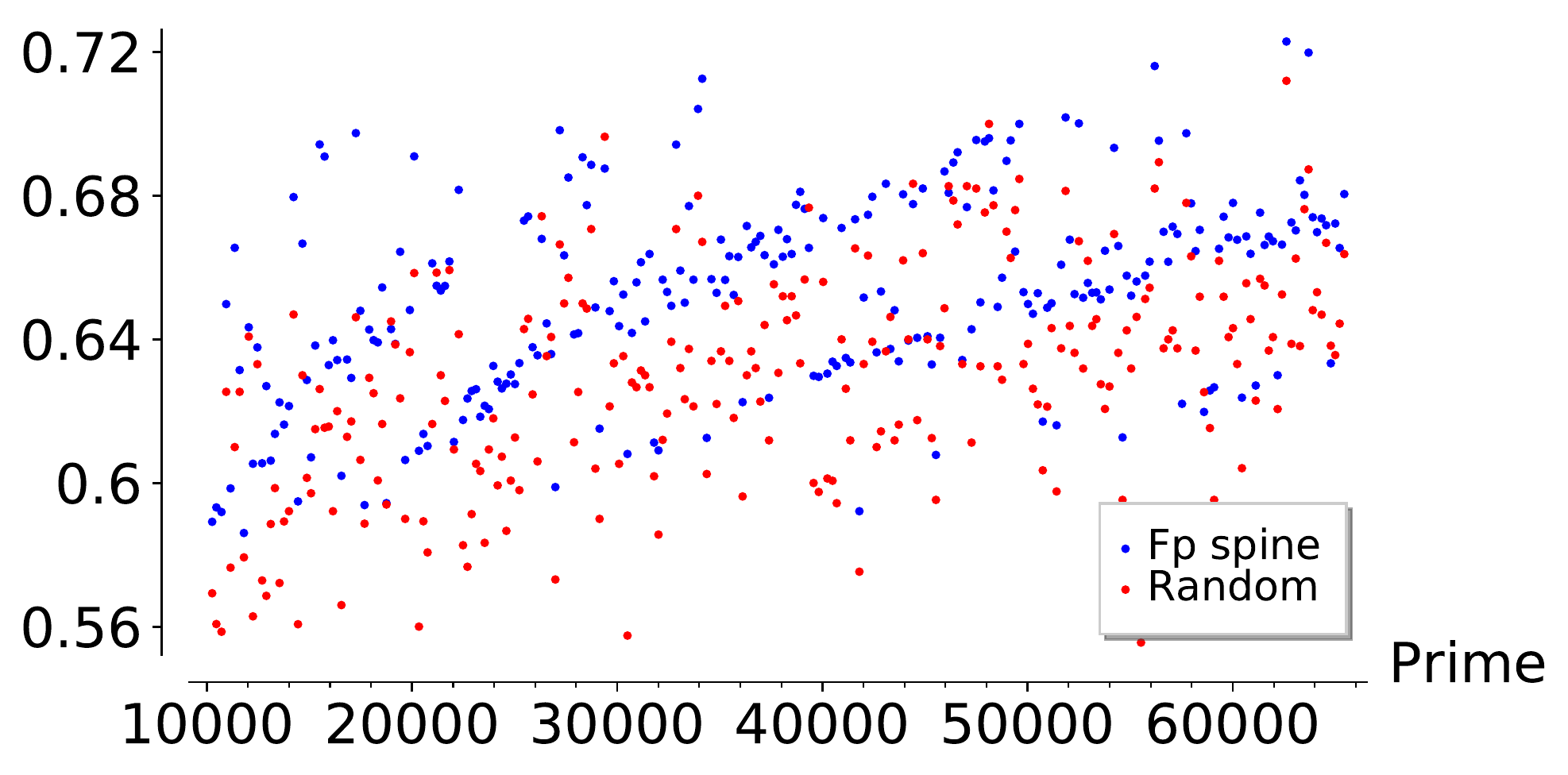}
  \caption{Normalized distance between $\FpSubGraph$ components (blue) and random pairs (red)}
\end{subfigure}
\begin{subfigure}{.45\textwidth}
  \centering
  \includegraphics[width=\linewidth]{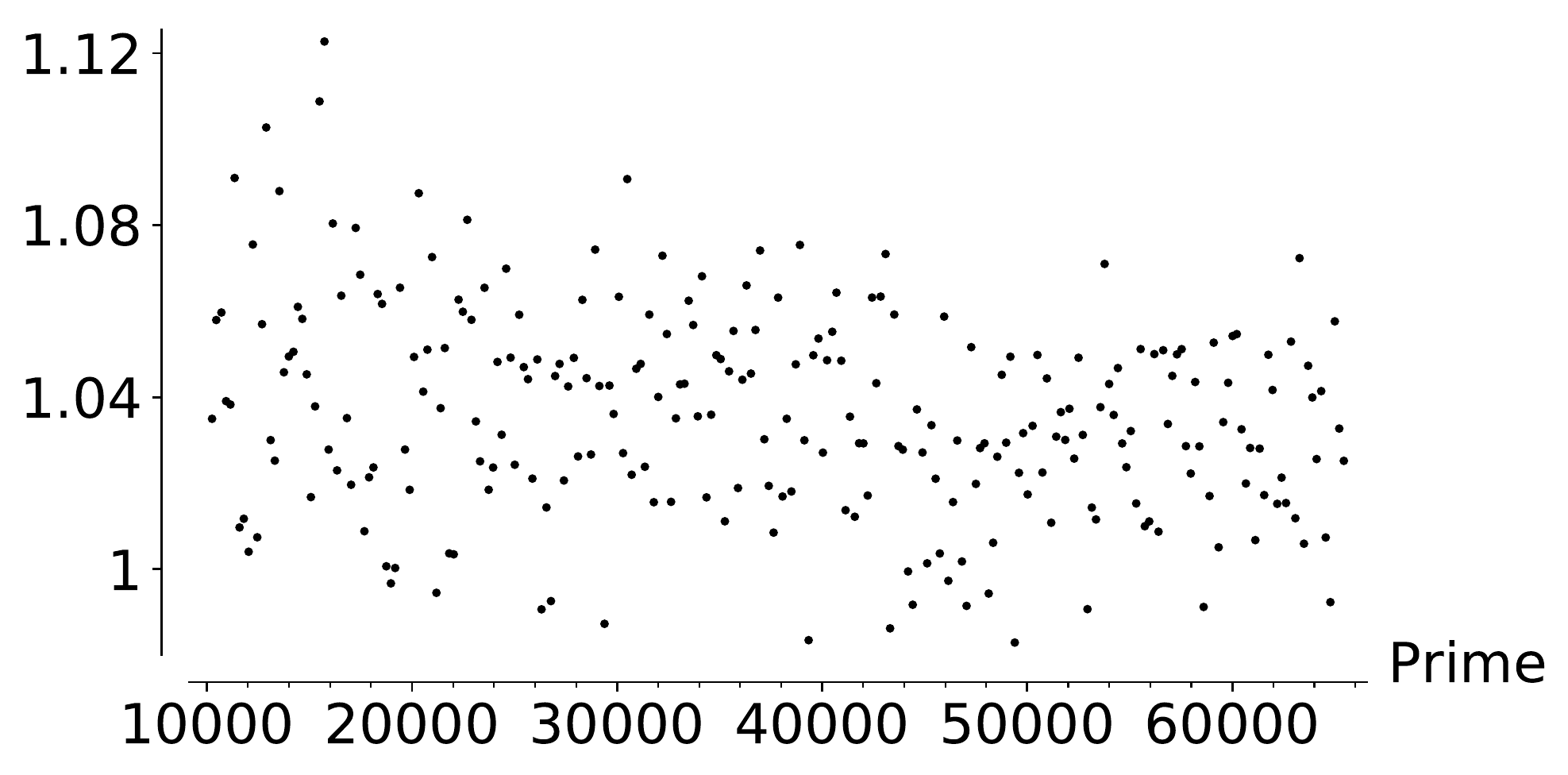}
  \caption{Ratio of distance between components vs distance between random pairs}
\end{subfigure}
\caption{Comparison of distances of $\FpSubGraph$ components versus distances between random vertices for $p\equiv1\mod{4}$.}
\label{fig:dist_comp}
\end{figure}

We do not know how to explain that the average distances between components seem to be larger than distance between two random points in $\FpBarGraphtwo$.

\subsubsection{$p \equiv 7 \mod 8$}

We start with the following easy lemma.

\begin{lemma}
Let $p\equiv 7\mod 8$ and set $K:=\Q(\sqrt{-p})$. Let $\mathfrak{p}$ denote a prime ideal of $\mathcal{O}_K$ above $(2)$ and suppose that $\langle\mathfrak{p} \rangle = \operatorname{Cl}(\mathcal{O}_K)$. 

Then, the $\FpGraphtwo$ is has only one connected component with $$\# \Cl(\OO_K) =  h(-p)$$ vertices on the surface and from every vertex on the surface, there is exactly one isogeny down.

A fortiori, the spine $\FpSubGraph \subset \FpBarGraph$ is connected.
\end{lemma}
\begin{proof}
An immediate consequence of \ref{sec:p3mod4}.
\end{proof}

It is interesting to know that the converse of this lemma is not true: If primes above $(2)$ do not generate the class group, it is still possible for the $\mathbb{F}_p$ subgraph of $\FpBarGraph$ to be connected, thanks to attaching.  

In the range $50,000 < p < 10000$, there are $217$ primes for which $\p$ does not generate the class group.

We have seen the following:
\begin{enumerate}
    \item for $12$ primes  $57119, 59471, 61871, 64439, 70439, 76871, 85199, 88799, 91631, 92399, 92951,$ and $ 96671$
    the spine $\FpSubGraph$ is nonetheless connected.
    \item for $57$ out of those $66$ primes there will be exactly $2$ connected components of $\FpSubGraph$ and those will be at most $6$ apart (with diameter being about $15$). For $29$ of these primes, $51287, 51383, 53639, 54559, 54767, 58511, 59063, 63439, 
    63799, 65831, 66863, 67751, 69191,\\ 70607, 72679, 74759, 76159, 79151, 80783, 82799, 83471, 84559, 
    85847, 86711, 90239, 91823, \\95959, 99079,$ and $ 99719$, these components are exactly $2$ apart.
 
 \end{enumerate}
The diameter is between 14 and 16.  The graphs have between 5400 and 8300 vertices.

These are the distances of non-normalized. The average distance of two random vertices for $2$-isogeny graphs of this size is around 9. This is approximately $0.6$ times the diameter of the graphs. This number grows slowly (for primes $p \approx 500,000$, the average distance of two random vertices is about $0.7$ times the diameter) and we expect it to converge to the diameter, however, we don't know how quickly.

We also computed the average of the mean distances of connected components of $\FpSubGraph \subset \FpGraphtwo$ for these primes. The mean is $4.3395$, with standard deviation $1.1092$, and the maximum is
$7.000$ and the minimum $2.333$, which indicates that the components tend to be close to each other. 

\subsubsection{The number of components}\label{sec:num-comp}

We estimate the number of connected components of $\FpSubGraph$, under the assumption $h(D) \approx \sqrt{D}$. By Theorem \ref{theorem:stacking-folding-attaching}, the number of vertices of $\FpSubGraph$ is approximately half (respectively, one fourth) of the size of $\FpSubGraph$ if $p \equiv 1 \mod 4$ (resp., $p \equiv 3 \mod 8$) and depends on the order of the prime lying above $2$ for $p \equiv 7 \mod 8$.

\begin{tabular}{c|c|c|c}
   & & &    \\
  prime mod 8 & shape of $\FpGraphtwo$   & $\# \FpSubGraph$ & $\approx$ number of components   \\ \hline
   $1 \mod 4$   & edges & $\frac{1}{2} h(-4p)$ & $\frac{1}{4} h(-4p)$  \\
   $3 \mod 8$ & claw & $2 h(-p)$ &$ \frac{1}{4} \cdot 2 h(-p) = \frac{1}{2}h(-p)$ \\ 
   $7 \mod 8$ & volcanoes  & $h(-p)$ &$ \frac{1}{2 \cdot \operatorname{ ord}(\p_2)} \cdot h(-p) << \frac{1}{2} h(-p) $ \\
    & (2 levels, size $\operatorname{ ord}(\p_2)$) & & 
\end{tabular}

\section{Conjugate vertices, distances, and the spine}

We examine several distances of cryptographic interest. In Section \ref{sec:dist-conj-pairs} we study the distance between Galois conjugate pairs of vertices, that is, pairs of $\jpp$-invariants of the form $\jpp$, $\jpp^p$. Our data suggests these vertices are closer to each other than a random pair of vertices in $\FpBarGraphtwo$. In Section \ref{sec:shortest-paths-through-spine} we test how often the shortest path between two conjugate vertices goes through the spine $\FpSubGraph$, or equivalently, contains a $j$-invariant in $\mathbb F_p$. We find conjugate vertices are more likely than a random pair of vertices to be connected by a shortest path through the spine. Finally, we examine the distance between arbitrary vertices and the spine $\FpSubGraph$ in Section \ref{sec:dist-to-Fp}. 

\subsection{Distance between conjugate pairs}
\label{sec:dist-conj-pairs}
Isogeny-based cryptosystems such as cryptographic hash functions and key exchange rely on the difficulty of computing paths (\textit{routing}) in the supersingular graph $\FpBarGraph$. Our experiments with $\ell=2$ show that two random conjugate vertices are ``closer" than two random vertices. 

We tested the distances of conjugate vertices as follows. First for a given prime $p$, we constructed the graph $\FpBarGraphtwo$. Then we computed the distances $\mathrm{dist}(\jpp_1,\jpp_2)$ between all pairs $\jpp_1,\jpp_2 \in \FpBarGraphtwo$. These values were organized into two lists:
\begin{align*}
    C_p &= [\mathrm{dist}(\jpp,\jpp^p) \colon \jpp \in \F_{p^2} \setminus \Fp]
    \\
    A_p &= [\mathrm{dist}(\jpp_1,\jpp_2) \colon \jpp_1,\jpp_2 \in \F_{p^2} \setminus \Fp].
\end{align*}

The distributions $C_p$ and $A_p$ for $p = 19489$ are shown as histograms in Figure~\ref{fig:distances-between-pairs-19489}. We call the pairs from $C_p$ \textit{conjugate} pairs and pairs from $A_p$ \textit{arbitrary} pairs.

\begin{figure}[H]

\centering
\begin{subfigure}{.45\linewidth}
  \includegraphics[width=\textwidth]{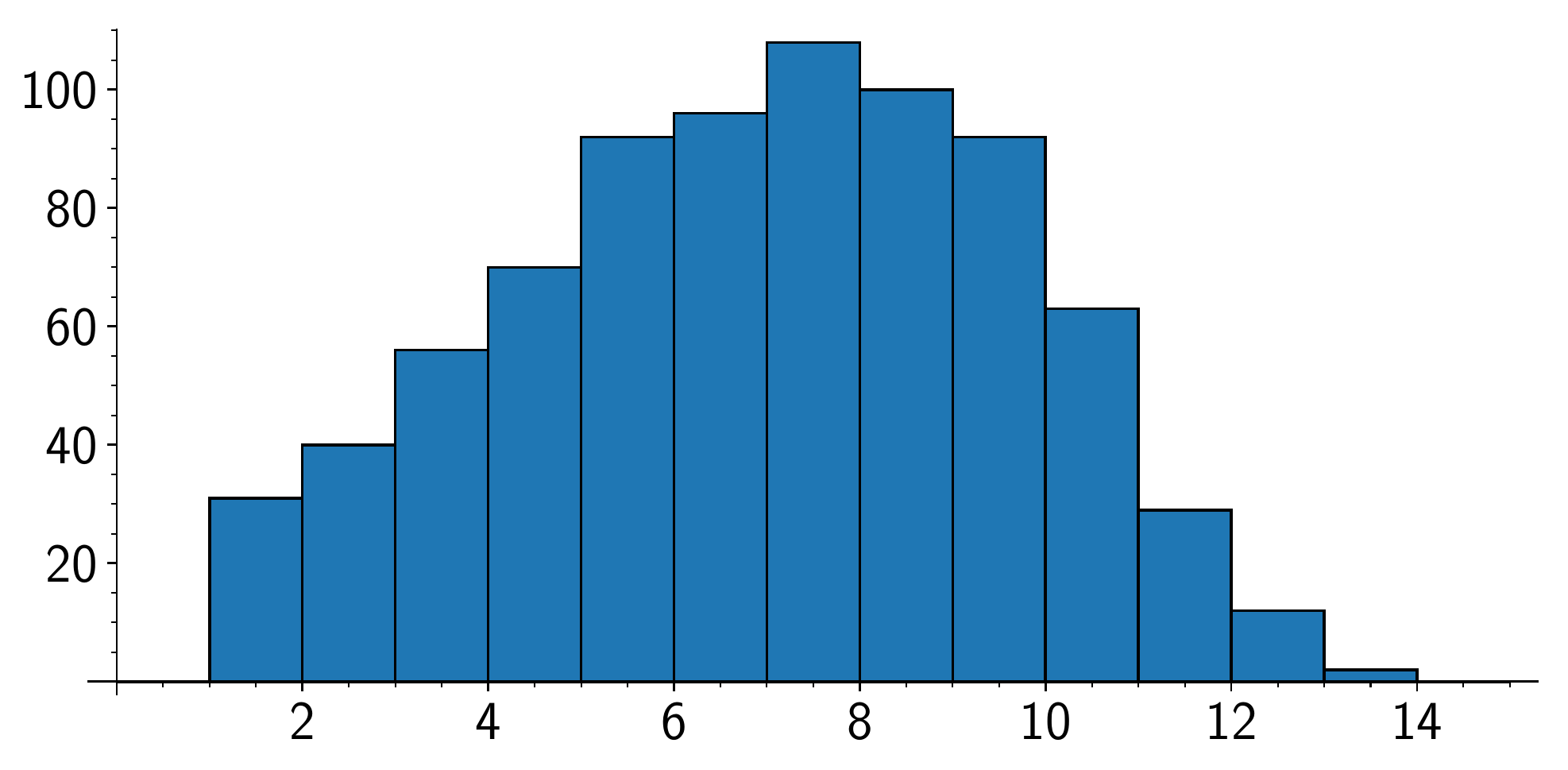}
  \caption{Distances between conjugate pairs.}
\end{subfigure}
\begin{subfigure}{.45\linewidth}
  \includegraphics[width=\textwidth]{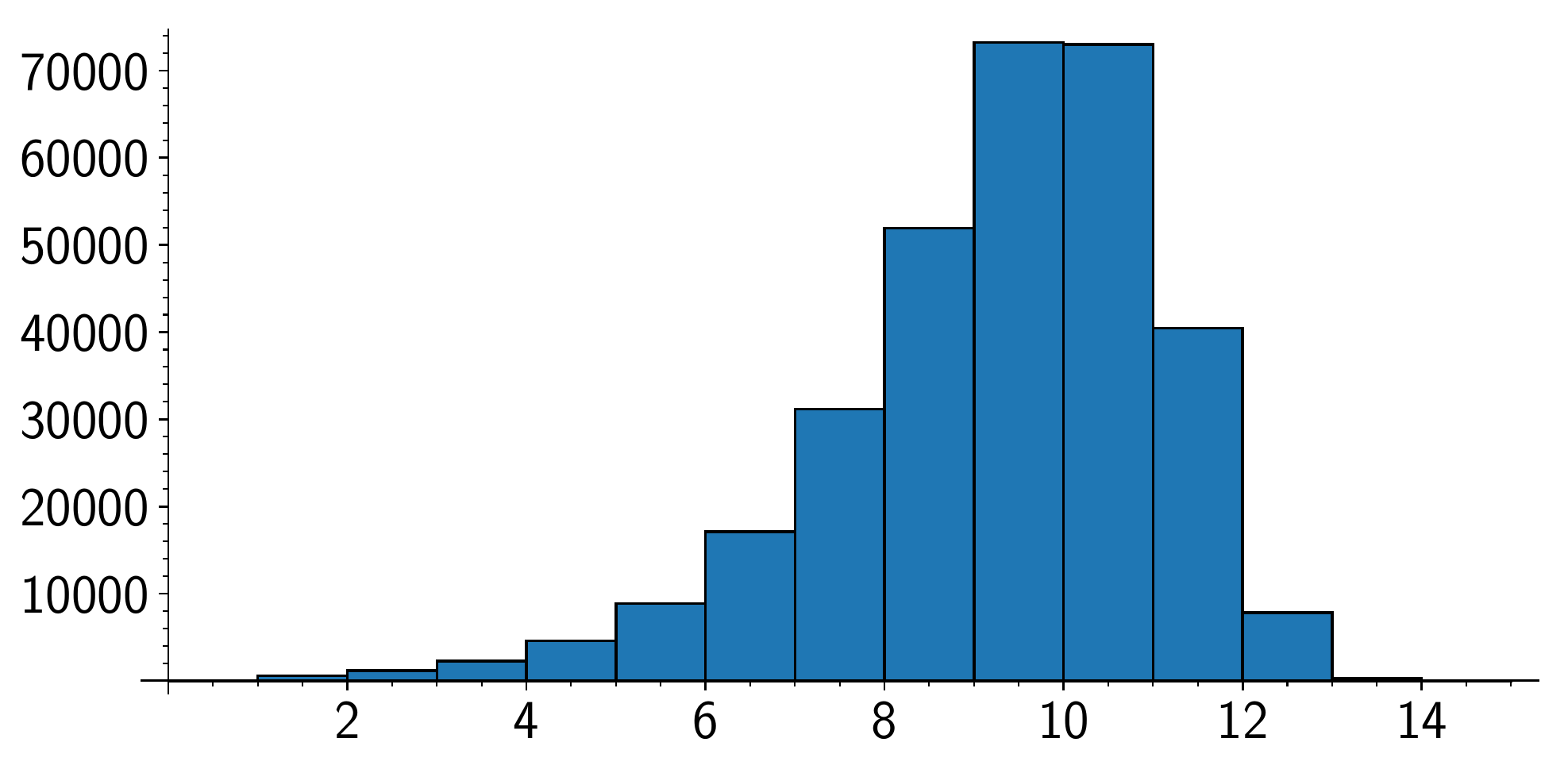}
  \caption{Distances between arbitrary pairs.}
\end{subfigure}
\caption{Distances measured between conjugate pairs and arbitrary pairs of vertices not in $\Fp$ for the prime $p = 19489$.}
\label{fig:distances-between-pairs-19489}
\end{figure}

For a larger prime, it is too costly to iterate over all vertices. Instead, we took a random sample of $1000$ conjugate and arbitrary pairs. The data collected for the prime $p = 1000003$ is shown in Figure~\ref{fig:distances-between-pairs-1000003}.

\begin{figure}[h!]

\centering
\begin{subfigure}{.45\linewidth}
  \includegraphics[width=\textwidth]{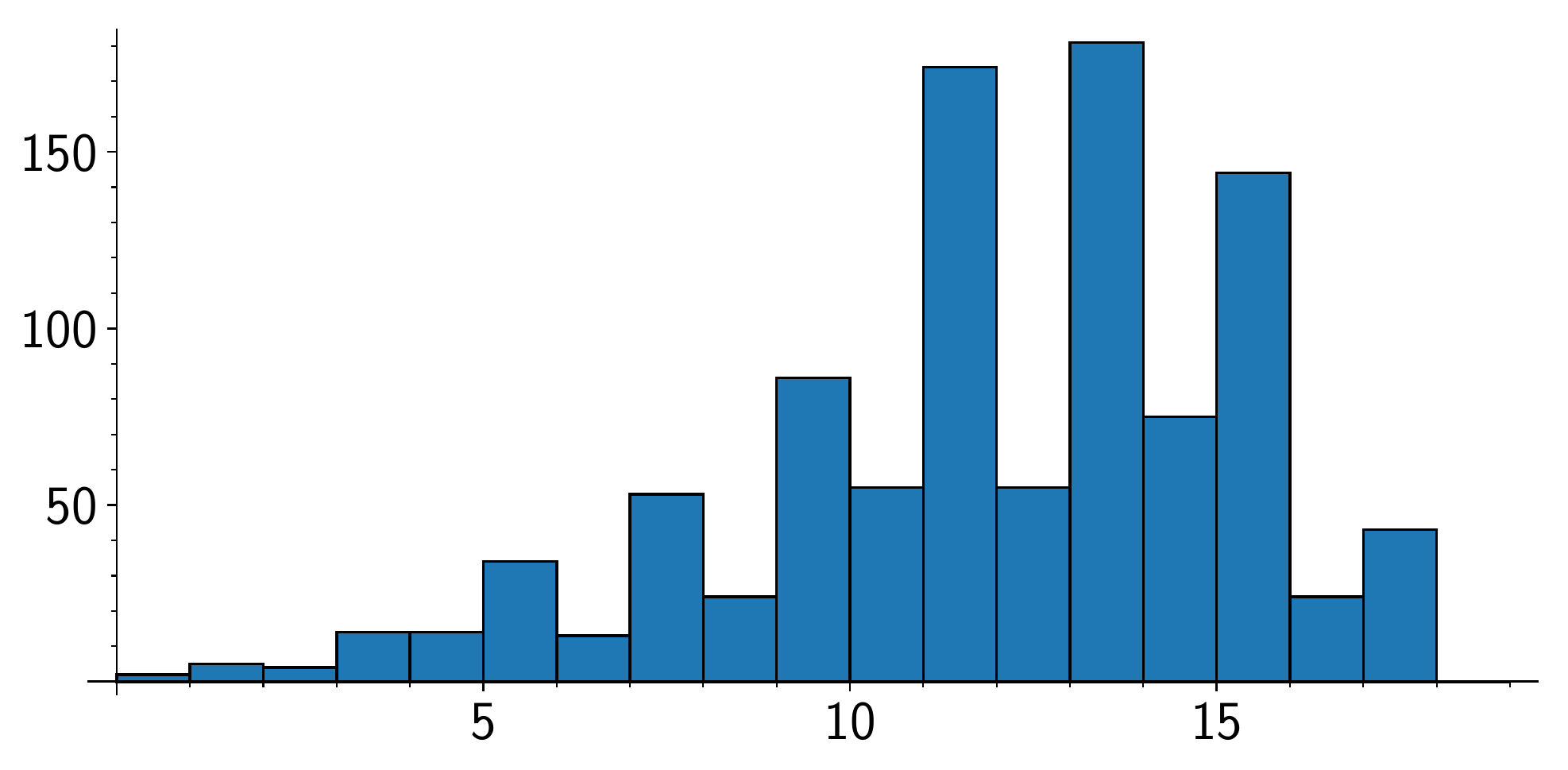}
  \caption{Distances between conjugate pairs.}
\end{subfigure}
\begin{subfigure}{.45\linewidth}
  \includegraphics[width=\textwidth]{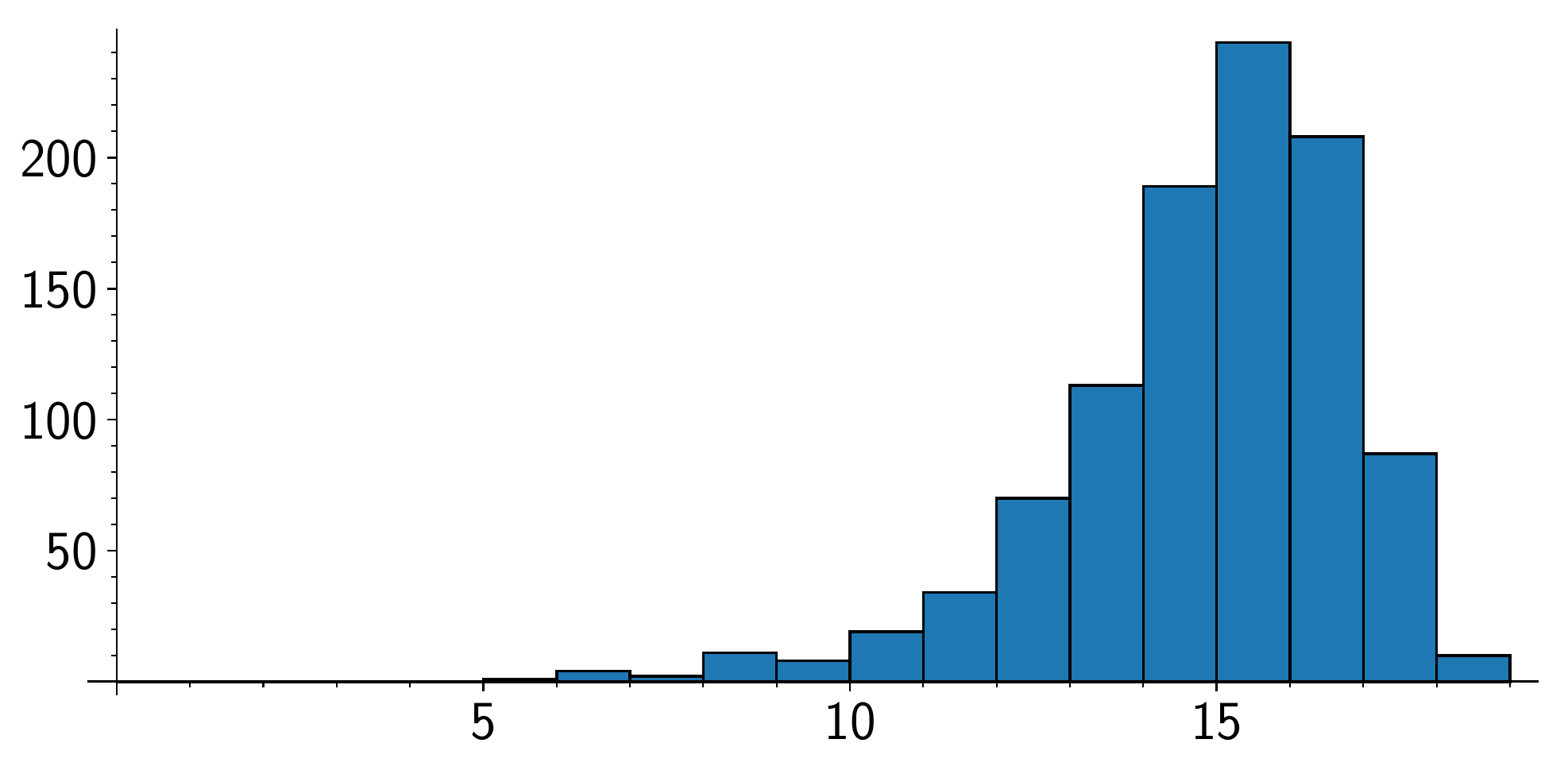}
  \caption{Distances between arbitrary pairs.}
\end{subfigure}
\caption{Distances between $1000$ randomly sampled pairs of arbitrary and conjugate vertices for the prime $p = 1000003$.}
\label{fig:distances-between-pairs-1000003}
\end{figure}

From our data, it seems likely that distances between conjugate vertices have a different distribution than distances between arbitrary vertices. However, more study on a broader sample of primes is needed.

\begin{remark}
    In Figure~\ref{fig:distances-between-pairs-1000003}, we see a clear bias towards paths of odd length (that is, odd number of edges). 
    This is due to the fact that conjugate $j$-invariants often admit a shortest path that is a mirror path (Definition~\ref{def:mirror_path}). These paths do not usually go through the spine $\FpSubGraph$, so they have an even number of vertices and an odd number of edges. This topic is studied further in Section~\ref{sec:shortest-paths-through-spine}.
\end{remark}

\subsection{How often do shortest paths go through the $\Fp$-spine}
\label{sec:shortest-paths-through-spine}

It was shown in \cite{DelGal01} that if one navigates to the spine $\FpSubGraph$, one obtains a subexponential attack on the path finding problem. This attack, however, uses $L$-isogenies, where $L$ is a set of small primes. We study the situation when one only uses $L =\{ 2\}$. When $\jpp'= \jpp^p$, any path from $\jpp$ to the spine $\FpSubGraph$ can then be \textit{mirrored} to obtain a path of equal length from $\jpp^p$ to the same point of the spine, and hence a path between $\jpp$ and $\jpp^p$ passing through the spine. This notion motivates the following definition:

\begin{definition}
A pair of vertices are \textbf{opposite} if there exists a shortest path between them that passes through the $\Fp$ spine.
\end{definition}

\subsubsection{Experimental methods} 

We tested how often a shortest path between two conjugate vertices went through the spine $\FpSubGraph$.
Shortest paths are not necessarily unique, so it is not enough to compute a shortest path and check whether passes through the spine.
We used the built-in function of Sage (\cite{sage}) to perform our computations. For efficiency, we did not compute all the shortest paths. Instead, to verify whether a pair $\mathbcal{j}_1,\mathbcal{j}_2$ is opposite, we run over all vertices in $\mathbf{j}\in\Fp$ and check whether there is a $\mathbf{j}$ such that
\[ dist(\mathbcal{j}_1,\mathbcal{j}_2) = dist(\mathbcal{j}_1,\mathbf{j}) + dist(\mathbf{j},\mathbcal{j}_2).\]

For smaller primes ($<5000$) we computed the proportions for all pairs of vertices in $\Fptwo\backslash\Fp$. For larger primes, we randomly selected $1000$ pairs of points $\mathbcal{j}_1,\mathbcal{j}_2$ in $\Fptwo\backslash\Fp$ and checked whether each of the pairs $(\mathbcal{j}_1,\mathbcal{j}_2), (\mathbcal{j}_1, \mathbcal{j}_1^p)$ were opposite.

\subsubsection{Conjugate pairs vs arbitrary pairs}

\begin{figure}[h!]
\begin{subfigure}{.45\textwidth}
  \centering
  \includegraphics[width=\linewidth]{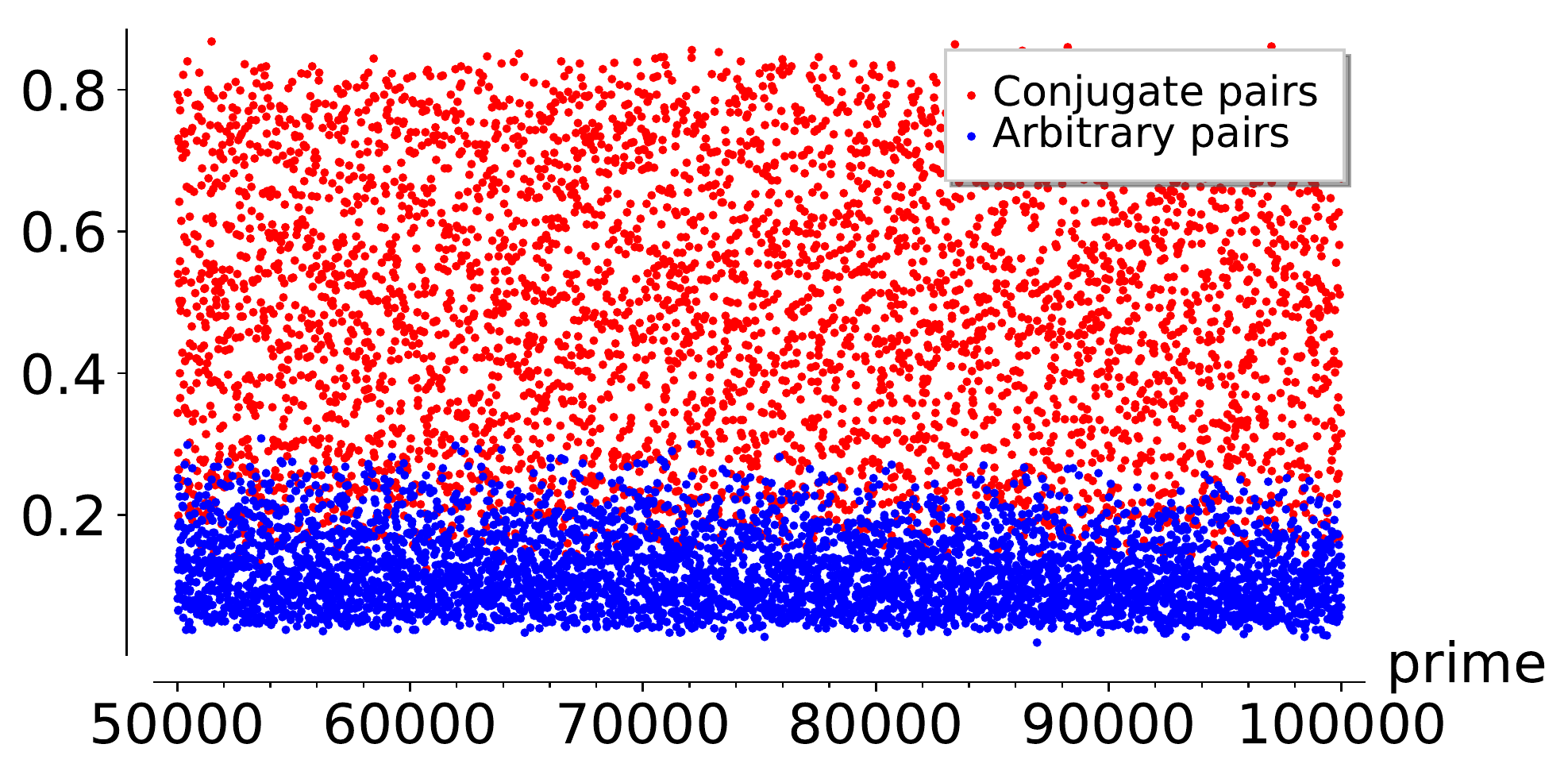}
\caption{Proportions of opposite pairs}
\end{subfigure}
\begin{subfigure}{.45\textwidth}
  \centering
  \includegraphics[width=\linewidth]{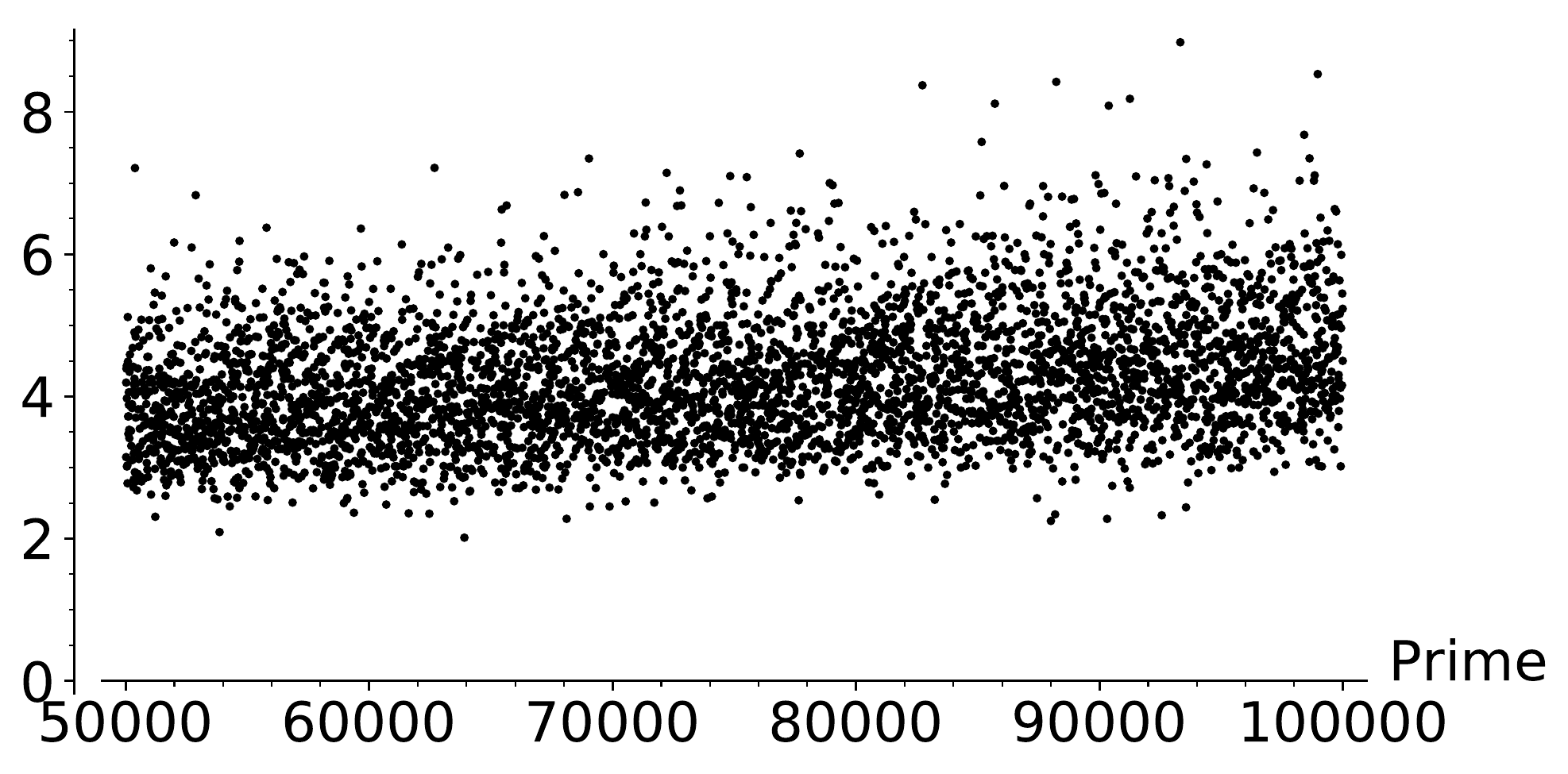}
\caption{Ratio of opposite conjugate pairs vs opposite arbitrary pairs}
\end{subfigure}
\caption{Data for random sample of 1000 pairs of conjugate and arbitrary pairs.}
\label{fig:conj_arb}
\end{figure}

\begin{figure}[h!]
\begin{subfigure}{.45\textwidth}
  \centering
  \includegraphics[width=\linewidth]{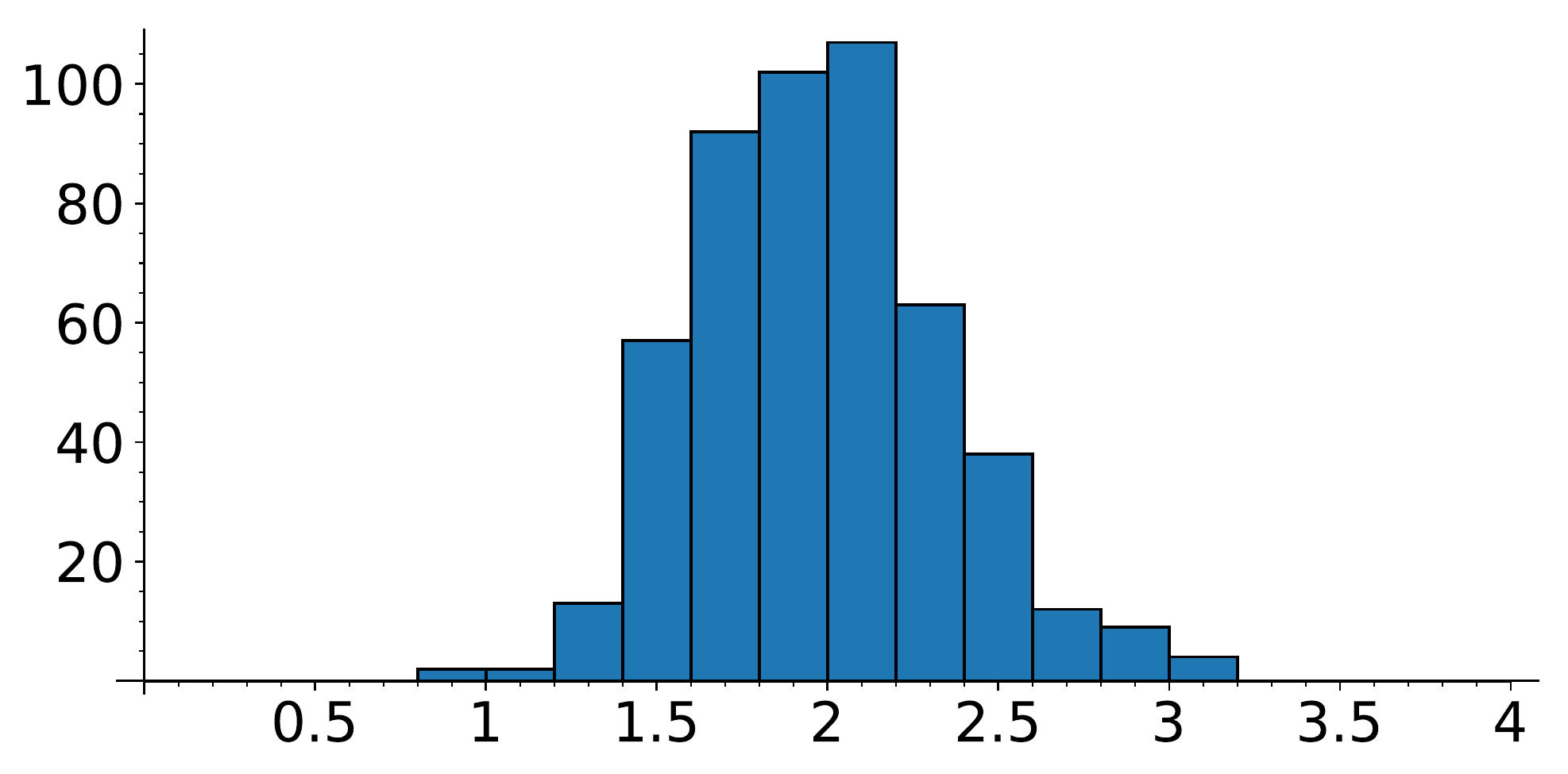}
  \caption{Primes from 1000 to 5000, average is $1.98$}
\end{subfigure}
\begin{subfigure}{.45\textwidth}
  \centering
  \includegraphics[width=\linewidth]{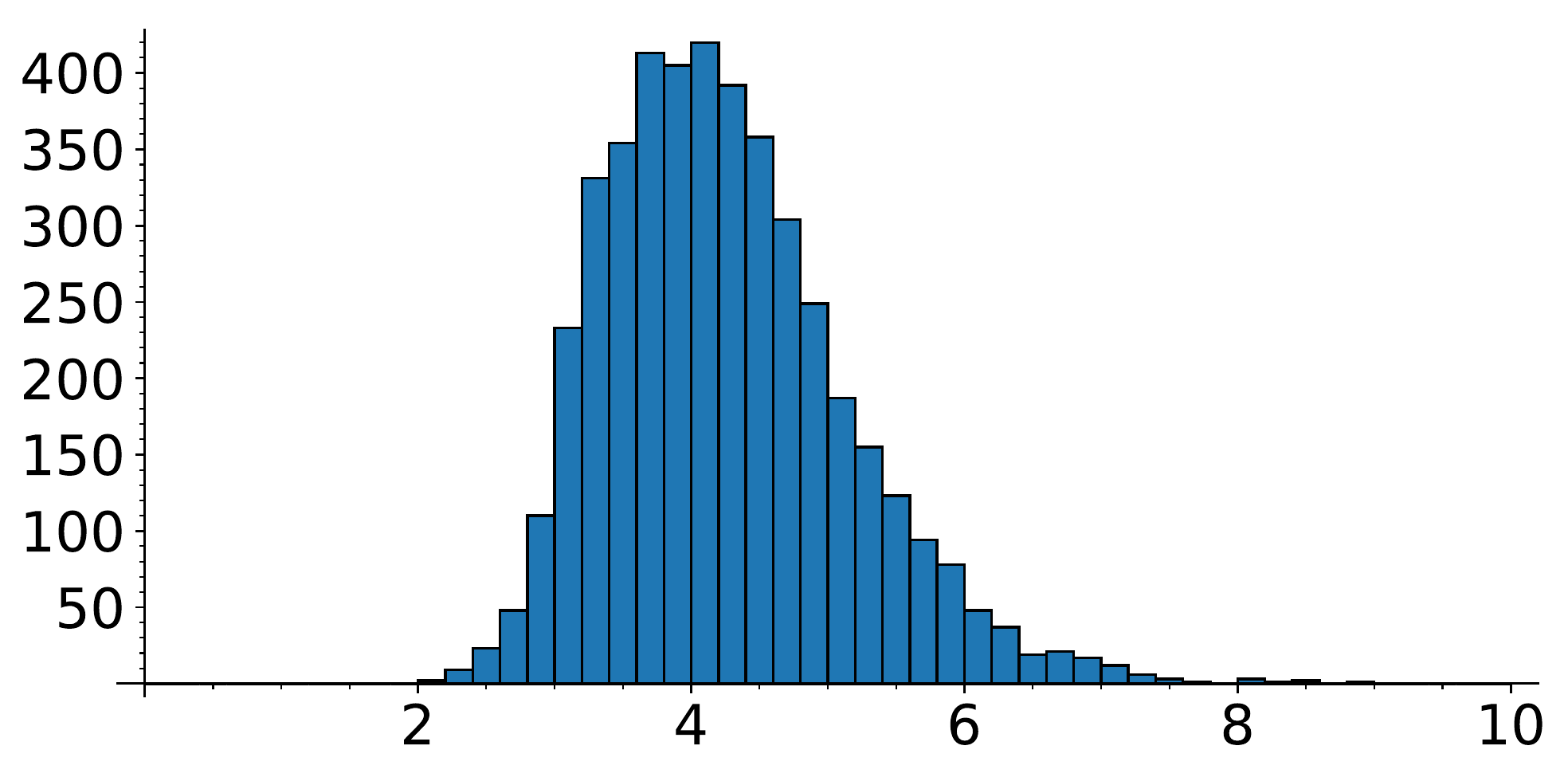}
  \caption{Primes from 50000 to 100000, random sampling of pairs, average is $ 4.25$.}
\end{subfigure}
\caption{ Histogram of primes with proportion of opposite conjugate pairs divided by the proportion of opposite arbitrary pairs as in \eqref{formula:proportion}.}
\label{fig:conj_arb_hist}
\end{figure}
Our data suggests that conjugate vertices are more likely to be opposite than arbitrary vertices. For a random sampling of pairs over primes between $50000$ and $100000$, we observe that 
\begin{align} \label{formula:proportion}
\text{average }\Bigg(\frac{\#\text{opposite conjugate pairs}}{\#\text{opposite arbitrary pairs}}\Bigg)\approx 4.25    
\end{align}

The ratio seems to increase with the size of the prime, as seen in Figures ~\ref{fig:conj_arb} and ~\ref{fig:conj_arb_hist}. 
This leads to the following observation: Due to the mirror involution, to build the graph $\FpBarGraph$, one can start with the spine $\FpSubGraph$ and keep adding edges along with their mirror edges. This might suggest that the spine is central to the graph. However, the shortest paths between arbitrary pairs of vertices are less likely to pass through the spine, contradicting that perspective.

\subsubsection{Proportions varying over different residue classes}

We observe that the proportion of pairs of opposite vertices varies based on the residue class of $p$. In this section, we consider arbitrary pairs of vertices. From the data, as shown in Figure~\ref{fig:opp_mod15}, the proportion is higher for primes $p\equiv 2\mod{3}$ compared to $p\equiv 1\mod{3}$ and higher for primes $p\equiv \pm2\mod{5}$ compared to $p\equiv \pm1\mod{5}$.

\begin{figure}[H]
\begin{subfigure}{.45\textwidth}
  \centering
  \includegraphics[width=\linewidth]{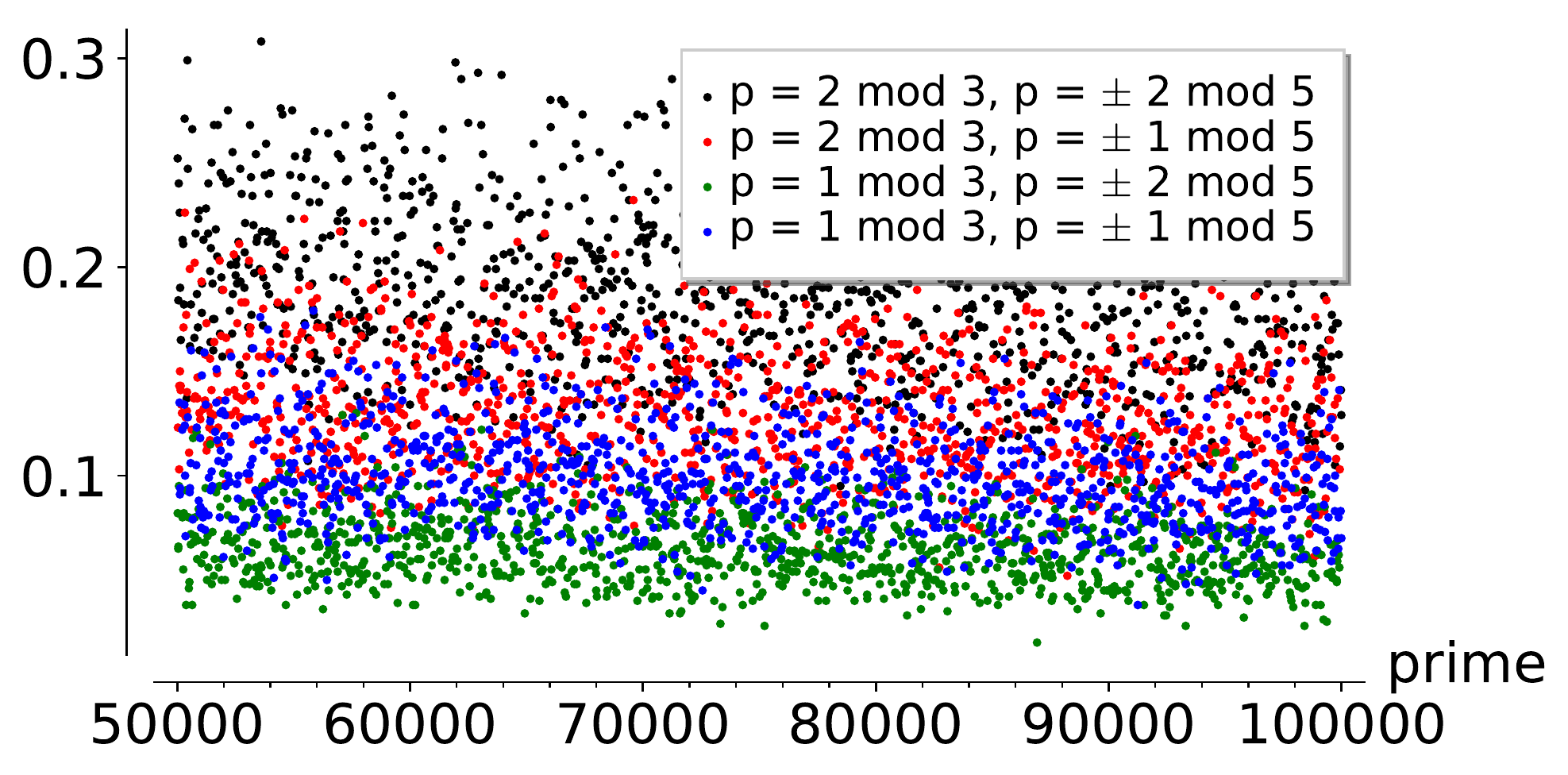}
  \caption{Proportion $\mod{15}$}
  \label{fig:opp_mod15}
\end{subfigure}
\begin{subfigure}{.45\textwidth}
  \centering
  \includegraphics[width=\linewidth]{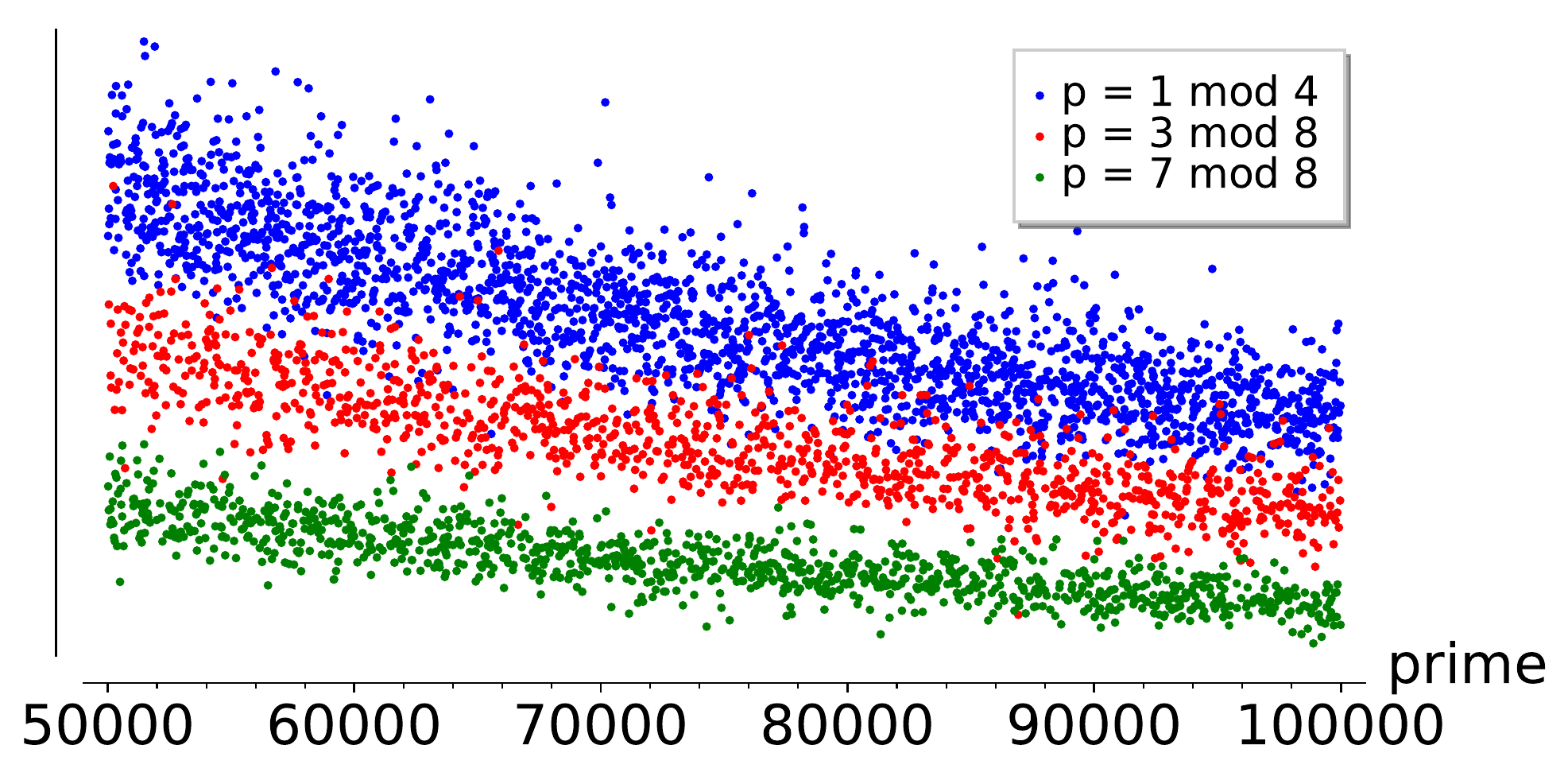}
  \caption{Normalized proportion $\mod{8}$}
  \label{fig:opp_mod8}
\end{subfigure}
\caption{Proportion of opposite pairs out of a random sample of $1000$ pairs.}
\end{figure}
 Based on our results, we suggest that the size and connectedness of the $\Fp$ spine could be key factors affecting the proportion of opposite pairs.

\begin{enumerate}
    \item Size of $\Fp$ spine: when the number of $\Fp$ points is higher, pairs are more likely to have shortest paths through these points. 
    \begin{itemize}
        \item To consider this effect, we study each proportion divided by the number of $\Fp$ points for the prime $p$. After normalizing the proportions, we no longer see clear differences when considering residue classes $\mod{3}$ and $\mod{5}$. This suggests that the underlying cause of the difference was the size of the $\Fp$ spine. 
        \item However, the normalized proportions as shown in Figure~\ref{fig:opp_mod8} appear to fall into three classes $p\equiv1\mod{4}$, $p\equiv3\mod{8}$ and $p\equiv7\mod{8}$. One possible cause for this is the connectedness of the $\Fp$ spine.
    \end{itemize}
    
    \item Connectedness of $\Fp$ spine: when the $\Fp$ spine is less connected to itself, pairs are more likely to have shortest paths through $\FpSubGraph$.
    \begin{itemize}
        \item From the table in Section \ref{sec:num-comp}, the spine is the least connected when $p\equiv1\mod{4}$, and can be highly connected when $p\equiv7\mod{8}$. This could explain the difference in proportions when normalized by the size of $\FpSubGraph$.
        \item For example, we consider the cases $p_1=19991$ ($p_1\equiv 7\mod{8}$, $\FpSubGraph$ is connected, $|\FpSubGraph| = 199$) and $p_2=19993$ ($p\equiv 1\mod{4}$, $\FpSubGraph$ is maximally disconnected, $|\FpSubGraph| = 30$). We would expect $199/30 > 6$ times more opposite pairs in the $p_1$ case. However, for 1000 random pairs, 266 pairs were opposite for $p_1$ compared to 112 pairs for $p_2$.
    \end{itemize}
    To further study whether differences occurring in the normalized proportion $\mod{8}$ were due to the connectedness of the $\Fp$ spine or other structures of $\FpBarGraphtwo$, we took a random subgraph of the same size as $\FpSubGraph$ and obtained the proportion of pairs with a shortest path passing through the random subgraph. We took the average of these results over $10$ random subgraphs for each prime between $1000$ and $5000$.
    
    \begin{figure}[h!]
    \begin{subfigure}{.45\textwidth}
    \centering
    \includegraphics[width=\linewidth]{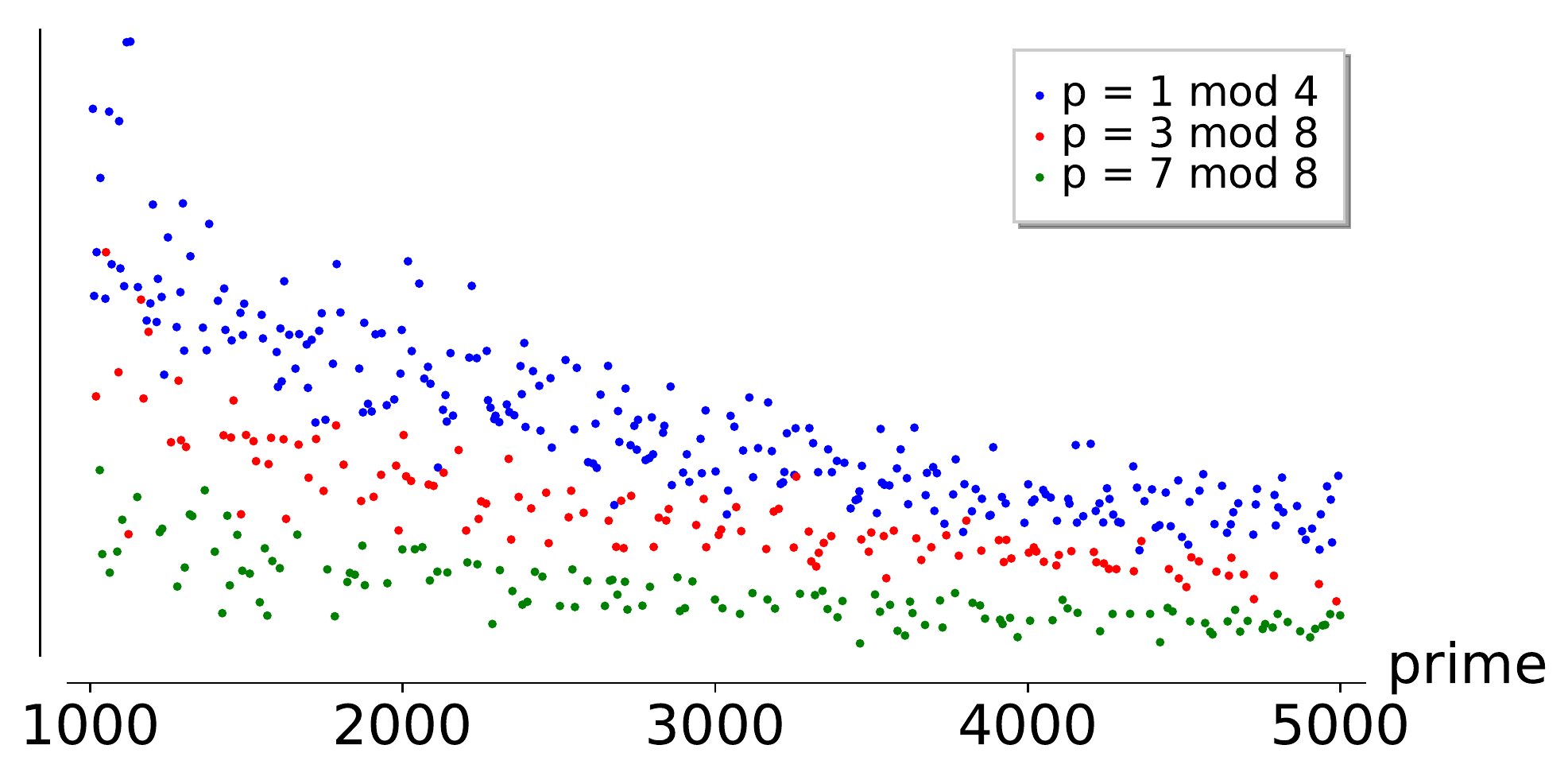}
    \caption{$\Fp$ spine}
    \end{subfigure}
    \begin{subfigure}{.45\textwidth}
    \centering
    \includegraphics[width=\linewidth]{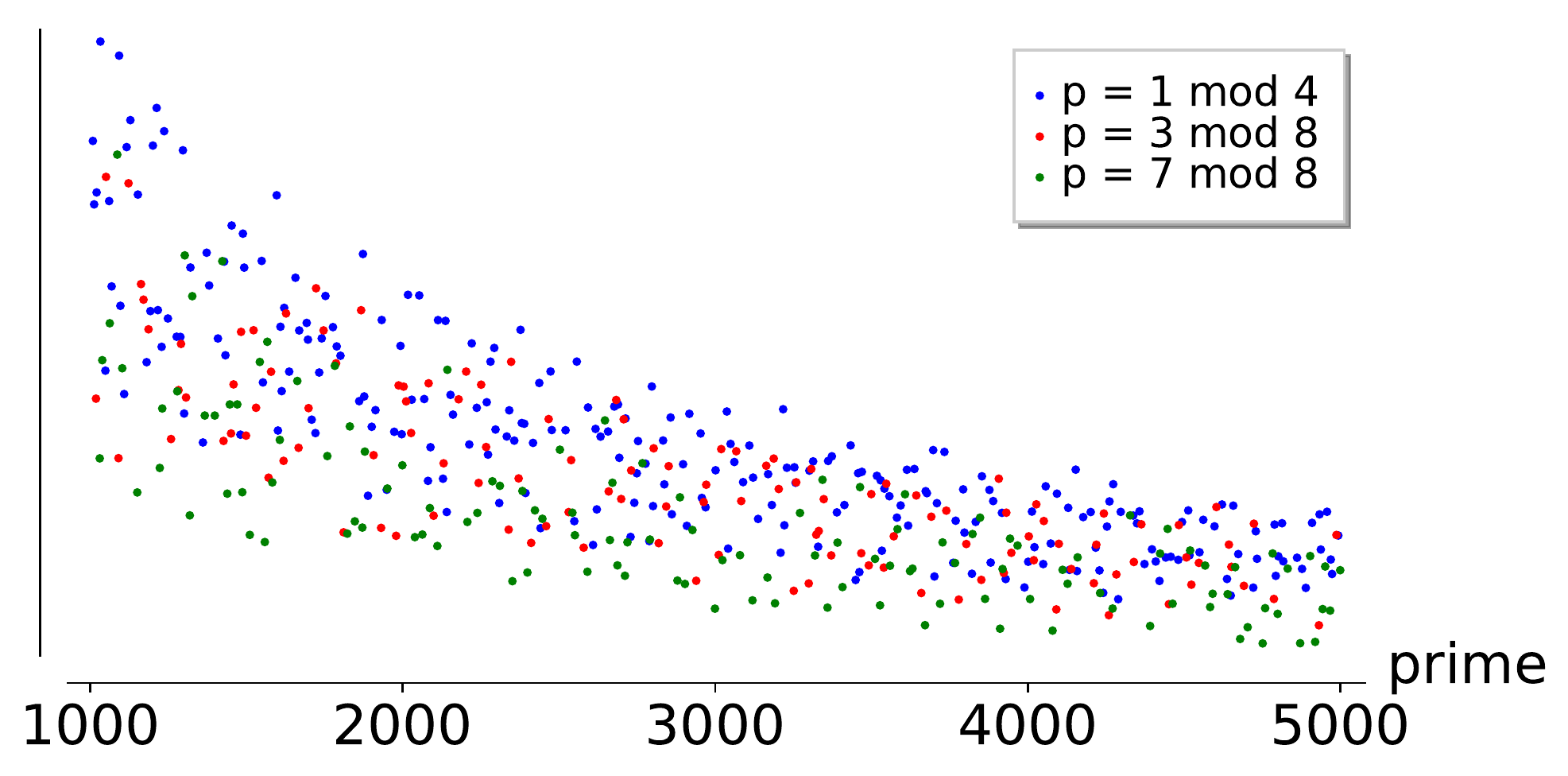}
    \caption{Random subgraph}
    \end{subfigure}
    \caption{Normalized proportion of pairs with a shortest path through the subgraph specified.}
    \label{fig:rand_vs_opp}
    \end{figure}
    
    From the data in Figure~\ref{fig:rand_vs_opp}, there is less distinction $\mod{8}$ for random subgraphs. This suggests that the connectedness of $\FpSubGraph$ is the dominant factor affecting the normalized proportion.
\end{enumerate}

\subsection{Distance to spine}\label{sec:dist-to-Fp}

In this section, we compare the {\it distance from a random vertex to the spine}, with the {\it distance from a random vertex to a random subgraph of the same size as the spine}. We observe that if the spine is connected, then the distance to the spine seems greater than the distance to a random subgraph. This agrees with the intuition that a small connected subgraph (remember that the spine has size $O(\sqrt{p})$) will be further from most vertices than a random subgraph, which will have many connected components uniformly distributed throughout the graph.

We tested the distances as follows. For a value of $p$, we constructed the graph $\FpBarGraphtwo$, the spine $S_0 := \FpSubGraph$, and chose several random subgraphs $S_1,\dots,S_n$. We define the distance between a vertex $j$ and a subgraph $S_i$ to be
\[
    \mathrm{dist}(j,S_i) = \min\{\mathrm{dist}(j,j'): j' \in S_i\}.
\]
We computed lists $d_i = [\mathrm{dist}(j,S_i) \colon j \in \FpBarGraphtwo]$ in order to measure how dispersed $S_i$ is in $\FpBarGraphtwo$.

Distances were computed for two primes, $p = 19991$ and $p = 19993$. Histograms of the distributions of the $d_i$ are given in Figure~\ref{fig:distances-to-spine-vs-random}. For $p = 19991$, the subgraph $\FpSubGraph$ is connected, whereas for $p = 19993$, $\FpSubGraph$ is maximally disconnected because $19993 \equiv 1 \mod{12}$ (see Lemma~\ref{lem:spine-1-mod-12}).

\begin{figure}[H]

\centering
\begin{subfigure}{.45\linewidth}
  \includegraphics[width=\textwidth]{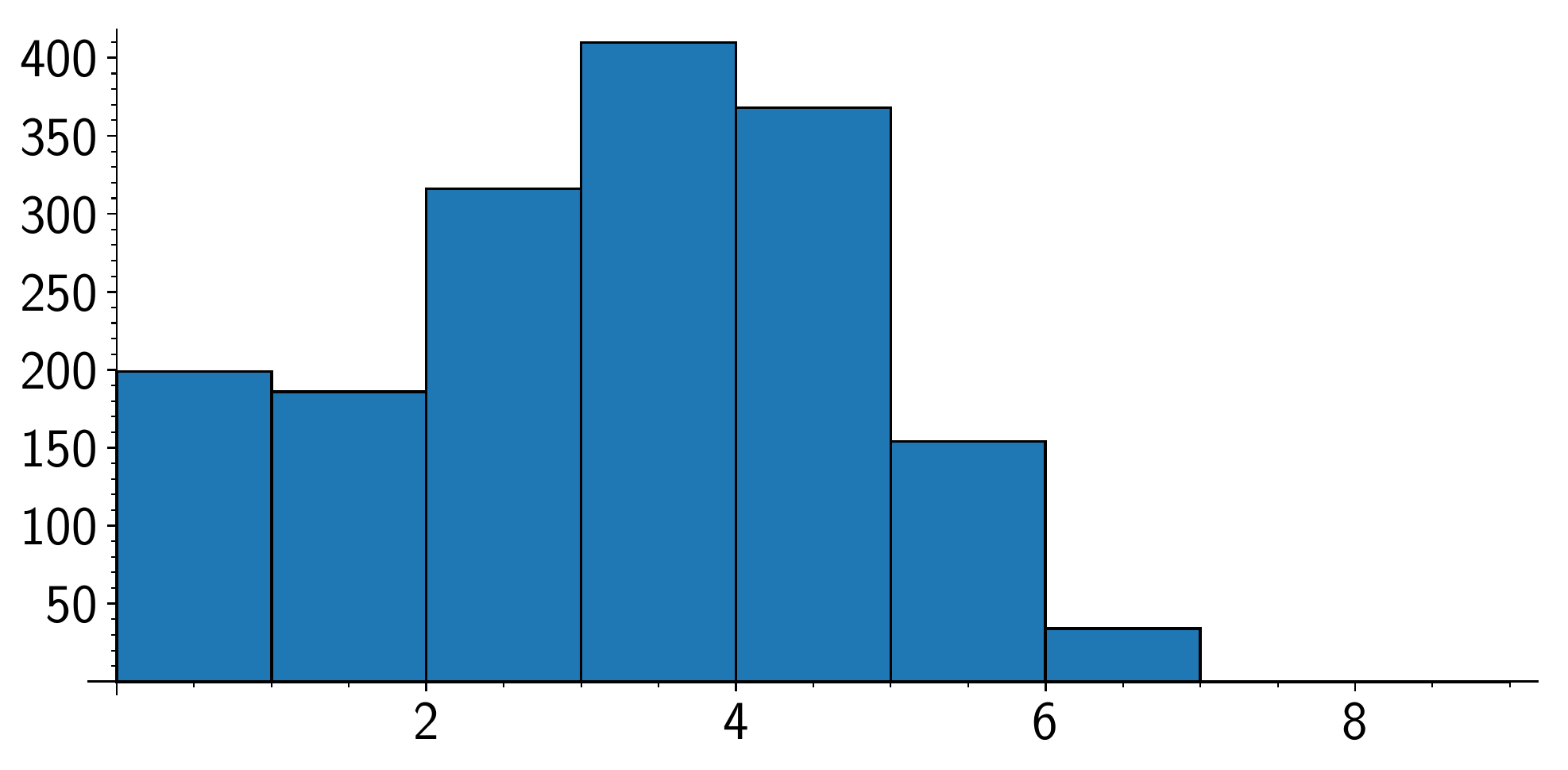}
  \caption{Distances to $\FpSubGraph$ for $p = 19991$.}
\end{subfigure}
\begin{subfigure}{.45\linewidth}
  \includegraphics[width=\textwidth]{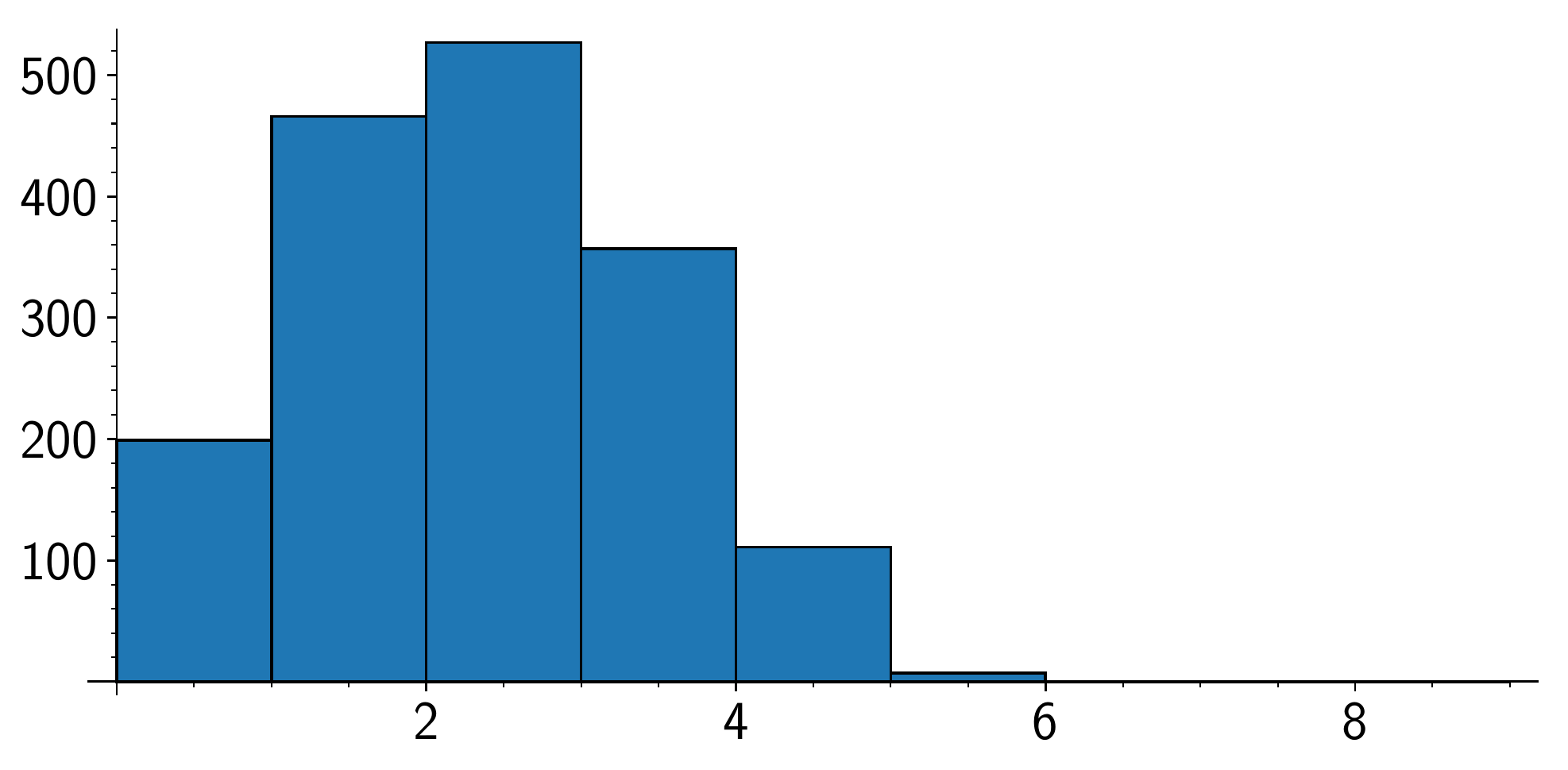}
  \caption{Distances to $R$ for $p = 19991$.}
\end{subfigure}
\begin{subfigure}{.45\linewidth}
  \includegraphics[width=\textwidth]{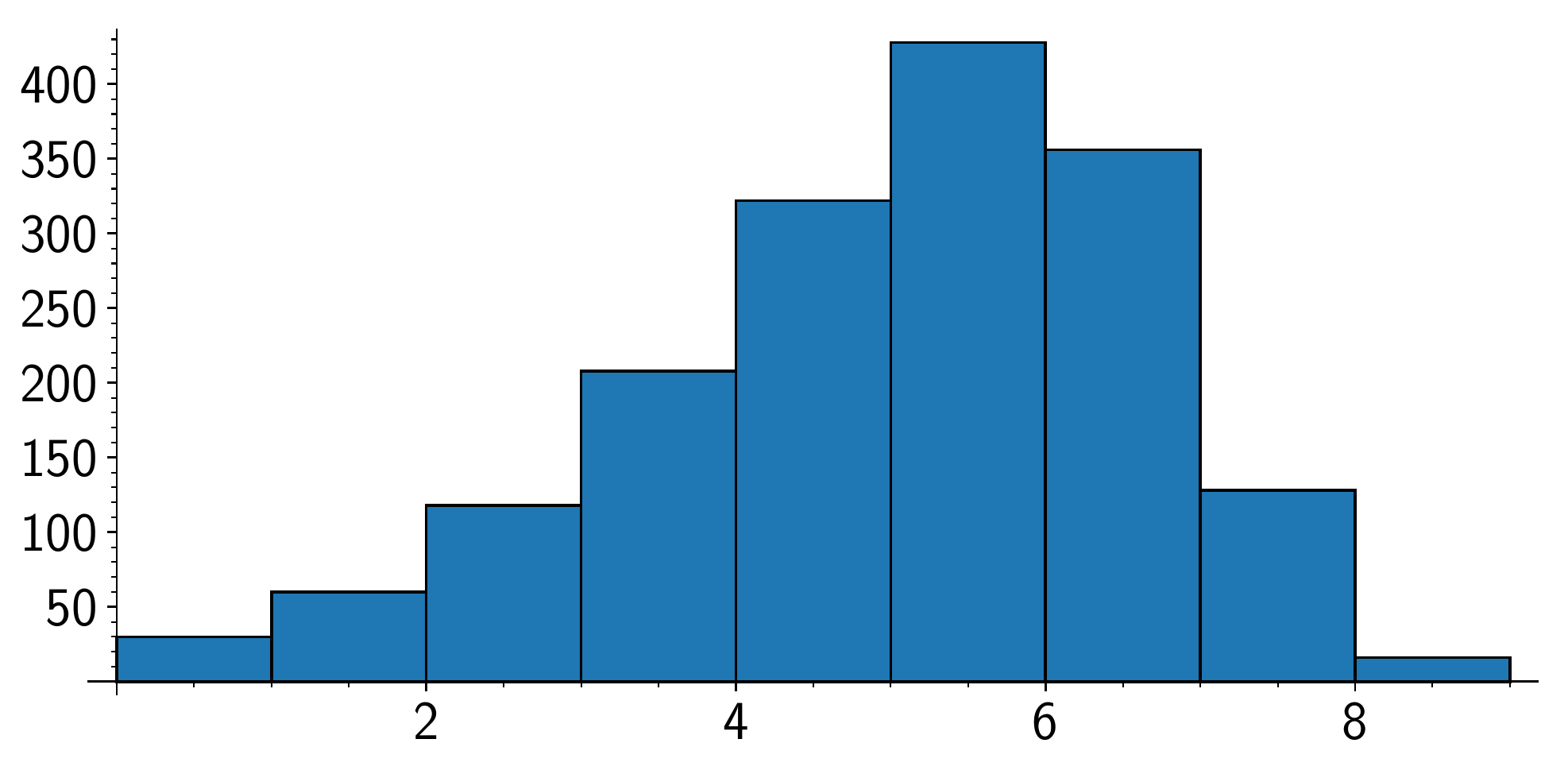}
  \caption{Distances to $\FpSubGraph$ for $p = 19993$.}
\end{subfigure}
\begin{subfigure}{.45\linewidth}
  \includegraphics[width=\textwidth]{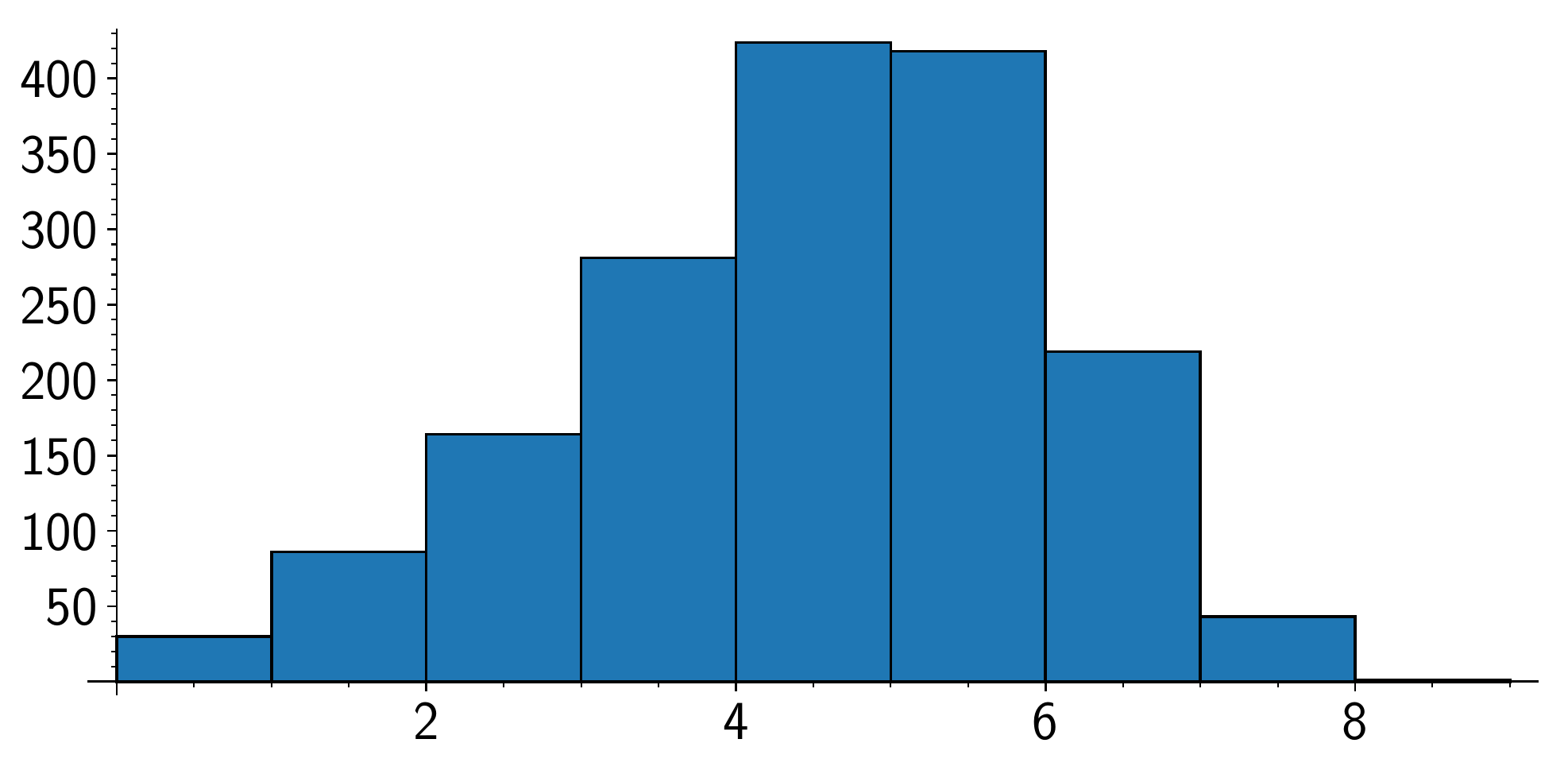}
  \caption{Distances to $R$ for $p = 19993$.}
\end{subfigure}
\caption{Distances to the spine $\FpSubGraph$ compared to distances to a random subgraph of the same size. The subgraph $\FpSubGraph$ is connected  for $p=19991$ and a union of disconnected edges for $p=19993$.}
\label{fig:distances-to-spine-vs-random}
\end{figure}

The significant difference between the two primes shown in Figure~\ref{fig:distances-to-spine-vs-random} can also be explained by the number of vertices in $\FpSubGraph$. Since $\FpBarGraphtwo$ is a $3$-regular graph, for a random vertex $j$, there are at most $3 \cdot 2^{d-1}$ vertices of distance $d$ away from $j$ (and this limit is achieved if there are no collisions on the paths leaving $j$). If $\FpBarGraphtwo$ has $N$ vertices and $H$ is a random subgraph with $M$ vertices, then the expected distance to $H$ from a random vertex should be $\approx \log_2(N/M)$.

For $p=19991$, $|\FpSubGraph| = 199$, so we expect the average distance to $\FpSubGraph$ to be $3.06$. For $p=19993$, $|\FpSubGraph| = 30$, so we expect the average distance to $\FpSubGraph$ to be $5.80$.

\subsubsection{Comparison across primes $p$}
In order to compare the distances to $\FpSubGraph$ across different primes and account for the expected average distance based on the size of $\FpSubGraph$ we consider normalized distances as follows:
\[d_p = (\text{average distance to } \FpSubGraph \text{ for prime }p) / \log_2(|\FpBarGraphtwo|/|\FpSubGraph|)\]
Recall that $\log_2(|\FpBarGraphtwo|/|\FpSubGraph|)$ is the expected distance to the spine from a random vertex.
We observed that the average distances were lower than the expected distance based on the connectedness of $\FpSubGraph$. There are also clear differences in the distributions of $d_p$ when considering residue classes of $p$ modulo $8$. This is shown in Figure \ref{fig:distance_mod_8}. In particular, the data $\mod{8}$ matches our findings on the proportion of opposite pairs, see Figure~\ref{fig:opp_mod8}.

However, the different behaviour of $d_p$ for the different congruence classes $\mod 8$ can be explained by the size of the spine. If the size of the spine $|\FpSubGraph|$ is large, we will need fewer steps to reach the spine from a random vertex $v$. Hence, when counting the paths of length $2^d$ from $v$, we will encounter less backtracking and the estimate is more precise. Looking at Figure \ref{fig:spine_mod_8}, we see that for $p \equiv 7 \mod 8$, the size of the spine is the largest, and for $p \equiv 1 \mod 4$, the size of the spine is the smallest. 

We also tested this within a fixed congruence class:
for primes with $p \equiv 7 \mod 8$ and $15,000< p < 20,000$, the mean distance to the spine is $4.040$ with standard deviation $0.413$ if $|\FpSubGraph| < 100$ and mean $3.007$ with standard deviation $ 0.335$ if $|\FpSubGraph| > 100$.

\begin{figure}[H]
    \begin{subfigure}{.45\textwidth}
    \centering
    \includegraphics[width=\linewidth]{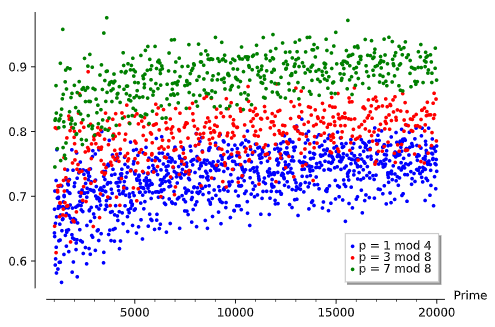}
    \caption{Comparison varying $p \mod{8}$}
    \label{fig:distance_mod_8}
    \end{subfigure}
    \begin{subfigure}{.45\textwidth}
    \centering
    \includegraphics[width=\linewidth]{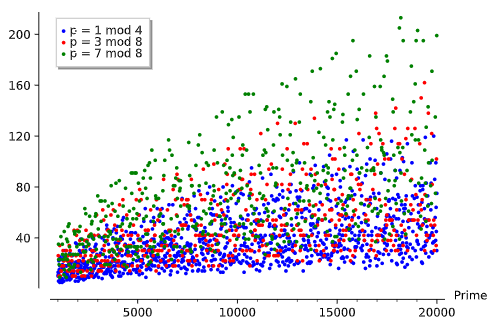}
    \caption{Size of spine $\FpSubGraph$ varying $p \mod 8$}
    \label{fig:spine_mod_8}
    \end{subfigure}
    \caption{Normalized average distances to the $\Fp$ spine versus the size of the spine.}
    \label{fig:dist_across_p}
\end{figure}

\section{When are conjugate $j$-invariants $\ell$-isogenous?}

\subsection{Motivation}
In Section~\ref{sec:shortest-paths-through-spine} we studied paths between conjugate $j$-invariants in $\FpGraphtwo$ that go through the spine $\FpSubGraph$. On the other hand, if $\jpp$ and $\jpp^p$ are $2$-isogenous, then the shortest path between them has length one and does not go through $\FpSubGraph$. This leads us to the natural question:

\begin{center}
\textbf{Question 1:} \textit{How often are conjugate $\mathbb{F}_{p^2}\setminus\mathbb{F}_p$ $j$-invariants $\ell$-isogenous, for $\ell = 2,3$?}
 \end{center}

\subsection{Methods}

For varying primes $p$, we want to collect data on how often conjugate $\mathbb{F}_{p^2}\setminus\mathbb{F}_p$ $j$-invariants are $\ell$-isogenous, for $\ell = 2,3$. We used two main approaches for this, and here we compare their efficiency. The first one corresponds to a modified Breadth-First Search (BFS) algorithm. The second one is based on the supersingular $j$-invariant polynomial and the modular polynomial.

\vspace{0.2in}

\textbf{Breadth-First Search.} Breadth-First Search (BFS) is an algorithm for exploring a graph starting at a fixed vertex.  From the starting vertex, the algorithm explores all of the neighbor nodes at a given depth before moving on to the nodes at the next depth. 

In our experiments we generated the graph $\FpBarGraph$ as follows. First we find a supersingular $j$-invariant $j_0$ over $\Fp$ using the CM method described in \cite{Broker}. It works by finding the smallest prime $q$ such that $p$ is inert in $\QQ(\sqrt{-q})$ and then finding a root of the Hilbert class polynomial for $\QQ(\sqrt{-q})$ over $\Fp$. In practice this step is very efficient for small $p$.

Next we generate the graph using BFS. BFS takes $O(|V| + |E|)$ steps when performed on a graph with $|V|$ vertices and $|E|$ edges. Because $\FpBarGraph$ is an $\ell$-regular graph with $\approx p/12$ vertices, this will run in $O(p)$ steps. Each step requires finding the neighbors of a particular vertex $j$ in the graph. This is done by finding the roots (with multiplicities) of $\Phi_\ell(j,x) \in \Fptwo[x]$. So the total time complexity for this step is $\tilde{O}(p)$.

 We modified this algorithm to collect the $\ell$-isogenous conjugate pairs of $\mathbb{F}_{p^2}\setminus\mathbb{F}_p$ $j$-invariants as we explore the graph, so the time complexity of the algorithm is essentially the same as the complexity for exploring the graph. 

\vspace{0.2in}

\textbf{Supersingular $j$-invariant polynomial with $\Phi_\ell$.}
Another method to calculate the proportion of  $\ell$-isogenous conjugate pairs uses the supersingular $j$-invariant polynomial. Sage has a built-in command to computes this polynomial. It uses the fact that an elliptic curve over $\Fpbar$ given by the Legendre equation $y^2=x(x-1)(x-\lambda)$ is supersingular if and only if $\lambda$ is a root of the polynomial 
\begin{align}\label{eq:Hpoly}
    H(t)=\sum_{i=0}^{m}{{m}\choose{i}}^2t^i
\end{align}
for $m=\frac{p-1}{2}$ (see \cite[Section V.4]{AEC}). 
If $j$ is the $j$-invariant of an elliptic curve as above, then the polynomial $$F(s,t)=st^2(t-1)^2-2^8(t^2-t+1)^3$$ vanishes at $(j,\lambda)$. Sage then computes the resultant $R(s)$ of $H(t)$ and $F(s,t)$ with respect to $t$, factors it over $\Fp$ and defines the supersingular $j$-invariant polynomial as the product of these factors counted only once. If $(s-1728)$ or $s$ are factors, they are excluded. 

Once we obtain this polynomial, we proceed as follows:
\begin{enumerate}
    \item Compute the roots of the supersingular $j$-invariant polynomial over $\mathbb{F}_{p^2}$.
    \item Sort the roots over $\mathbb{F}_{p^2}\setminus\mathbb{F}_p$ by conjugate pairs and count the number of such pairs.
    \item Use the $\ell^{th}$ modular polynomial $\Phi_{\ell}$ to determine which conjugate pairs are $\ell$-isogenous.  
\end{enumerate}
For a prime $p$, the number of supersingular $j$-invariants over $\mathbb{F}_{p^2}$ is $\lfloor\frac{p}{12}\rfloor + \varepsilon$ for $\varepsilon \in \{0,1,2\}$ \cite[Thm~V.4.1]{AEC} and the supersingular $j$-invariant polynomial is thus a polynomial of degree $\lfloor\frac{p}{12}\rfloor$.

\subsubsection{Timing data}

 We expected the BSF algorithm to be faster than the supersingular $j$-polynomial one, since the latter must factoring a polynomial of degree $(p-1)/2$. To experimentally verify the difference in running time, we used both algorithms to find the $2$-isogenous conjugate pairs for 37 primes between 103 and 95471. The resulting data is displayed in Figure \ref{fig:timing}. 
 \begin{figure}[ht!]
    \centering
    \includegraphics[scale=0.65]{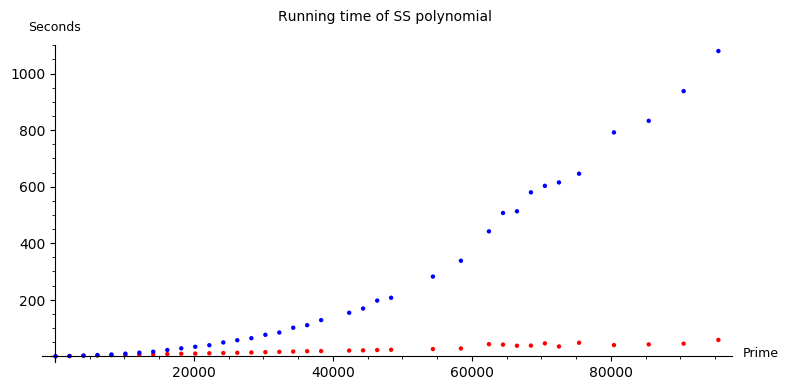}
    \caption{Timing data for the BFS (red) and Supersingular polynomial algorithms (blue), showing the time it took each algorithm to find all $2$-isogenous conjugate pairs for that prime.}
    \label{fig:timing}
\end{figure}

\subsection{Experimental data: $2$-isogenies}
\label{Section:ExpData2IsogenousConjs}

We collected data on supersingular $j$-invariants over $\mathbb{F}_{p^2}$  for all primes $5\leq p\leq 100193$ (a total of 9,605 primes). For each $p$, we collected all of the $\Fptwo\setminus\Fp$ $j$-invariants and counted those that are also $2$-isogenous. 
The plot shown in Figure \ref{fig:2IsogAll} shows the proportion of conjugate pairs that are $2$-isogenous.

\begin{figure}[h!]
    \centering
    \includegraphics[scale=0.65]{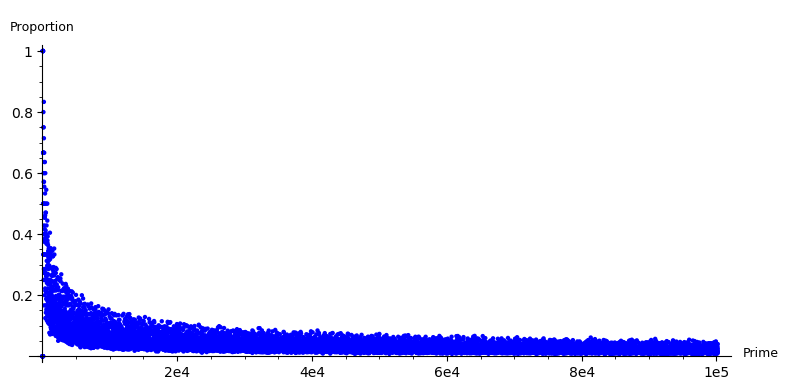}
    \caption{Proportion of $2$-Isogenous Pairs of Conjugate $j$-Invariants in $\FpBarGraph$. Points are of the form $(p,y)$ where $p$ is a prime and $y$ is the proportion of conjugate pairs of $j$-invariants which are $2$-isogeneous.}
    \label{fig:2IsogAll}
\end{figure}

With a few exceptions, all of the proportions computed are positive and strictly less than 1. The small primes (roughly $p<5000$) have a wide range of proportions, between $0$ and $1$. This is expected due to the small number of points on their $\FpBarGraph$ graphs. For example: there are some primes $p$ such that all of the pairs of conjugates are $2$-isogenous. On the other hand, if $\Fptwo\setminus\Fp = \emptyset$,  which can happen for small primes, then the proportion will be trivially zero. Notably, the only examples of $p$ for which the proportion is zero are $p = 101,131$. 

To avoid small prime phenomena, we focused on analyzing the data we collected for $10007\leq p \leq 100193$ (a total of 8378 primes). When referring to this data, we will use the phrase ``main data''. When referring to all of the data collected for $5\leq p \leq 100193$, we use the phrase ``all data''. 

The graph of proportions for the main data can be found in Figure \ref{fig:2IsogMain}.
\begin{figure}[h!]
    \centering
    \includegraphics[scale=0.65]{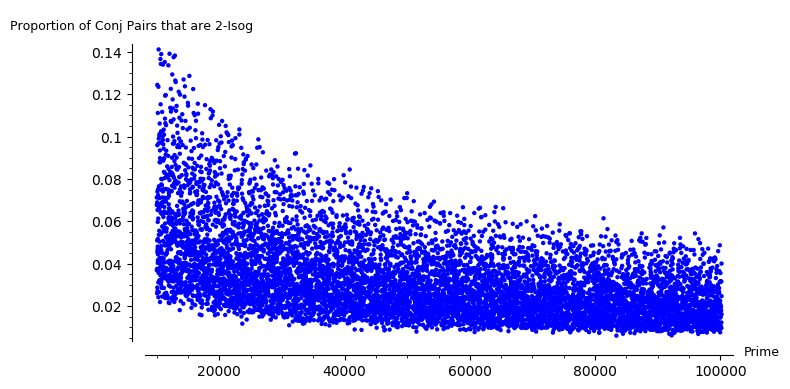}
    \caption{Proportion of 2-isogenous conjugate pairs in $\FpBarGraphtwo$ for $p>10000$}
    \label{fig:2IsogMain}
\end{figure}

In this collection of data, for primes $10007\leq p \leq 100193$, we found there to be a mean proportion of $0.032780$ with standard deviation of $0.019134$.

We then sorted the data by congruence conditions to look for patterns. The biggest difference appeared when we re-sorted the data according to the congruence class of the primes modulo $12$.

\subsubsection{Primes Modulo $12$} \label{sec:conjugate_pairs_mod_12}
In Table \ref{Tab:2IsogMod12}, we summarize the differences between the different congruence classes modulo 12. Note the similar, higher means for $p\equiv 1,7\mod{12}$ and the similar, lower means for $p\equiv 5,11\mod{12}$.

\begin{table}[h!]
\begin{center}
\begin{tabular}{|l|l|l|}
\hline
                    & $p \equiv 1\pmod{12}$ & $p \equiv 5\pmod{12}$  \\ \hline
Total \# of primes: & 2079                  & 2104                   \\ \hline
Mean:               & 0.043551    & 0.021969     \\ \hline
Standard Deviation: & 0.019815    & 0.010206     \\ \hline
                    &                       &                        \\ \hline
                    & $p \equiv 7\pmod{12}$ & $p \equiv 11\pmod{12}$ \\ \hline
Total \# of primes: & 2101                  & 2094                   \\ \hline
Mean:               & 0.043375    & 0.022244     \\ \hline
Standard Deviation: & 0.020140    & 0.010512     \\ \hline
\end{tabular}
\caption{Proportions of 2-isogenous conjugates, $10007\leq p \leq 100193$, sorted by $p\mod{12}$}
\label{Tab:2IsogMod12}
\end{center}
\end{table}

These distributions are skewed according to the congruence class, as we can also see from the graph in Figure \ref{fig:2IsogMainMod12}.

\begin{figure}[h!]
    \centering
    \includegraphics[scale=0.65]{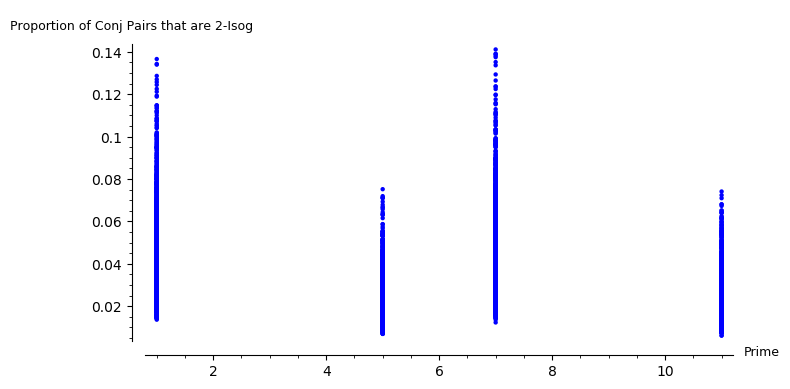}
    \caption{Proportions of 2-isogenous conjugates, $10007\leq p \leq 100193$, sorted by $p\mod{12}$}
    \label{fig:2IsogMainMod12}
\end{figure}

There appears to be a correlation between primes $p\equiv 1,7\mod{12}$ and between primes $p\equiv 5,11\mod{12}$. A two-sample $t$-test confirms these correlations at the $99.8\%$ level.

\subsection{Experimental data: $3$-isogenies}
We collected data on the supersingular $j$-invariants over $\mathbb{F}_{p^2}$  for all the primes $5\leq p\leq 100193$ (a total of 9,605 primes) and computed the proportion of conjugate pairs that are also $3$-isogenous. We present this data in the same format as the $2$-isogeny data presented in \ref{Section:ExpData2IsogenousConjs}.

In Figure \ref{fig:3IsogAll}, observe the proportions of conjugate pairs of primes for all of the primes we collected data on. 
\begin{figure}
    \centering
    \includegraphics[scale=0.65]{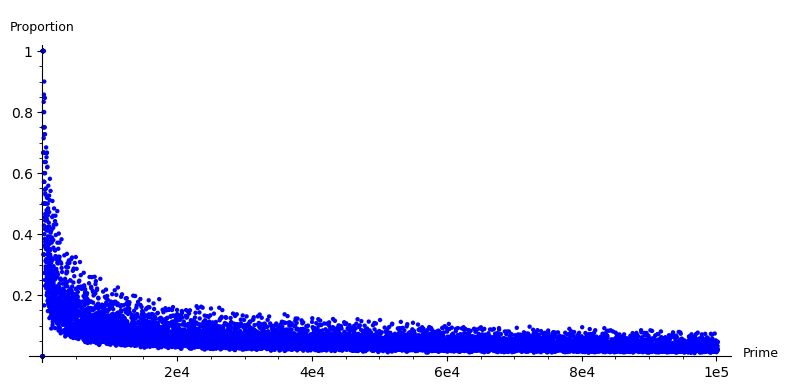}
    \caption{Proportion of $3$-Isogenous Pairs of Conjugate $j$-Invariants in $\FpBarGraph$}
    \label{fig:3IsogAll}
\end{figure}

Again, we observe some small prime phenomena (proportions of 1 and 0 for $p$ small). However, in the $3$-isogeny case we do not have nontrivial examples of primes $p$ for which the proportion of $3$-isogenous conjugates is 0: if there exist conjugate $j$-invariants in $\mathbb{F}_{p^2}\setminus\Fp$, then there is at least one pair of $3$-isogenous conjugates. (Recall that the two counterexamples to this statement in the $2$-isogeny case were $p=101,131$.)

To avoid small prime phenomena, we again focused on analyzing the data we collected for $10007\leq p \leq 100193$ (a total of 8378 primes). Again, when referring to this data, we will use the phrase ``main data''. When referring to all of the data collected for $5\leq p \leq 100193$, we use the phrase ``all data''. 

The graph of proportions for the main data can be found in Figure \ref{fig:3IsogMain}.
\begin{figure}
    \centering
    \includegraphics[scale=0.65]{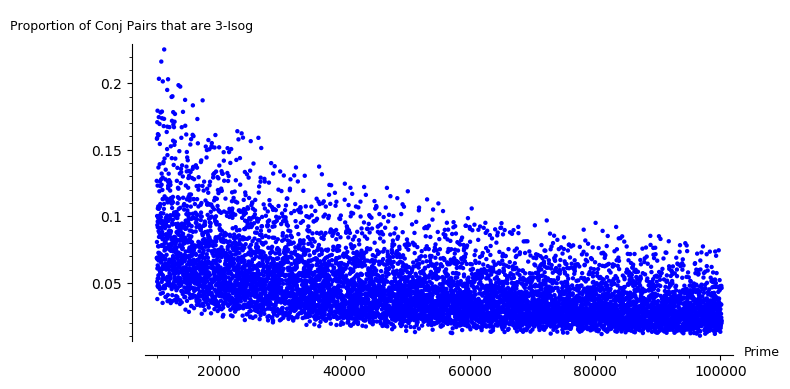}
    \caption{Proportion of 3-isogenous conjugate pairs in $\FpBarGraph$ for primes $p>10000$}
    \label{fig:3IsogMain}
\end{figure}
In this collection of data, for primes $10007\leq p \leq 100193$, we found there to be a mean proportion of $0.047306$ with standard deviation of $0.026568$.

As in the $2$-isogeny case, we again sorted the data by congruence conditions to look for patterns. The biggest difference appeared when we re-sorted the data according to the congruence class of the primes modulo $12$.

\subsubsection{Primes Modulo $12$}

In Table \ref{Tab:3IsogMod12}, we summarize the differences between the different congruence classes modulo 12. Note the similar and higher means for $p\equiv 1,5\mod{12}$ and the similar and lower means for $p\equiv 7,11\mod{12}$.

\begin{table}[h]
\centering
\begin{tabular}{|l|l|l|}
\hline
                    & $p \equiv 1\pmod{12}$ & $p \equiv 5\pmod{12}$  \\ \hline
Total \# of primes: & 2079                  & 2104                   \\ \hline
Mean:               & 0.058526    & 0.059034     \\ \hline
Standard Deviation: & 0.029488    & 0.029729     \\ \hline
                    &                       &                        \\ \hline
                    & $p \equiv 7\pmod{12}$ & $p \equiv 11\pmod{12}$ \\ \hline
Total \# of primes: & 2101                  & 2094                   \\ \hline
Mean:               & 0.035620    & 0.036107     \\ \hline
Standard Deviation: & 0.016369    & 0.016706     \\ \hline
\end{tabular}
\caption{Proportions of $3$-isogenous conjugates for $10007\leq p \leq 100193$, sorted by $p\mod{12}$}
\label{Tab:3IsogMod12}
\end{table}

These distributions are skewed according to the congruence class, as we can also see from the graph in Figure \ref{fig:3IsogMainMod12}.

\begin{figure}[h]
    \centering
    \includegraphics[scale=0.65]{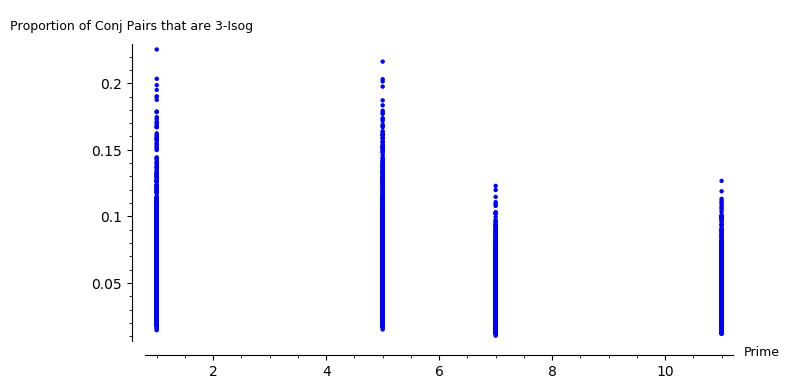}
    \caption{Proportions of 3-isogenous conjugates for $10007\leq p \leq 100193$, sorted by $p\mod{12}$}
    \label{fig:3IsogMainMod12}
\end{figure}

There appears to be a correlation between primes $p\equiv 1,5\mod{12}$ and between primes $p\equiv 7,11\mod{12}$. A two-sample $t$-test confirms these correlations at the $99.8\%$ level.

\subsection{Analysis of data}

Our experimental data suggests that, at least for $\ell=2,3$ and with the exception of a few small primes, the proportion of conjugate pairs that are $\ell$-isogenous is a small positive number. In particular, all of the primes $p\neq 101,131$ with supersingular $j$-invariants in $\mathbb{F}_{p^2}\setminus\mathbb{F}_p$ observed have at least one such pair. This motivates the following two questions:

\textbf{Question 2:}  \textit{For $p>131$, is there always at least one pair of $\ell$-isogenous conjugate $j$-invariants on $\FpBarGraph$?}\\
\indent \textbf{Question 3:} \textit{For large p, is there a nontrivial lower and/or upper bound for the proportion of $\ell$-isogenous conjugate $j$-invariants on $\FpBarGraph$?}\\

There is a significant difference on the average of the proportion $\ell$-isogenous conjugate pairs when we look at the congruence class of modulo 12. We see that this number tends to be smaller when $p\equiv 5,11\pmod {12}$ than when $p\equiv 1,7\pmod {12}$.  

\textbf{Question 4:} \textit{How does the proportion of $\ell$-isogenous conjugate $j$-invariants on $\FpBarGraph$ relate to the conjugacy class of $p$ $\mod$ 12?}

\section{Diameter}

Numerical experiments in \cite{Sardari} estimated the diameters of $k$-regular LPS Ramanujan graphs and random Cayley graphs to be asymptotically $(4/3)\log_{k-1}n$ and $\log_{k-1}n$ respectively, where $n$ is the number of vertices. 
In this section, we present data on the diameters of the supersingular $2$-isogeny graphs, which are $3$-regular on approximately $p/12$ vertices (precisely $\lfloor p/12\rfloor + 0,1$ or $2$ vertices, depending on $p$). 

We can see a lower bound 
$$\log_2 \left(\lfloor \frac{p}{12} \rfloor\right) - \log_2(3) + 1 $$
on the diameter as follows.
Starting from a random vertex and taking a walk of length $n$, the walk reaches at most $3 \cdot 2^{n-1}$ vertices as endpoints (exactly that number if there are no collisions).  Since there are 
$\lfloor \frac{p}{12} \rfloor + \epsilon$ vertices in the graph, with 
$\epsilon = 0,1,2$,
the diameter cannot be less that the smallest $n_0$ such that 
$$3 \cdot 2^{n_0-1} \geq  \lfloor \frac{p}{12} \rfloor.$$
This lower bound is shown in green in Figure~\ref{fig:DiametersAll} below.

Our numerical data suggests the diameter of the supersingular $2$-isogeny graph do \textit{not} grow like $(4/3)\log_{2}(p/12)$, contrary to the behaviour of LPS graphs. This can be seen from the blue line in Figure \ref{fig:DiametersAll}, which has been shifted vertically to fit the data as well as possible, but has too large a slope to match the shape of the distribution. We found $O(\log_{2}(p/12))$ (the red line in Figure \ref{fig:DiametersAll}) to be a better fit, suggesting the $2$-isogeny graph might behave more like random Cayley graphs.

We collected graph data for the diameters of the supersingular $2$-isogeny graphs $\mathcal{G}_2(\overline{\mathbb{F}}_p)$ for 3387 primes $p$. We used the built-in Sage "diameter" function (\cite{sage}) on graphs.  The implementation can be found in the walk.sage worksheet available on our github repository. We collected the data in batches, taking snapshots of the possible diameters of $\FpBarGraphtwo$ for ranges of primes. The smallest prime we have data for is $p = 1009$ and the largest is $p = 4010173$. This data is displayed in figure \ref{fig:DiametersAll}.

\begin{figure}[h!]
    \centering
    \includegraphics[scale=0.45]{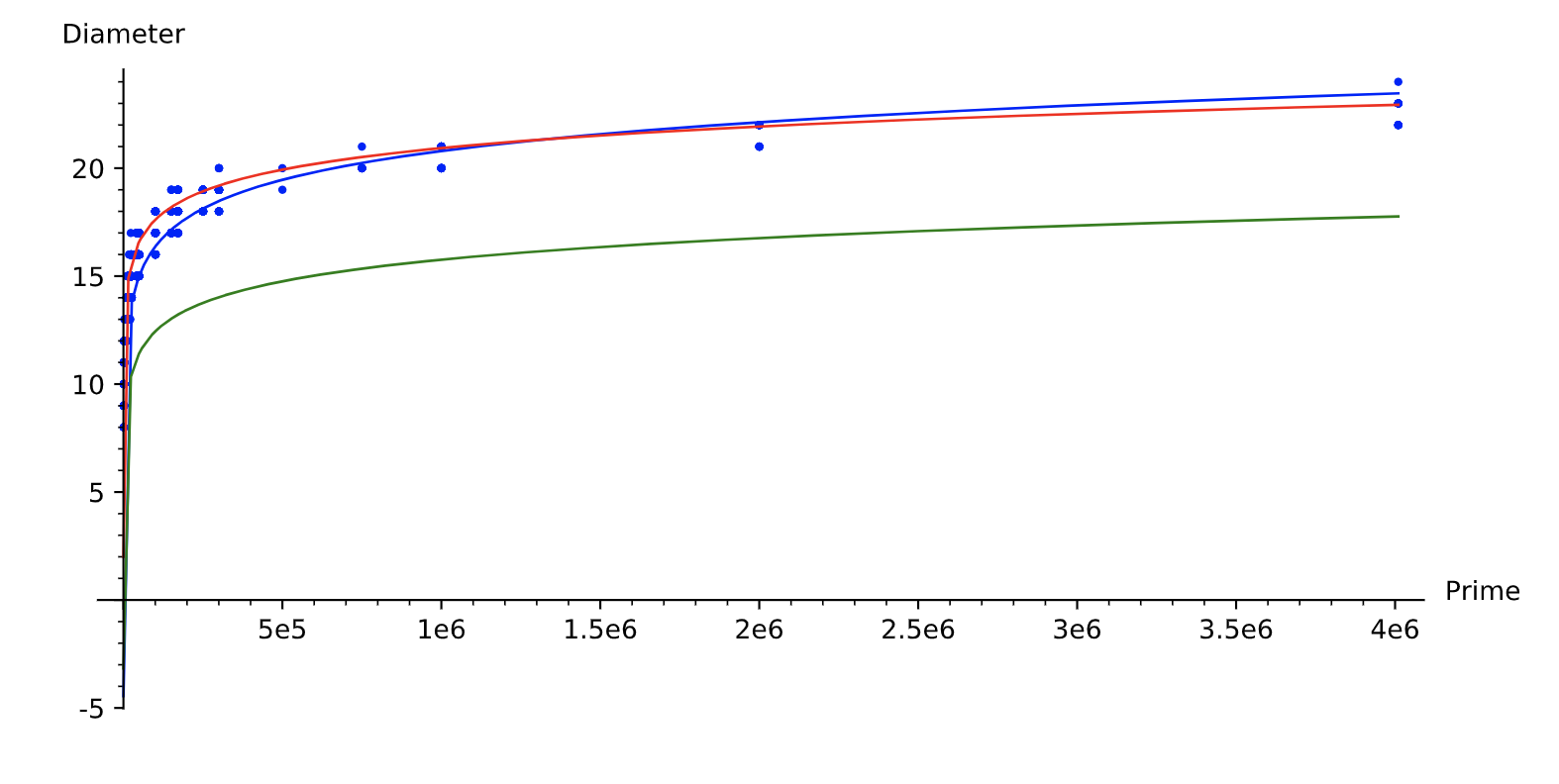}
    \caption{Diameters of $2$-isogeny graph over $\overline{\mathbb{F}}_p$, with $y = \log_2(p/12) + \log_2(12) + 1$ (red) and $y = \frac{4}{3}\log_2(p/12)-1$ (blue).}
    \label{fig:DiametersAll}
\end{figure}

\begin{figure}[h!]
    \centering
    \includegraphics[scale=0.55]{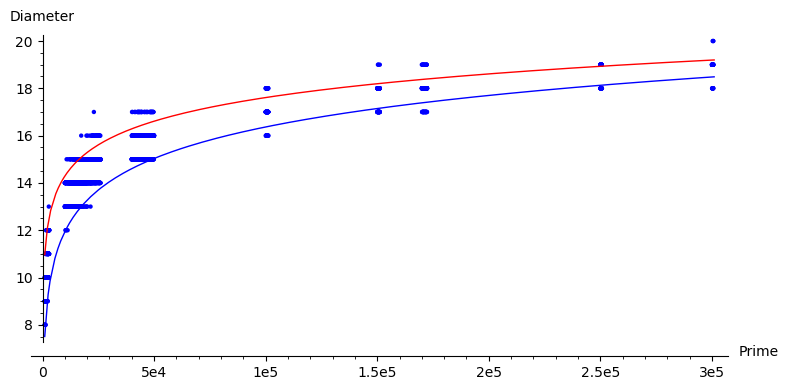}
    \caption{Cropped and enlarged graph of Figure~\ref{fig:DiametersAll}, for the data collected on 3313 primes $p$ with $1009 \leq p \leq 300361$.}
    \label{fig:DiametersZoomedIn}
\end{figure}

\subsection{Diameters of Primes Modulo 12} \label{subsec:diameters_mod_12}
Recall that the number of spinal components and $2$-isogenous conjugate pairs is dependent on the congruence class of $p$ modulo $8$. Motivated by this, we investigated the behaviour of the diameter as $p$ varies modulo $8$. We found a slight, but noticeable, bias for primes congruent to $5$ and $11$ modulo $12$ to have a $2$-isogeny graph of larger diameter compared with primes congruent to $1$ or $7$ modulo $12$. 

This is visible in Figures \ref{fig:Diameters1And7} and \ref{fig:Diameters5And11}. Notice that in Figure \ref{fig:Diameters5And11}, the scatter plot points tend to be slightly higher than the graph of $y = \log_2(p/12)+\log_2(12)+1$, whereas those in Figure~\ref{fig:Diameters1And7} tend to be more evenly distributed above and below $y = \log_2(p/12)+\log_2(12)+1$. 

\begin{figure}[h!]
    \centering
    \includegraphics[scale=0.55]{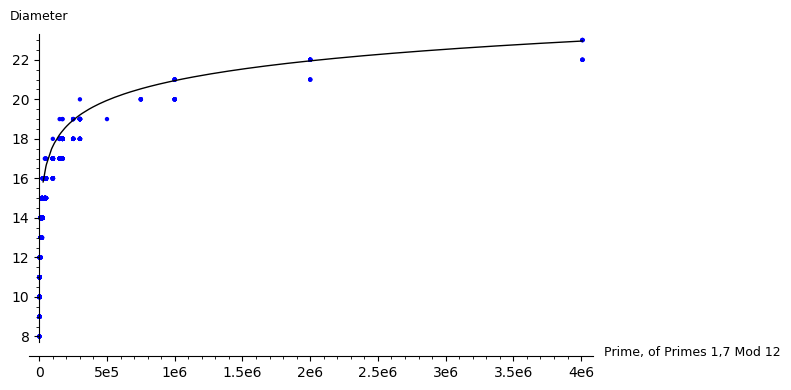}
    \caption{Diameters of $\FpBarGraphtwo$ for $p\equiv 1,7\mod{12}$, with $y =\log_2(p/12) + +\log_2(12)+1$}
    \label{fig:Diameters1And7}
\end{figure}

\begin{figure}[H]
    \centering
    \includegraphics[scale=0.55]{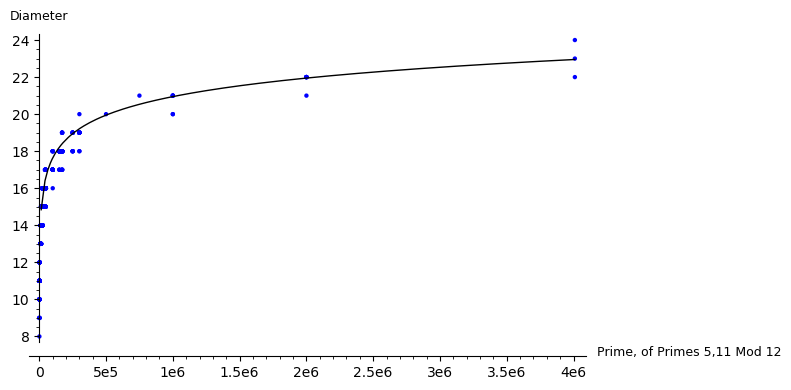}
    \caption{Diameters of $2$-isogeny graph over $\overline{\mathbb{F}}_p$, for $p\equiv 5,11\mod{12}$, with $y = \log_2(p/12)++\log_2(12)+1$}
    \label{fig:Diameters5And11}
\end{figure}

Table \ref{tab:diameter_difference_mod_12} confirms the visible bias.
\begin{table}[htbp]
    \centering
    \begin{tabular}{c|c|c|c}
 \multicolumn{4}{c}{ average diameter for $100,000 < p < 300,000$ } \\ 
 \hline $ 1 \mod 12 $ & $17.2190476190476 $  & $ 5 \mod 12 $ & $17.8761061946903$ \\
 $ 7 \mod 12 $ & $17.7346938775510$ & 
 $ 11 \mod 12 $ & $17.9919354838710$  \\
 \hline  \multicolumn{4}{c}{ average diameter for $300,000 < p < 500,000$  } \\  
 \hline    $ 1 \mod 12 $ & $18.4000000000000 $ & 
 $ 5 \mod 12 $ & $18.9230769230769$  \\ 
 $ 7 \mod 12 $ & $18.8235294117647$ &
 $ 11 \mod 12 $ & $19.1000000000000$ \\
    \end{tabular}
    \caption{Average diameters sorted by primes modulo $12$. The first data set contains around $100$ primes in each congruence class, the latter between $10 $ to $17$ primes.}
    \label{tab:diameter_difference_mod_12}
\end{table}

\section{Conclusions}

We determined how the connected components of $\FpGraph$ merge together to give the spine $\FpSubGraph \subset \FpBarGraph$. For any specific $\ell$ and any $p$, one can determine the resulting shape explicitly if one knows the structure of the class group $\Cl(\OO_K)$. 

For $\ell =2$, we gave heuristics on the distances of the connected components of $\FpSubGraph$, paths that pass through the spine, the proportion of conjugate pairs, and the diameters of graphs $\FpBarGraphtwo$.
We saw differences between the congruence classes modulo $12$. In summary, the data suggests the following, although more careful analysis is needed to confirm:
\begin{itemize}
    \item $p\equiv 1,7\mod{12}$: 
    \begin{itemize}
        \item smaller $2$-isogeny graph diameters
        \item larger number of spinal components
        \item larger proportion of $2$-isogenous conjugate pairs
    \end{itemize}
    
    \item $p\equiv 5,11\mod{12}$: 
    \begin{itemize}
        \item larger $2$-isogeny graph diameters
        \item smaller number of spinal components
        \item smaller proportion of $2$-isogenous conjugate pairs
    \end{itemize}
\end{itemize}

\bibliographystyle{alpha}
\bibliography{biblio}\vspace{0.75in}

\newcommand{\etalchar}[1]{$^{#1}$}
\begin{thebibliography}{CLM{\etalchar{+}}18}

\bibitem[BJS14]{BJS}
Jean-François Biasse, David Jao, and Anirudh Sankar.
\newblock {\em {A quantum algorithm for computing isogenies between
  supersingular elliptic curves}}.
\newblock Springer, 2014.

\bibitem[Br{\"o}09]{Broker}
Reinier Br{\"o}ker.
\newblock {Constructing Supersingular Elliptic Curves}.
\newblock {\em Journal of Combinatorics and Number Theory 1}, 1:269--273, 2009.

\bibitem[CFL{\etalchar{+}}18]{WIN4}
Anamaria Costache, Brooke Feigon, Kristin Lauter, Maike Massierer, and Anna
  Puskas.
\newblock Ramanujan graphs in cryptography.
\newblock Cryptology ePrint Archive, Report 2018/593, 2018.
\newblock \url{https://eprint.iacr.org/2018/593}.

\bibitem[CGL06]{CGL06}
Denis Charles, Eyal Goren, and Kristin Lauter.
\newblock Cryptographic hash functions from expander graphs.
\newblock Cryptology ePrint Archive, Report 2006/021, 2006.
\newblock \url{https://eprint.iacr.org/2006/021}.

\bibitem[CLM{\etalchar{+}}18]{CSIDH}
Wouter Castryck, Tanja Lange, Chloe Martindale, Lorenz Panny, and Joost Renes.
\newblock {CSIDH}: An efficient post-quantum commutative group action.
\newblock Cryptology ePrint Archive, Report 2018/383, 2018.
\newblock \url{https://eprint.iacr.org/2018/383}.

\bibitem[{Cox}89]{Cox}
David {Cox}.
\newblock {\em {Primes of the form $x^2 + ny^2$}}.
\newblock John Wiley and Sons, Inc., New York, 1989.

\bibitem[DG16]{DelGal01}
C.~{Delfs} and S.~D. {Galbraith}.
\newblock {Computing isogenies between supersingular elliptic curves over
  $\mathbb{F}_p$}.
\newblock {\em Des. Codes Cryptography}, 78(2):425–440, 2016.
\newblock \url{https://arxiv.org/pdf/1310.7789.pdf}.

\bibitem[EHL{\etalchar{+}}18]{EHLMP}
Kirsten Eisentraeger, Sean Hallgren, Kristin Lauter, Travis Morrison, and
  Christophe Petit.
\newblock Supersingular isogeny graphs and endomorphism rings: reductions and
  solutions.
\newblock Cryptology ePrint Archive, Report 2018/371, 2018.
\newblock \url{https://eprint.iacr.org/2018/371}.

\bibitem[FJP11]{DFJP11}
Luca~De Feo, David Jao, and Jérôme Plût.
\newblock Towards quantum-resistant cryptosystems from supersingular elliptic
  curve isogenies.
\newblock Cryptology ePrint Archive, Report 2011/506, 2011.
\newblock \url{https://eprint.iacr.org/2011/506}.

\bibitem[GPST16]{GPSTOnTheSecurity}
Steven~D. Galbraith, Christophe Petit, Barak Shani, and Yan~Bo Ti.
\newblock On the security of supersingular isogeny cryptosystems.
\newblock Cryptology ePrint Archive, Report 2016/859, 2016.
\newblock \url{https://eprint.iacr.org/2016/859}.

\bibitem[Ibu82]{IBUKIYAMA}
Tomoyoshi Ibukiyama.
\newblock On maximal orders of division quaternion algebras over the rational
  number field with certain optimal embeddings.
\newblock {\em Nagoya Math. J.}, 88:181--195, 1982.

\bibitem[Igu58]{IGUSA}
Jun-Ichi Igusa.
\newblock Class number of a definite quaternion with prime discriminant.
\newblock {\em Proceedings of the National Academy of Sciences of the United
  States of America}, 44(4):312--314, 1958.

\bibitem[Kan89]{KANEKO}
Masanobu Kaneko.
\newblock Supersingular j-invariants as singular moduli mod p.
\newblock {\em OSAKA JOURNAL OF MATHEMATICS}, 26, 12 1989.

\bibitem[KLPT14]{KLPT}
David Kohel, Kristin Lauter, Christophe Petit, and Jean-Pierre Tignol.
\newblock On the quaternion {$\ell$}-isogeny path problem.
\newblock {\em LMS J. Comput. Math.}, 17(suppl. A):418--432, 2014.

\bibitem[Koh96]{Kohel}
David Kohel.
\newblock {\em Endomorphism rings of elliptic curves over finite fields}.
\newblock PhD thesis, University of California, Berkely, 1996.

\bibitem[LPS88]{LPS}
A.~Lubotzky, R.~Phillips, and P.~Sarnak.
\newblock Ramanujan graphs.
\newblock {\em Combinatorica}, 8(3):261--277, Sep 1988.

\bibitem[{Sar}19]{Sardari}
Naser~T. {Sardari}.
\newblock {Diameter of Ramanujan Graphs and Random Cayley Graphs}.
\newblock {\em Combinatorica}, 39:427–446, 2019.
\newblock \url{https://doi.org/10.1007/s00493-017-3605-0}.

\bibitem[Sho99]{shor1999polynomial}
Peter~W. Shor.
\newblock Polynomial-time algorithms for prime factorization and discrete
  logarithms on a quantum computer.
\newblock {\em SIAM Rev.}, 41(2):303--332, 1999.

\bibitem[{Sil}09]{AEC}
Joseph~H. {Silverman}.
\newblock {\em {The Arithmetic of Elliptic Curves, 2nd Edition}}.
\newblock Springer-Verlag, New York, N.Y., 2009.

\bibitem[Sut13]{Sut13}
Andrew~V. Sutherland.
\newblock Isogeny volcanoes.
\newblock In {\em A{NTS} {X}---{P}roceedings of the {T}enth {A}lgorithmic
  {N}umber {T}heory {S}ymposium}, volume~1 of {\em Open Book Ser.}, pages
  507--530. Math. Sci. Publ., Berkeley, CA, 2013.

\bibitem[{The}19]{sage}
{The Sage Developers}.
\newblock {\em {S}ageMath, the {S}age {M}athematics {S}oftware {S}ystem
  ({V}ersion 8.7)}, 2019.
\newblock {\tt https://www.sagemath.org}.

\end{thebibliography}

\end{document}